\documentclass[reqno,12pt]{amsart}
\usepackage[pagebackref]{hyperref}
\usepackage[a4paper,margin=2.5cm]{geometry}
\usepackage{amsmath, amscd}
\usepackage{amssymb}
\usepackage{xfrac}
\usepackage{tikz}

\newcommand{\R}{\mathbb{R}}

\def\vp{\varphi}

\newcommand{\II}{\mathbf{1}}

\newcommand{\scalar}[1]{\langle {#1} \rangle}

\newcommand{\biggscalar}[1]{\bigg\langle {#1} \bigg\rangle}

\newcommand{\Id}{\textup{Id}}
\newcommand{\id}{\textup{id}}
\newcommand{\tr}{\textup{tr}}
\newcommand*{\ind}[1]{\mathbf{1}_{\left\{#1\right\}}}
\newcommand*{\Ind}[1]{\mathbf{1}_{#1}}

\newcommand{\supp}{\textup{supp}}

\newcommand{\I}{\textup{I}}
\newcommand{\dist}{\textup{dist}}
\newcommand{\diag}{\textup{diag}}

\DeclareMathOperator{\rad}{rad}
\DeclareMathOperator{\re}{Re}
\DeclareMathOperator{\im}{Im}
\DeclareMathOperator{\Hess}{Hess}
\DeclareMathOperator{\dom}{dom}
\DeclareMathOperator{\interior}{int}
\DeclareMathOperator{\cl}{cl}
\DeclareMathOperator{\bd}{bd}

\newtheorem{theorem}{Theorem}
\newtheorem{lemma}[theorem]{Lemma}
\newtheorem{fact}[theorem]{Fact}
\newtheorem{prop}[theorem]{Proposition}
\newtheorem{corollary}[theorem]{Corollary}

\theoremstyle{definition}
\newtheorem{dfn}[theorem]{Definition}
\newtheorem{remark}[theorem]{Remark}
\newtheorem{example}[theorem]{Example}

\numberwithin{theorem}{section}
\numberwithin{equation}{section}

\begin{document}

\title{Positive Gaussian kernels also have Gaussian minimizers}

\author{Franck Barthe}
\address[F.~Barthe and P.~Wolff]{Institut de Math{\'e}matiques de Toulouse; UMR 5219, Universit{\'e} de Toulouse; CNRS, France}
\email[F.~Barthe]{franck.barthe@math.univ-toulouse.fr}
 \thanks{This paper is partially based upon work supported by the National Science Foundation under Grant No. DMS-1440140 while the first named author was in residence at the Mathematical Sciences Research Institute in Berkeley, California, during the Fall 2017 semester. }

\author{Pawe{\l} Wolff}
\email[P.~Wolff]{pwolff@mimuw.edu.pl}
\thanks{Research of the second named author was partially supported by ANR-11-LABX-0040-CIMI within the program ANR-11-IDEX-0002-02
and by the National Science Center, Poland project number 2015/18/A/ST1/00553.}

\begin{abstract}
We study lower bounds on multilinear operators with Gaussian kernels acting on Lebesgue spaces, with exponents below one.  We put forward natural conditions when the optimal constant can be computed by inspecting centered Gaussian functions only, and we give necessary and sufficient conditions for this constant to be positive. Our work provides a counterpart to Lieb's results on maximizers of multilinear operators with real Gaussian kernels, also known as the multidimensional Brascamp-Lieb inequality. It unifies and extends several inverse inequalities.
\end{abstract}

\maketitle

\section{Introduction}

\subsection{Background and motivation}
Our title hints at Lieb's article ``Gaussian kernels have only Gaussian maximizers'' \cite{lieb-1990}. This remarkable work studies operators with Gaussian kernels between Lebesgue spaces $L^p(\R^n,\mathbb C)$, with norm $\|f\|_p=\left(\int_{\R^n}
|f(y)|^p \, dy \right)^{1/p}$. These operators are of the following form: for any integrable $f \colon \R^n\to\mathbb C$,
$$ Gf(x)=\int_{\R^n} e^{-\mathcal Q(x,y)} f(y)\, dy, \qquad x\in\R^m,$$
where $\mathcal Q \colon \R^m\times \R^n\to\mathbb C$ is such that $\re(\mathcal Q) \colon \R^m\times \R^n\to\R$
is a semi-definite positive quadratic form and  $\im(\mathcal Q)\colon\R^m\times \R^n\to\R$ is a quadratic form. Lieb has given conditions which ensure that the operator  norm of $G \colon L^p(\R^n,\mathbb C)\to
L^q(\R^m,\mathbb C)$ can be computed by inspecting centered Gaussian functions only, i.e. functions of the form $f(y)=e^{-q(y)}$ where $\re(q)$ is $q$ positive definite quadratic form and $\im(q)$ is a quadratic form (when $q$ is real-valued, $f$ is a real-valued Gaussian function).
Here is a simplified version of his result:

\begin{theorem}[\cite{lieb-1990}]
	With the above notation, the relationship
	$$ \|G\|_{L_p\to L_q}=\sup\left\{\frac{\|Gf\|_q}{\|f\|_p};\; g \; \mathrm{centered}\; \mathrm{Gaussian} \right\} $$
	holds in any of the following cases
	\begin{itemize}
		\item $1<p\le q<+\infty$,
		\item $1<p, q<+\infty$ and the Gaussian kernel is real (i.e. $\mathcal{Q}$ is a quadratic form). In this case it is enough to consider real-valued Gaussian functions.
	\end{itemize}
\end{theorem}
An important step of the proof consists in the study of the non-degenerate case when $\re(\mathcal Q)$ is definite positive: the operator is shown to be compact, weak topology arguments yield the existence of maximizers of the ratio $\|Gf\|_q/\|f\|_p$, and a careful study of the product 
operator $G\otimes G$  with kernel $e^{-\mathcal Q(x,y)}e^{-\mathcal Q(x',y')}$ allows to show that they are Gaussian. For further comparison, let us emphasize that
these arguments use Banach space techniques, which only apply when $p\ge 1$.

Lieb's theorem extends and unifies many important analytic results.
By considering the kernel $e^{i\langle x,y\rangle}$,
it recovers the calculation of the norm of the Fourier transform from $L^p$ to $L^{p'}$ for $p\in(1,2)$,  which was first achieved by Beckner \cite{beckner-1975}.  It also encompasses Nelson's  theorem for the Ornstein-Uhlenbeck semigroup:
$$ P_tf(x)=\int_{\mathbb R^n} f\big( e^{-t}x+\sqrt{1-e^{-2t}}y\big) \, d\gamma_n(x),$$
where $\gamma_n$ is the standard Gaussian measure on $\R^n$, $d\gamma_n(x)= (2\pi)^{-n/2}\exp(-|x|^2/2) \, dx$.  Nelson's theorem \cite{N} asserts that this operator is hypercontractive: for $p,q>1$ satisfying $e^{2t}\ge \frac{q-1}{p-1}$, and every measurable function $f$,
$$\|P_tf \|_{L^q(\gamma_n)}\le \|f\|_{L^p(\gamma_n)}.$$ 
Lieb's article also addresses multilinear operators with Gaussian kernels, and features an extension of the latter theorem in the case of real-valued kernels and functions.
(From now on we only consider real-valued functions).

\begin{theorem}[\cite{lieb-1990}]\label{theo:lieb-multi}
Let $m\ge 1$ and for $i=1,\ldots m$ let $p_i\ge 1$ and let $B_i \colon \R^n \to \R^{n_i}$
be a linear surjective map. Let $\mathcal Q$ be a semi-definite positive quadratic form  on $\mathbb R^n$. For non-identically zero functions $f_i\in L^{p_i} (\R^{n_i},\R)$, let 
$$H(f_1,\ldots, f_m)=\frac{\int_{\R^n} e^{-\mathcal Q(x)}\prod_{i=1}^m f_i(B_i x) \, dx}{\prod_{i=1}^m \|f_i\|_{p_i} }.$$
Then the supremum of $H$ over all such functions is equal to its supremum over centered Gaussian functions only.
\end{theorem}
Setting $c_i=1/p_i$ and replacing $f_i$ with $f_i^{c_i}$ gives an analogous result
for the following functional on integrable functions:
$$ I(f_1,\ldots, f_m)=\frac{\int_{\R^n} e^{-\mathcal Q(x)}\prod_{i=1}^m f_i(B_i x)^{c_i} \, dx}{\prod_{i=1}^m \left( \int_{\R^{n_i}}f_i\right)^{c_i}}.$$

The above theorem is a far-reaching extension of H\"older's inequalities.
In the case when $\mathcal Q=0$ and the maps $B_i$ are linear forms (i.e. $n_i=1$),
it recovers a celebrated inequality of Brascamp and Lieb \cite{B-L-1976}, which allowed these authors to compute the optimal constants in Young's convolution inequality, independently of Beckner \cite{beckner-1975}. Indeed using duality
$$ \|f*g\|_r \le C \|f\|_p \|g\|_q $$
can be rewritten as 
$$ \int_{\R^{2n}} f(x-y)g(y)h(x) \, dx dy\le C \|f\|_p \|g\|_q \|h\|_{r'}.$$
The classical Loomis-Whitney inequality \cite{L-W-1949} and its extension by Finner \cite{finner} are also particular cases of Theorem \ref{theo:lieb-multi}.
The Brascamp-Lieb inequalities found striking applications in convex geometry thanks to the work
of K. Ball, see e.g.  \cite{ball-1,ball-3,ball-handbook}. He put forward a  situation where the optimal constant  is 1, and  could use it in order to derive various sharp upper bounds on volumes of  convex sets.  The ``geometric'' Brascamp-Lieb inequality reads as follows: if $u_1,\ldots, u_m$ are unit vectors in $\R^n$ and if $c_1,\ldots, c_m\ge 0$ verify
\begin{equation}\label{eq:decomposition-identity-1d}
\sum_{i=1}^m c_i u_i\otimes u_i=\Id_{\R^n},
\end{equation}
where $u_i\otimes u_i$ is the orthogonal  projection onto $\R u_i$ and $\Id_{\R^n}$ is the identity of $\R^n$, then for all integrable functions $f_i \colon \R\to \R^+$,
$$
\int_{\R^n} \prod_{i=1}^m f_i\big(\langle x,u_i\rangle \big)^{c_i} dx \le \prod_{i=1}^m \left( \int_{\R} f_i\right)^{c_i}.
$$
Observe that when $m=n$ and $(u_i)_{i=1}^n$ is an orthonormal basis, the inequality becomes an equality by Fubini's theorem. So, in some sense, 
 the geometric Brascamp-Lieb
inequality describes an extremal property of orthonormal bases among sets of vectors which decompose the identity as in \eqref{eq:decomposition-identity-1d}.
An extension to functions of several variables appears in \cite{barthe-inventiones}.

\medskip
Over the years, several related inverse inequalities appeared in the literature.
A first and very simple example is the inverse H\"older inequality (obviously, H\"older's inequalities are a particular case of Lieb's theorem): if $\lambda\ge 1$ 
and $f,g\colon\R^n\to \R^+$  then
$$ \int_{\R^n} f^\lambda g^{1-\lambda} \ge \left(\int_{\R^n} f\right)^\lambda
 \left(\int_{\R^n} g\right)^{1-\lambda},$$
 provided $\int g>0$ and with the convention that $0 \cdot \infty=0$.
By rearranging the terms, this is easily deduced from the usual inequality. 
The inverse H\"older inequality can also be rewritten as a sort of duality for the (non-normed) spaces $L^p$  when  $p\in (-\infty,0)\cup(0,1)$: for $f,g\colon\R^n\to \R^+$,
$$ \int_{\R^n} fg \ge \|f\|_p \|g\|_{p'},$$
where $p'\in (-\infty,0)\cup(0,1)$ is still defined by $p^{-1}+(p')^{-1}=1$. Since
the latter inequality is sharp, it follows that for $f\ge 0$,
\begin{equation}\label{eq:dual-p<1}
\|f\|_p=\inf_{g\ge 0} \frac{\int fg}{\|g\|_{p'}} .
\end{equation}

Another instance appears in the article of Brascamp and Lieb \cite{B-L-1976}, where a sharp inverse Young inequality is proved:  given $p,q,r\in (0,1]$ with $1+1/r=1/p+1/q$, the best constant  $C$ such for all positive functions  
$$ \|f*g\|_{r}\ge C \|f\|_p \|g\|_q $$
is described, and is achieved by Gaussian functions. Observe that thanks to \eqref{eq:dual-p<1}, the latter can be rewritten as 
$$ \int_{\R^{2n}} f(x-y)g(y)h(x) \, dx dy\ge C \|f\|_p \|g\|_q \|h\|_{r'}.$$

Later on, Borell \cite{borell-1982} proved a reverse form of Nelson's hypercontractivity: if $p,q\in (-\infty,1)$ and $e^{2t}\ge \frac{1-q}{1-p}$ then for all positive functions $f$ on $\R^n$:
$$\|P_tf \|_{L^q(\gamma_n)}\ge \|f\|_{L^p(\gamma_n)}.$$ 
This bound shows that the Ornstein-Uhlenbeck semigroup improves the positivity 
of functions (for $p<0$, $\|f\|_p=1/\|1/f\|_{|p|}$ and $q\le p$).

Among the examples of reverse inequalities are the Pr\'ekopa-Leindler inequalities
\cite{prekopa,leindler}: for all $\lambda\in (0,1)$ and all  $f,g\colon\mathbb R^n\to \R^+$,
\[ \int^*_{\R^n} \sup_{z=\lambda x+(1-\lambda)y} f(x)^{\lambda}g(y) ^{1-\lambda} \; dz\ge 
\left(\int_{\R^n} f\right)^\lambda
 \left(\int_{\R^n} g\right)^{1-\lambda},\]
 where the left hand side term is an outer integral and the supremum is over all $(x,y)\in (\mathbb R^n)^2$ verifying $z=\lambda x+(1-\lambda)y$. This functional version of the Brunn-Minkowski inequality is actually a particular case of the more general reverse Brascamp-Lieb inequalities proved by the first-named author \cite{barthe-cras,barthe-inventiones}. For shortness, we only state the rank one geometric version of the reverse Brascamp-Lieb inequalities (and refer to Section \ref{sec:dual-inverse} for more details): given unit vectors $u_1,\ldots, u_m$ in $\R^n$ and $c_1,\ldots,c_m\ge 0$ satisfying the decomposition of identity \eqref{eq:decomposition-identity-1d}, for all integrable functions $f_i\colon\R \to \R^+$, it holds
$$ \int_{\R^n}^* \sup_{x=\sum_i c_i \theta_i u_i }
\prod_{i=1}^m f_i(\theta_i)^{c_i} \, dx \ge \prod_{i=1}^m \left( \int_{\R} f_i\right)^{c_i} .$$
This inequality allows to derive geometric properties which are dual to the ones that can be proved using the Brascamp-Lieb inequality (which was the motivation of K. Ball's conjecture of the reverse inequality). See \cite{barthe-inventiones,barthe-simplex} for the first results in this direction. 
Again, centered Gaussian functions achieve (or almost achieve) the optimal constant.
The reader may object that the latter inequality seems rather different from the other ones.
Nevertheless, the supremum being an $L^\infty$ norm can be viewed as a limit of $L^p$ norms. Building on this idea, Brascamp and Lieb \cite{B-L-1976} were able to deduce the Pr\'ekopa-Leindler inequality as a limit case of their inverse Young inequality. In the same way, the first-named author proved in \cite{bart-these} an extension of the inverse Young inequality which recovers as a limit the geometric reverse Brascamp-Lieb inequalities. See \cite{chen-dafnis-paouris} for further results in this direction. Actually,
in view of the content of the present paper and of the dual features of their applications,  a better name for reverse Brascamp-Lieb inequalities could be dual Brascamp-Lieb inequalities.

\medskip It is natural to ask for a general principle that would unify and explain these reverse inequalities. They definitely share some common features: they provide lower bounds on integrals involving products of positive functions in terms of their $L^p$ norms, often with $p<1$, and Gaussian functions play a special role. A significant progress in this direction was recently made by 
Chen, Dafnis and Paouris. 
One of their main results is stated in probabilistic terms:
\begin{theorem}[\cite{chen-dafnis-paouris}]\label{theo:CDP}
Let $n = n_1 + n_2 + \ldots + n_m$ be positive integers and $(X_1,\ldots,X_m)$ be a Gaussian random vector in $\R^n$ (where $X_i \in \R^{n_i}$), with covariance matrix $\Sigma$.  For each $i$, let $\Sigma_i$ denote the covariance matrix of $X_i$
(which is a diagonal block of $\Sigma$). Let $p_1,\ldots,p_m \in \R\setminus\{0\}$ and consider the block diagonal matrix $P=\mathrm{diag}(p_1\Sigma_1,\ldots,p_m\Sigma_m)$. Then for all positive functions $f_1, \ldots, f_m$,
\begin{align*}
\mathrm{if} \; \Sigma\le P, \;\mathrm{then} \quad &\mathbb E\left(
 \prod_{i=1}^m f_i(X_i)\right) \le \prod_{i=1}^m\mathbb E \left( f_i(X_i)^{p_i} \right)^{\frac{1}{p_i}},\\
 \mathrm{if} \; \Sigma\ge P, \;\mathrm{then} \quad &\mathbb E\left(
 \prod_{i=1}^m f_i(X_i)\right) \ge \prod_{i=1}^m\mathbb E \left( f_i(X_i)^{p_i} \right)^{\frac{1}{p_i}},
\end{align*}
Here the order on matrices is for the cone of semi-definite positive matrices.
\end{theorem}
The first part of the theorem is actually a direct consequence of the general Brascamp-Lieb inequality. The second part is a very neat reverse inequality.
Observe that the condition $\Sigma\ge P$ implies, by restriction to diagonal 
blocks, that $1\ge p_i$. The functional inequality can be rewritten as a lower bound on a multilinear
operator with a generalized Gaussian kernel (i.e. the exponential of a quadratic form,
without sign condition). Chen, Dafnis and Paouris use 
transformations of this inequality by linear changes of variables in order to get more,
and doing so they succeed to recover most of the above mentioned reverse inequalities. Nevertheless, their results do not have full generality and involve conditions which are sometimes difficult to check. 
 In the note \cite{barthe-wolff-cras},
we have announced a general result on the optimal constant in inequalities
of the form 
$$ 
\int_H e^{-\mathcal Q(x)} \prod_{k=1}^m f_k^{c_k}(B_k x) \, dx\ge C \prod_{k=1}^m \left(\int_{H_k} f_k \right)^{c_k}
$$
for all positive functions $f_i$. The goal of the present paper is to give a full proof of the results of \cite{barthe-wolff-cras}, and to provide an extensive answer to the following questions: when is it possible to calculate the best constant $C$ 
by considering only  Gaussian functions? or only centered Gaussian functions? when
is there a non-trivial inequality ($C>0$)?

\subsection{Notation and main results}
Here is a description of the setup of this article.
Let $0 \le m^+ \le m$ be integers and $H$, $H_1, \ldots, H_m$ be Euclidean spaces endowed with the usual Lebesgue measure. For $k=1,\ldots, m$ let $c_k$ be a real number satisfying $c_i > 0$ for $i \le m^+$ and $c_j < 0$ for $j > m^+$,
%!!! shouldn't be $c_j \le 0$???
and let $B_k \colon H \to H_k$ be a surjective linear map.
Further, let $\mathcal Q\colon H\to \R$ be a quadratic form with signature $\big(s^+(\mathcal Q), s^-(\mathcal Q)\big)$.
 %Further, let $Q \colon H \to H$ be a self-adjoint operator with $s^+(Q)$ positive eigenvalues (counting with multiplicities).
 For measurable functions $f_k \colon H_k \to [0,+\infty]$ satisfying $0 < \int_{H_k} f_k < +\infty$ define
\begin{align}\label{def:J}
  J(f_1, \ldots, f_m) = \frac{\int_H e^{-\mathcal Q(x)} \prod_{k=1}^m f_k^{c_k}(B_k x) \, dx}{\prod_{k=1}^m \left(\int_{H_k} f_k \right)^{c_k}}
\end{align}
and assume the convention $0 \cdot \infty = 0$ for the product $\prod_{k=1}^m f_k^{c_k}(B_k x)$.

Our goal is to study the minimization problem for the functional $J$. It turns out that centered Gaussian functions, i.e. the functions of the form $e^{- q(x)}$ for a positive definite quadratic form  $q$, play a pivotal role. One of our main results is the following counterpart to Lieb's Theorem~\ref{theo:lieb-multi}:

\begin{theorem}\label{thm:main-result}
Let $c_1, \ldots, c_{m^+} > 0$, $c_{m^++1}, \ldots, c_{m} < 0$ with $0 \le m^+ \le  m$. Assume that the map
$x\mapsto (B_1x,\ldots,B_{m^+}x)$ from $H$ to $H_1\times\cdots\times H_{m^+}$ is onto  and that
\[
  \dim H \ge s^+(\mathcal Q) + \dim H_1 + \cdots + \dim H_{m^+}.
\]
Then $\inf J  = \inf_{\mathcal{CG}} J$, where the right-hand side stands for the infimum of $J$ over all choices of centered Gaussian functions $g_k$.
\end{theorem}

Hence a Gaussian minimizers principle holds under some hypotheses. It may fail when they are not fulfilled, but this happens only in degenerate situations. 
The purpose of Section \ref{sec:well-posedness} is to give a full picture of these degenerate cases. This is done via a careful inspection of the values of the functional $J$ on centered
as well as general Gaussian functions (i.e. non necessarily centered) of the form $e^{-q+\ell}$,
where $q$ is a positive definite quadratic form and $\ell$ is a linear form. This allows to put forward a natural and convenient non-degeneracy condition:
\begin{equation}
\label{eq:non-degeneracy2}
 \mathcal Q_{|\bigcap_{i=1}^{m^+}\ker B_i} \mbox{ is positive definite and }\dim H \ge s^+(\mathcal Q) + \dim H_1 + \cdots + \dim H_{m^+}.
\end{equation} 

Section \ref{sec:proof-main} gives a proof of the Gaussian minimizers principle under the above
condition \eqref{eq:non-degeneracy2}. The main tool is monotone transportation as in the proof of Brascamp-Lieb inequalities of \cite{barthe-inventiones}, see also \cite{barthe-young,valdimarsson}. The presence of negative exponents introduces substantial additional difficulties.
A crucial technical step is to use a particular decomposition of the quadratic form $\mathcal{Q}$
which is adapted to the geometric structure of the problem, and is inspired by \eqref{eq:non-degeneracy2}.
One could have tried and follow other techniques which applied to Brascamp-Lieb inequalities, as semigroup interpolation or stochastic representations \cite{carlen-lieb-loss,BCCT-structure,borell,barthe-cordero,bennett-bez,lehec,chen-dafnis-paouris,neeman,ledoux}. Nevertheless the transportation technique has the advantage that it does not require any a priori structural study of extremizers.

In Section \ref{sec:geometric}, we establish the analogue in our setting of the geometric Brascamp-Lieb inequality,
and show that it is equivalent to the correlation inequality of Chen, Dafnis and Paouris presented here as Theorem~\ref{theo:CDP}. Our structural study allows a better analysis of equality conditions. Note however that their semigroup proof of the geometric inequality is simpler than ours (somehow for the transportation approach the geometric situation is not easier than the general case).

Section \ref{sec:dual-inverse} presents a dual form of the inverse Brascamp-Lieb
inequalities, which can be obtained from the very same proof. A brief summary of the various types of inequalities is provided.

The next sections are devoted to the question of existence of a non-trivial inverse Brascamp-Lieb
inequality. In other words, when is it true that $\inf J>0$? For Brascamp-Lieb inequalities, the 
analogous question (of boundedness of multilinear Gaussian operators) was settled in the general case by Bennett, Carbery, Christ and Tao \cite{BCCT-structure,BCCT-finiteness} after other contributions in the rank one case \cite{barthe-inventiones,carlen-lieb-loss}. 

Section \ref{sec:interpolation} establishes, for fixed geometric data $(\mathcal{Q},(B_k)_{k=1}^m)$,
a convexity property of the set of exponents $(c_k)_{k=1}^m$ for which $\inf J>0$. We call this set  the positivity domain of $J$.

Section \ref{sec:finiteness-1} gives a description of the positivity domain in the rank one case, i.e. when the maps $B_k$ are linear forms and when  $s^+(\mathcal Q), \, s^-(\mathcal Q)\le 1$.
In this case, the proof is simple and based on explicit calculations on Gaussian functions.
The positivity domain is a  polyhedral convex cone which we can describe as an intersection of half-spaces or  by generating vectors.

Section \ref{sec:finiteness-n} deals with the general case. We follow the inductive approach of Bennett, Carbery, Christ and Tao \cite{BCCT-finiteness}. In our setting, the fact that
the quadratic form can have positive and negative parts (and thus corresponds to fixing two Gaussian 
functions instead of one) makes the analysis more delicate.
For simplicity, we state here our characterization in the case when no kernel is involved (the general result is formulated as Theorem~\ref{th:characterization-in-general-case}).
\begin{theorem}\label{th:positivity-general-rank-no-kernel}
Let $c_1, \ldots, c_{m^+} > 0$, $c_{m^++1}, \ldots, c_{m} < 0$ with $0 \le m^+ \le  m$. Assume that the map
$x\mapsto (B_1x,\ldots,B_{m^+}x)$ from $H$ to $H_1\times\cdots\times H_{m^+}$ is a bijection. For any integrable functions $f_k \colon H_k\to [0,+\infty]$ with $\int f_k>0$, let
\[ 
  J(f_1, \ldots, f_m) = \frac{\int_H \prod_{k=1}^m f_k^{c_k}(B_k x) \, dx}{\prod_{k=1}^m \left(\int_{H_k} f_k \right)^{c_k}}\cdot 
  \]
 Then $\inf J>0$ if and only if the following two conditions are verified:
 \begin{enumerate}
\item  $\dim H=\sum_{i=1}^m c_k \dim H_k$,
 \item For every linear subspace $V\subset H$ such that $\dim V=\sum_{i=1}^{m^+} \dim B_iV$, it holds
 \[ \dim V\ge \sum_{k=1}^m c_k \dim B_kV .\]
\end{enumerate}  
If $x\mapsto (B_1x,\ldots,B_{m^+}x)$ is not surjective then $\min J=0$. If it is surjective but not injective then $\inf J=+\infty$.
\end{theorem}
 Let us emphasize that in the positivity domain, $c_i\ge 1$ for $i=1,2,\ldots,m^+$ (see Proposition~\ref{prop:c-i-ge-1}), which means that our results can be stated in terms of $L^{p_k}$-spaces with $p_k=1/c_k\le 1$ and possibly negative.

\medskip

Let us conclude this introduction with some more notation and comments on the setting.
In the rest of the paper, we use the notation $\inf_{\mathcal{G}} J$ for the infimum 
of $J$ on $m$-tuples of Gaussian functions (not necessarily centered).
 
We could only require the sets $H_k$ to be finite dimensional vector spaces equipped with a Lebesgue measure. Euclidean structures are not relevant for our problems, but working in Euclidean spaces is convenient for explicit calculations for Gaussian functions, as quadratic functions are represented by symmetric linear maps. Also Euclidean structures induce a canonical choice of Lebesgue measure. In a context of a Euclidean space, we will use $\scalar{\cdot,\cdot}$ for the inner product, $|\cdot|$ for the Euclidean norm and for a linear map $L$ between Euclidean spaces, $L^*$ will stand for its adjoint. We denote by $Q$ a self-adjoint map on $H$ such that for all $x \in H$,
\[
  \mathcal{Q}(x) = \pi \scalar{x, Qx}.
\]

Eventually let us mention that we allow  $H_k=\{0\}$ and choose by convention
the Lebesgue measure to be the Dirac mass at $0$. For such a $k$, the terms involving
$f_k$ can be  canceled out from \eqref{def:J}.
%!!! a word on $c_j = 0$???

 \subsection{Acknowledgments} This work has benefited from discussions on related topics with several colleagues. We would like to thank in particular  J\'er\^ome Bertrand,  Max Fathi, Piotr Mi{\l}o{\'s}, Krzysztof Oleszkiewicz, Grigoris Paouris.

\section{Well-posedness of the minimization problem and the minimum value}
\label{sec:well-posedness}

 We denote by  $B_+$ the linear map $(B_1,\ldots,B_{m^+})$, that is 
  \begin{equation}\label{def:B+}
\begin{array}{rcc}
B_+: \;  H & \to & H_1\times\cdots\times H_{m^+} \\
  x & \mapsto & (B_1x,\ldots,B_{m^+} x).
\end{array}
\end{equation}

\subsection{A non-degeneracy condition}

We put forward a simple condition on the above map $B_+$ which allows the functional $J$ to vanish.

\begin{lemma}\label{lem:nonsurjective1}
\begin{enumerate}
\item[(i)] If the map $B_+$ from $H$ to $H_1\times\cdots\times H_{m^+}$
is not onto, then $\inf J=\min J=0$.
\item[(ii)] Conversely, if the map $B_+$ is onto, then all $f_k \colon H_k \to [0, +\infty]$ with $0 < \int_{H_k} f_k < +\infty$ one has $J((f_k)) \in (0, +\infty]$.
\end{enumerate}
\end{lemma}
\begin{proof}
(i) Consider a point $(a_1,\ldots,a_{m^+})\in H_1\times\cdots\times H_{m^+}$, so that its Euclidean
distance to the range  of the non-surjective linear map $B_+$ is at least $\sqrt{m^+}$.
If we denote by $B_H(x,r)$ the open ball of center $x$ and radius $r$ in $H$, then 
$$ \big(B_{H_1}(a_1,1)\times \cdots \times  B_{H_m^+}(a_{m^+},1) \big) \cap \big\{ (B_1x, \ldots, B_{m^+}x); \; x\in H\big\}=\emptyset.$$
For $1\le i\le m^+$ consider the function $f_i\colon H_i\to \R^+$ defined as the characteristic function
of $B_{H_i}(a_i,1)$. Then the latter empty intersection ensures that for all $x\in H$,
$$\prod_{i=1}^{m^+} f_i^{c_i}(B_ix)=0.$$
Therefore, for any choice of functions $(f_j)_{j>m^+}$, it holds $J(f_1,\ldots,f_m)=0$.

(ii) Since $\prod_{i\le m^+} \int_{H_i} f_i > 0$, the measure of points $(z_1,\ldots,z_{m^+})\in H_1\times\cdots\times H_{m^+}$ for which $\prod_{i\le m^+} f_i^{c_i} (z_i)>0$ is positive. From the hypothesis that $B_+$ is onto it follows that the measure of 
 $$\{x\in H; \prod_{1\le i\le m^+} f_i^{c_i} (B_ix)>0\}$$
 is positive. To conclude it is enough to notice that integrability of $f_j$ (for $m^+ < j < m$) implies that $\prod_{j = 1+m^+}^m f_j^{c_j}(B_j x) > 0$ $x$-a.e. in $H$. 
\end{proof}

As a consequence of the previous lemma, we will often work under the non-degeneracy  assumption that
$B_+$ is surjective.

\subsection{Calculations for centered Gaussian functions}\label{subsec:centered-gaussians}

Recall the classical formula $\int_{\mathbb R} e^{-\pi t^2} dt=1$. From this, it follows that
for any self-adjoint operator $A$ on $\mathbb R^d$ (or a $d$-dimensional Euclidean space),
$$ \int_{\R^d} e^{-\pi\scalar{x,Ax}} dx=\left\{ 
\begin{array}{ll} \det(A)^{-1/2} & \mbox{if $A$ is positive definite},\\
   +\infty & \mbox{otherwise}.
\end{array}
\right.  $$
When $A$ is positive definite, we denote $g_A$ the centered Gaussian function defined by
\[
  g_A(x) = e^{-\pi \scalar{x, Ax}}.
\]
An elementary computation shows that
\begin{align}\label{eq:J-on-Gaussian-input}
  J(g_{A_1}, \ldots, g_{A_m}) = \begin{cases}
    \left(\frac{\det(Q + \sum_{k=1}^m c_k B_k^\ast A_k B_k)}{\prod_{k = 1}^m (\det A_k)^{c_k}}\right)^{-1/2} & \text{if $(A_1, \ldots, A_m) \in \Lambda$,} \\
    \infty & \text{otherwise,}
  \end{cases}
\end{align}
where the set $\Lambda$ is defined as follows
\[
  \Lambda = \Big\{ (A_1, \ldots, A_m) \colon A_k \colon H_k \to H_k \text{ and } Q + \sum_{k=1}^m c_k B_k^\ast A_k B_k \colon H \to H \text{ are positive definite} \Big\}.
\]
Therefore the infimum of $J$ over centered Gaussian functions equals $D^{-1/2}$, where
\begin{align}\label{def:BL-constant}
  D = \sup\left\{\frac{\det(Q + \sum_{k=1}^m c_k B_k^\ast A_k B_k)}{\prod_{k=1}^m (\det A_k)^{c_k}} \colon (A_1, \ldots, A_m) \in \Lambda \right\},
\end{align}
with the convention that $D = 0$ for $\Lambda = \emptyset$ (and thus $D^{-1/2} = \infty$). 

\subsection{Ensuring finiteness for some functions}
We investigate the  existence of non-zero functions for which $J$
takes a finite value.
The right setup for the functional $J$ to be non degenerate on centered Gaussian functions is the following condition:
\begin{equation}\label{eq:injectivity-condition}
\mathcal Q \textup{ is positive definite on } \ker B_+.
\end{equation}

\begin{prop}\label{prop:lamba-emptyset}
The following assertions are equivalent:
\begin{enumerate} 
\item[(i)] there exist centered Gaussian functions $g_1, \ldots, g_m$ with    $J(g_1, \ldots, g_m) <+\infty$,
  \item[(ii)] $\Lambda \neq \emptyset$,
  \item[(iii)] $\mathcal Q_{|\ker B_+}$ is positive definite.
\end{enumerate}
\end{prop}
\begin{proof}
The equivalence of  $(i)$ and $(ii)$ is a direct consequence of Formula \eqref{eq:J-on-Gaussian-input}.
Assertion $(ii)$ is equivalent to the existence of positive maps $(A_k)_{k=1}^m$ such that 
$Q+\sum_{k=1}^m c_k B_k^*A_k B_k$ is positive definite.
This can be rewritten as 
\[Q+\sum_{i=1}^{m^+} c_i B_i^*A_i B_i>\sum_{j=1+m^+}^m |c_j| B_j^*A_j B_j. \]
Since one may choose the matrices $A_j>0$ arbitrarily small, $(ii)$ is equivalent
to the existence of positive maps $(A_i)_{i=1}^{m^+}$ such that 
\[Q+\sum_{i=1}^{m^+} c_i B_i^*A_i B_i>0. \]
Similarly, one may choose each matrix $A_i$ as an arbitrarily large multiple of the identity on $H_i$. Hence $(ii)$ is equivalent to the existence of  $D>0$ such 
that 
$Q+\sum_{i=1}^{m^+} D B_i^* B_i>0$,
or in terms of quadratic forms
\[x\mapsto \langle x,Qx\rangle + D \sum_{i=1}^{m^+} | B_ix|^2=
 \langle x,Qx\rangle + D |B_+x|^2\]
 is positive definite. We may conclude thanks to 
Lemma \ref{lem:defpos-subspace} below for  $L = B_+$.
\end{proof}

\begin{lemma}\label{lem:defpos-subspace}
Let $\mathcal R$ be a quadratic form on $\mathbb R^d$. Let $L \colon \mathbb R^d\to \mathbb R^k$ be a linear map between Euclidean spaces. Then the following assertions
are equivalent:
\begin{enumerate}
 \item[(i)] There exists $D>0$ such that the quadratic form $x\mapsto \mathcal R(x)+D |Lx|^2$ is positive definite,
 \item[(ii)] $\mathcal R_{|\mathrm{ker}L}$ is positive definite.
% \item[(iii)] $\{x\in \mathbb R^d;\; \mathcal R(x)\le 0\}\cap \mathrm{ker} L =\{0\}$
\end{enumerate}
\end{lemma}
\begin{proof}
%The last two assertions are obviously equivalent.
If  $(i)$ holds then for  $D$ large enough $x\mapsto \mathcal R(x)+D |Lx|^2$ is positive definite. Hence its restriction to $\ker L$ is also positive definite, namely 
$\mathcal R_{|\mathrm{ker}L}$ is positive definite.

Next, let us show that $(ii)$ implies $(i)$, by contradiction. If $(i)$ is not true then for every
 integer $N$ there exists a unit vector $x_N\in S^{d-1}\subset \mathbb R^d$ such that
  $$  \mathcal R(x_N)+N |Lx_N|^2\le 0.$$
  By compactness of the unit sphere, one can find a converging subsequence $(x_{N_k})$. Let
  $x\in S^{d-1}$ denote its limit.
  
  Since $\mathcal R(x_{N_k})\le -N_k  |Lx_{N_k}|^2\le 0$, passing to the limit gives $\mathcal R(x)\le 0$.
  Moreover 
  $$ |Lx_{N_k}|^2\le -\frac{\mathcal R(x_{N_k})}{N_k}\le -\frac{\min_{S^{d-1}}\mathcal R}{N_k},$$
  so by continuity, letting $k$ go to infinity $|Lx|=0$. Hence $x\in \mathrm{ker} L\setminus\{0\}$
  verifies $\mathcal R(x)\le 0$, meaning that the restriction of  $\mathcal R$ to $\mathrm{ker}L$
  is not positive definite.
  \end{proof}

 Combined with the above proposition, the forthcoming one shows that if the map $B_+$ defined in \eqref{def:B+}
 is surjective then the following holds:  there  exists  functions for which  $J$ is finite
 if and only if there exists  centered Gaussian functions for which $J$ is finite.

\begin{prop}\label{prop:finitevalues2}
Assume  that the map $B_+$ is surjective.
If $\mathcal Q_{|\ker B_+}$ is not positive definite, then for every functions $f_k \colon H_k\to \R^+$
with $\int f_k \in (0,+\infty)$, the quantity $J((f_k)_{1\le k\le m})$ is $+\infty$.
\end{prop}

\begin{proof}
Without loss of generality, we consider arbitrary functions $f_k$ with $\int f_k=1$.
By our hypothesis, there exists a unit vector $v\in H$ such that
$\langle v,Qv\rangle\le 0$ and for all $1\le i\le m^+$, $B_iv =0$. Let $S\subset H$ be any linear complement of  $\R v$. Then there is a positive constant $c_S$ such that,
decomposing each element of $H$ as $x=y+tv$ with $y\in S$ and $t\in\R$
\begin{eqnarray*}
J((f_k))&=& c_S \int  e^{-\pi\big(\langle y,Qy\rangle+ 2 t \langle Q y,v\rangle+t^2 \langle v,Qv\rangle\big)} \prod_{1\le i\le m^+} f_i^{c_i} (B_iy) 
\prod_{m^+ < j \le m} f_j^{c_j} (B_jy+tB_jv) \, dt dy \\
&\ge & c_S \int_S e^{-\pi\langle y,Qy\rangle} \prod_{1\le i\le m^+} f_i^{c_i} (B_iy) 
\left( \int_{\R} e^{-2\pi t \langle Q y,v\rangle} \prod_{m^+ < j \le m} f_j^{c_j} (B_jy+tB_jv) \, dt\right) dy
\end{eqnarray*}
Let us prove that $y$-almost everywhere in $S$, the inner integral equals $+\infty$. To do this, we prove that $y$-a.e. in $S$, the non-negative function $t\mapsto \prod_{m^+ < j \le m} f_j^{c_j} (B_jy+tB_jv)$ is bounded from below by a positive constant, except maybe on a set of finite Lebesgue measure. Here are the details:

  If $m^+ < j \le m$ is such that $B_jv=0$ then for all $t$,  $f_j^{c_j} (B_jy+tB_jv)= f_j^{c_j} (B_jy)$. Since $B_j \colon H \to H_i$ is surjective 
  and $B_jv=0$ it follows that the restriction of $B_j$ to $S$ is also surjective.
  Since $\int_{H_j} f_j<+\infty$, we know that $f_j<+\infty$ a.e. in $H_j$.
  As the preimage of a Lebesgue negligible set by a linear surjection is also Lebesgue negligible,
 we deduce that $y$-a.e in $S$, $f_j(B_jy)<+\infty$.  Using that $c_j$ is negative, we get that $y$-a.e. in $S$, $f_j^{c_j}(B_jy)>0$.
 
 If $m^+ < j \le m$ is such that $B_jv \neq 0$ we proceed differently. First, for each $y$, one can decompose $B_jy$ using orthogonal projections as follows 
 $$B_j y=P_{(\R B_j v) ^\bot} B_j y + P_{\R B_j v} B_jy=L_j y+ t_j(y) B_j v$$
 where $L_j=  P_{(\R B_jv) ^\bot} B_j$. 
By translation invariance of  Lebesgue's measure 
\begin{eqnarray*}
 \mathrm{vol}_1\left( \{t\in \R; f_j^{c_j} (B_jy+tB_jv) \le 1 \}\right)&=&
 \mathrm{vol}_1\left( \{t\in \R; f_j (L_jy+(t+t_j(y))B_jv) \ge 1 \}\right)\\
 &=& \mathrm{vol}_1\left( \{t\in \R; f_j (L_jy+ tB_jv)\ge 1 \}\right).
\end{eqnarray*}
Next 
$$ 1=\int_{H_j} f_j= |B_jv| \int_{(\R B_jv)^\bot} \int_\R f_j(z+tB_jv) dt dz,$$
hence $z$-a.e. in $(\R B_jv)^\bot$, the inner integral is finite and therefore
 $$\mathrm{vol}_1\left( \{t\in \R; f_j (z+ tB_jv)\ge 1 \}\right)<+\infty.$$
 Since by construction the above map $L_j \colon S \to (\R B_jv)^\bot$ is linear and onto, 
 it follows that $y$-a.e. in $S$, $\{t\in \R; f_j^{c_j} (B_jy+tB_jv) \le 1 \}$
has finite Lebesgue measure.

Putting everything together, we obtain as claimed that $y$-a.e. in $S$,  
$t\mapsto \prod_{j> m^+} f_j^{c_j} (B_jy+tB_jv)$ is bounded from below by $1$, except for a set of finite Lebesgue measure. Lemma \ref{lem:exp} below then yields that 
$y$-a.e. in $S$ the inner integral in the latter expression for $J((f_k))$ is infinite.
Consequently 
$$J((f_k))\ge c_S \int_S e^{-\pi\langle y,Qy\rangle} \prod_{1\le i\le m^+} f_i^{c_i} (B_iy) \times
(+\infty)\, dy.$$
So $J((f_k))=+\infty$ provided the set of elements $y\in S$ for which 
$\prod_{i\le m^+} f_i^{c_i} (B_iy)>0$ has positive measure (for at least one choice of $S$). To this end we use the hypothesis that $B_+$ is surjective and Lemma~\ref{lem:nonsurjective1}(ii) to obtain that $J((f_k)) > 0$, which readily implies that the set of $x \in H$ for which $\prod_{i\le m^+} f_i^{c_i} (B_ix) > 0$ has positive measure. By integrating over the Grassmannian of hyperplanes in $H$ there exists a non-negligible set of hyperplanes $\mathcal{S}$ such that for each $S \in \mathcal{S}$, the set of $y \in S$ for which $\prod_{i\le m^+} f_i^{c_i} (B_iy)>0$ has positive measure. Since the set of hyperplanes of $H$ containing $v$ is negligible, there must be a hyperplane not containing $v$ in $\mathcal{S}$.
\end{proof}

\begin{lemma}\label{lem:exp}
Let $A\subset \R$ be a Borel set. 
If $A^c$ has finite Lebesgue measure then $\int_A e^t dt=+\infty$
\end{lemma}
\begin{proof}
Assume on the contrary that $\int_A e^t dt=C<+\infty$. Then for every $N\in \mathbb N$,
 $e^N \mathrm{vol}_1(A\cap [N,N+1)) \le C$. Hence $\mathrm{vol}_1(A^c\cap [N,N+1))\ge 1-C e^{-N}$. Summing over $N\in \mathbb N$ gives that $A^c$ has infinite measure.
\end{proof}

\subsection{On the effect of translating Gaussian functions and consequences of positivity}\label{subsec:translated-gaussians}

In order to explain the relevance of the hypothesis
\begin{equation}\label{eq:surjectivity-condition}
  \dim H \ge s^+(\mathcal Q) + \dim H_1 + \cdots + \dim H_{m^+}
\end{equation}
which appears in Theorem~\ref{thm:main-result}, we study the value of the functional $J$ on non-centered Gaussian functions. 

In order to handle the Gaussian kernel $\exp(-\mathcal{Q})$ as two additional (fixed) Gaussian functions (one function corresponding to a positive exponent and the other corresponding to a negative exponent), we will decompose the quadratic form $\mathcal{Q}$ into a positive and negative part. To this end, note the following simple fact:
\begin{lemma}\label{lem:kerS-kerT-H}
Let $S \colon H \to X$ and $T \colon H \to Y$ be linear maps. The map $(S, T) \colon H \to X \times Y$ is surjective if and only if $S$ and $T$ are surjective and
\begin{equation}\label{eq:kerS-kerT-H}
  \ker S + \ker T = H.
\end{equation}
It is a linear isomorphism if and only if $S$ and $T$ are surjective and $\ker S \oplus \ker T=H$.
\end{lemma}
\begin{proof}
Assume that the  map $(S, T)$ is surjective. Then  $S$ and $T$ are surjective too. For~\eqref{eq:kerS-kerT-H}, consider any $x \in H$ and we aim to decompose it into $\ker S$ and $\ker T$. By surjectivity of the map $(S, T)$, there exists $y \in H$ such that $(Sy, Ty) = (Sx, 0)$.
Therefore $x - y \in \ker S$ and $y \in \ker T$, hence
\[
  x = (x - y) + y \in \ker S + \ker T.
\]

For the other implication, take any $x, y \in H$. From the hypothesis~\eqref{eq:kerS-kerT-H} it follows that there exists $v \in \ker S$ and $w \in \ker T$ such that $x - y = v + w$, i.e.
\[
  x - v = y + w.
\]
Denote $z = x - v$. We clearly have
\[
  S z = S (x - v) = S x \quad \textup{and} \quad T z = T (y + w) = T y,
\]
i.e. $(S, T)z = (Sx, Ty)$. This means that $(S, T)$ is surjective, since $(Sx, Sy)$ is arbitrary in $S H \times T H = X \times Y$,
% where $S H = X$ and $T H = Y$ holds true 
by surjectivity of $S$ and $T$.

The second part of the lemma follows from  $\ker (S,T)=\ker S\cap \ker T$.
\end{proof}

In what follows we consider any decomposition of $\mathcal{Q}$ of the form
\begin{equation}\label{eq:any-decomposition-of-Q}
  \mathcal{Q}(x) = c_0 \mathcal{Q}_+(B_0 x) + c_{m+1} \mathcal{Q}_-(B_{m+1} x)
\end{equation}
where $c_0 > 0 > c_{m-1}$ are real numbers, $B_0 \colon H \to H_0$ and $B_{m+1} \colon H \to H_{m+1}$ are surjective linear  maps onto Euclidean spaces $H_0$ and $H_{m+1}$ (respectively), such that the map $(B_0, B_{m+1})$ is surjective, or equivalently, by Lemma~\ref{lem:kerS-kerT-H},
\begin{equation}\label{eq:ker-B0-Bm1-H}
\ker B_0 + \ker B_{m+1} = H
\end{equation}
and $\mathcal{Q}_+$, $\mathcal{Q}_-$ are positive definite quadratic forms on $H_0$ and $H_{m+1}$ (respectively). 

The existence of such decomposition is obvious by considering an eigenvalue decomposition of the self-adjoint map $Q$. Then $B_0$ can be taken as the orthogonal projection of $H$ onto $H_0$ being a subspace spanned by eigenvectors corresponding to positive eigenvalues of $Q$, and similarly $B_{m+1}$. One can take $c_0 = 1$ and $c_{m+1} = -1$. Condition~\eqref{eq:ker-B0-Bm1-H} follows from orthogonality of $H_0$ and $H_{m+1}$ in $H$. Moreover, we clearly have
\begin{equation}\label{eq:signature-dim-H0-Hm1}
s^+(\mathcal{Q}) = \dim H_0, \quad s^-(\mathcal{Q}) = \dim H_{m+1}.
\end{equation}

Conversely, any decomposition of $\mathcal{Q}$ as in~\eqref{eq:any-decomposition-of-Q} satisfies~\eqref{eq:signature-dim-H0-Hm1}. Indeed, by~\eqref{eq:ker-B0-Bm1-H}, one can find a complement subspace $V$ of $\ker B_0$ in $H$ which satisfies $V \subseteq \ker B_{m+1}$ and hence $\mathcal{Q}$ is positive definite on $V$. This yields $s^+(\mathcal{Q}) \ge \dim V = \dim H - \dim \ker B_0 = \dim H_0$. On the other hand, $\mathcal{Q}$ is negative semi-definite on $\ker B_0$, hence $s^+(\mathcal{Q}) \le \dim H - \dim \ker B_0 = \dim H_0$. The same argument shows the second assertion of~\eqref{eq:signature-dim-H0-Hm1}.

\medskip

The starting point of the forthcoming calculations is that for any self-adjoint map $A$ on $\mathbb R^d $ and any vector $b\in \mathbb R^d$,
\begin{equation}\label{eq:integral-of-shifted-gaussian}
\int_{\R^d} e^{-\pi\scalar{x,Ax}+2\pi \scalar{b,x}} dx=\left\{ 
\begin{array}{ll} e^{\pi\scalar{A^{-1}b,b}}\det(A)^{-1/2} & \mbox{if $A$ is positive definite},\\
   +\infty & \mbox{otherwise}.
\end{array}
\right.
\end{equation}

For $k=1,\ldots, m$, let $A_k$ be a positive definite map on $H_k$. Moreover, let $A_0$ be positive definite map $H_0$ such that $\mathcal{Q}_+(x) = \pi \scalar{x, A_0 x}$, and similarly define $A_{m+1}$ for $\mathcal{Q}_-$. With this notation~\eqref{eq:any-decomposition-of-Q} becomes
\begin{equation}\label{eq:any-decomposition-of-Q-as-sefl-adj-map}
  Q = \sum_{k\in \{0,m+1\}} c_k B_k^\ast A_k B_k.
\end{equation}
For $k=0, \ldots, m+1$ fix any $b_k\in H_k$. Since the map $(B_0, B_{m+1})$ is surjective, we can find a vector $b \in H$ such that $B_0 b = b_0$ and $B_{m+1} b = b_{m+1}$.

We calculate the value of $J$ on the translated Gaussian functions $g_{A_k}(\cdot+b_k)$. By translation invariance of Lebesgue's measure, $\int g_{A_k}(\cdot+b_k) = \det(A_k)^{-1/2}$. In order to introduce a translation also in the Gaussian kernel, we perform a change of variable $y=x+b$ in the integral
\begin{eqnarray*}
J\big((g_{A_k}(\cdot+b_k))\big) \prod_{k=1}^m \det(A_k)^{-c_k/2} 
&=&  \int_H e^{-\pi\scalar{y,Qy}} \prod_{k=1}^m g_{A_k}^{c_k}(B_k y + b_k) \, dy \\
&=& \int_H e^{-\pi\scalar{x+b,Q(x+b)}} \prod_{k=1}^m g_{A_k}^{c_k}(B_k x+B_k b+b_k) \, dx \\
&=& \int_H e^{-\pi\sum_{k \in \{0, m+1\}} \scalar{B_k x+b_k,A_k(B_k x+b_k)}} \prod_{k=1}^m g_{A_k}^{c_k}(B_k x+B_k b+b_k) \, dx.
\end{eqnarray*}
Here it is convenient to set $u_k=B_k b+b_k$ for $k=1,2,\ldots,m^+$ and $u_k = b_k$ for $k \in \{0,m+1\}$ (for the sake of consistency of notation). Developing all the quadratic terms shows that the 
latter integral is equal to 
\begin{eqnarray*}
&&\int_H e^{-\pi\big(\sum_{k=0}^{m+1} c_k \scalar{A_k(B_k x+u_k),B_k x+u_k} \big)} \, dx \\
&=& \int_H e^{-\pi\big(\scalar{x,Ax}+2\scalar{x,v}+\sum_{k=0}^{m+1} c_k\scalar{A_k u_k,u_k} \big)} \, dx,
\end{eqnarray*}
 where we have set $A=\sum_{k=0}^{m+1} c_k B_k^*A_k B_k$ and $v = \sum_{k=0}^{m+1} c_k B_k^*A_k u_k$.
From the above calculations and~\eqref{eq:integral-of-shifted-gaussian} it follows that
\begin{equation}\label{eq:J-finite-on-shifted-gaussians}
J\big((g_{A_k}(\cdot+b_k))\big) < +\infty \quad \iff A \textup{ is positive definite (i.e. $(A_k)_{k=1}^m\in \Lambda$)}
\end{equation}
and in case $A$ is positive definite,
\begin{equation} \label{eq:J-translated-gaussians}
 J\big((g_{A_k}(\cdot-B_kb+u_k))\big) = \left( \frac{\det(A)}{\prod_{k=1}^m \det(A_k)^{c_k}}\right)^{-\frac12}
e^{\pi\big(\scalar{A^{-1}v,v} - \sum_{k=0}^{m+1} c_k\scalar{A_k u_k,u_k}\big)}.
\end{equation}
In terms of the 
translation parameters $u_k$ (for $k=0,\ldots,m+1$), the term inside the exponential is a quadratic form. Hence
its infimum is 0 if the quadratic form is positive semi-definite and  $-\infty$ else.  
In the latter case, we get that the infimum of $J$ is zero because of certain non-centered Gaussian 
functions, while in the former case we get that $J$ takes smaller values on  centered Gaussians $(g_{A_k})$
than on their translates. In short,
$$ \inf_{\mathcal G} J\in \big\{ 0, \inf_{\mathcal{CG}}J \big\}.$$

\begin{prop}\label{prop:nonzeroinf}
Suppose that $\mathcal Q$ is positive definite on $\ker B_+$ (which guarantees that $J$ is finite
for some centered Gaussian functions) and that $\inf_{\mathcal{G}} J>0$. Assuming the notation that is involved in~\eqref{eq:any-decomposition-of-Q} and~\eqref{eq:any-decomposition-of-Q-as-sefl-adj-map}, the following assertions hold true:
\begin{enumerate}
\item If $(A_k)_{k=1}^m\in \Lambda$ then for all $v_k\in H_k$, $k=0,\ldots, m+1$, setting $A=\sum_{k=0}^{m+1} c_k B_k^*A_k B_k$ and  $v=\sum_{k=0}^{m+1} c_k B_k^* v_k$, it holds
$$\scalar{A^{-1}v,v}\ge \sum_{k=0}^{m+1} c_k\scalar{A_k^{-1}v_k,v_k}.$$
\item The map $x \mapsto (B_0 x,B_1 x,\ldots,B_{m^+}x)$ from $H$ to $H_0\times \cdots\times H_{m^+}$ is onto.
\item $\dim H \ge s^+(\mathcal Q)+\sum_{i=1}^{m^+} \dim H_i$.
\end{enumerate}
\end{prop}

\begin{proof}
Since  $\inf_{\mathcal{G}} J>0$, reasoning as above on the argument of the exponential term in \eqref{eq:J-translated-gaussians} shows that if  $v=\sum_{k=1}^m c_k B_k^*A_k u_k$ then 
$$\scalar{A^{-1}v,v} \ge \sum_{k=0}^{m+1} c_k\scalar{A_k u_k,u_k}.$$
Applying this to $u_k=A^{-1}v_k$ ($k=0,1,\ldots,m+1$) concludes the proof of the first item.

Let us address the second part of the claim.
By duality, our goal is to show that the map $(v_0,\ldots,v_{m^+})\mapsto \sum_{i=0}^{m^+}
c_i B_i^* v_i$ is injective. So we assume that $\sum_{i=0}^{m^+} c_i
B_i^* v_i=0$, and we want to prove that $v_0=\cdots=v_{m^+}=0$ (recall that $c_i\neq 0$).
If we set $v_j=0$ for $m^+ < j \le m+1$, it holds that $0=\sum_{k=0}^{m+1} c_k B_k^*v_k$.
 Thanks to Proposition \ref{prop:lamba-emptyset}, we may find $(A_k)_{k=1}^m\in \Lambda$ and apply the first item of the present Proposition \ref{prop:nonzeroinf}; it gives
 that 
 $$0=\scalar{A^{-1}0,0}\ge \sum_{k=0}^{m+1} c_k\scalar{A_k^{-1}v_k,v_k}= \sum_{i=0}^{m^+} c_i\scalar{A_i^{-1}v_i,v_i}.$$
 Since $c_i>0$ for $0 \le i \le m^+$, we deduce that $\scalar{A_i^{-1}v_i,v_i}=0$, thus $v_i=0$. 
 
 Eventually the third point of the claim is a direct consequence of the second one (surjectivity implies that the dimension of the target space is not bigger than that of the initial space, and $s^+(\mathcal Q)=\dim H_0$ by~\eqref{eq:signature-dim-H0-Hm1}).
\end{proof}

\begin{prop}\label{prop:nonzeroinf-for-centered-gaussians}
Assume that $\mathcal Q$ is positive definite on $\ker B_+$ and that 
   $\inf_{\mathcal{CG}} J>0$. Then $B_+$ is surjective.
\end{prop}

\begin{proof}
We proceed by contradiction. Assume that $B_+$ is not onto. Then for some $i \in \{1,2,\ldots, m^+\}$ and $v \in H_i \setminus\{0\}$, the vector 
$$(0,\ldots, 0, \underbrace{v}_{\textup{$i$-th component}}, 0, \ldots, 0) \in H_1 \times \cdots \times H_{m^+}$$
 is not in the image of the  $B_+=(B_1,\ldots, B_{m^+}) \colon H \to H_1 \times \cdots \times H_{m^+}$. Fix such $i$ and $v$ and let $P_{(\R v)^\bot} \colon H_i \to H_i \cap (\R v)^\bot$ be an orthogonal projection. Put $\tilde{H}_i = H_i \cap (\R v)^\bot$ and
\[
  \tilde{B}_i = P_{(\R v)^\bot} B_i \colon H \to \tilde{H}_i.
\]
From the surjective maps $B_1, \ldots, B_{i-1}, \tilde{B}_i, B_{i+1}, \ldots, B_{m^+}$,
we construct a map $\tilde B_+=(B_1, \ldots, B_{i-1}, \tilde{B}_i, B_{i+1}, \ldots, B_{m^+})$ from $H$ to $H_1 \times \cdots \times  H_{i-1} \times  \tilde{H}_i \times H_{i+1} \times\cdots \times  H_{m^+}$.

Now we show that $Q$ is positive definite on $\ker \tilde{B}_+$. To this end, take any $x \in H$ for which $(B_1 x, \ldots B_{i-1} x, \tilde{B}_i x, B_{i+1} x, \ldots, B_{m^+} x) = (0, \ldots, 0)$. Hence 
$$B_+x=(B_1 x, \ldots, B_{m^+} x) \in (0, \ldots, 0, \underbrace{\R v}_{\textup{$i$-th component}}, 0, \ldots, 0),$$ 
but since $(0, \ldots, 0, 
%\underbrace{\R v}_{\textup{$i$-th component}},
v, 0, \ldots, 0)$
 is not in the image of $B_+=(B_1, \ldots, B_{m^+})$, we must have $B_+x =0$.
 By  assumption, $Q$ is positive definite on $\ker B_+$, which gives $\scalar{Qx, x} > 0$ if $x\neq 0$. 

Applying Proposition~\ref{prop:lamba-emptyset} to $Q$ and the maps $B_1, \ldots, B_{i-1}, \tilde{B}_i, B_{i+1}, \ldots, B_m$ we have positive maps $A_k \colon H_k \to H_k$ for $k\neq i$  and $\tilde{A}_i \colon \tilde{H}_i \to \tilde{H}_i$ such that the map
\[
  Q + c_i \tilde{B}_i^\ast \tilde{A}_i \tilde{B}_i + \sum_{\stackrel{1 \le k \le m}{k \neq i}} c_k B_k^\ast A_k B_k \quad \textup{is positive.}
\]
For $t > 0$ define a positive map $A^{(t)}_i = P_{(\R v)^\bot}^\ast \tilde{A}_i P_{(\R v)^\bot} + t v v^\ast \colon H_i \to H_i$. Note that
\[
  \lim_{t \to 0^+} \det \Big(Q + c_i B_i^\ast A^{(t)}_i B_i + \sum_{\stackrel{1 \le k \le m}{k \neq i}} c_k B_k^\ast A_k B_k \Big) =
  \det\Big(Q + c_i \tilde{B}_i^\ast \tilde{A}_i \tilde{B}_i + \sum_{\stackrel{1 \le k \le m}{k \neq i}} c_k B_k^\ast A_k B_k \Big) > 0
\]
while $\lim_{t \to 0+} \det A^{(t)}_i = 0$. Therefore using the formula~\eqref{eq:J-on-Gaussian-input} we see that
\[
  \lim_{t \to 0^+} J(g_{A_1}, \ldots, g_{A_{i-1}}, g_{A^{(t)}_i}, g_{A_{i+1}}, \ldots, g_{A_m}) = 0,
\]
where $g_A(x)$ is a centered Gaussian function $e^{-\pi\scalar{Ax,x}}$.
\end{proof}

\subsection{Case analysis and non-degeneracy hypotheses}\label{subsection:case-analysis}
The goal of this section is to give a full view of the cases when the best 
constant in inverse Brascamp-Lieb inequalities can be computed with Gaussian
functions only. 

\smallskip
Case 0.0:  The restriction of $\mathcal Q$ to $\ker B_+$ is not positive definite and $B_+$ is not surjective. 
In this case, Lemma~\ref{lem:nonsurjective1}(i) implies that $\min J=0$. On the other hand, Proposition~\ref{prop:lamba-emptyset}
implies that $\inf_{\mathcal{CG}}J= +\infty$, or equivalently $\Lambda = \emptyset$, which combined with~\eqref{eq:J-finite-on-shifted-gaussians} implies that also $\inf_{\mathcal G} J=+\infty$.
Gaussian functions do not allow to compute the infimum of $J$.

\smallskip
Case 0.1: The restriction of $\mathcal Q$ to $\ker B_+$ is not positive definite and $B_+$ is surjective.
 Proposition \ref{prop:finitevalues2} ensures that $\inf J=+\infty$.
The functional is always infinite. In a very degenerate sense, centered Gaussian 
functions allow to compute the infimum of $J$.

\smallskip
Case 1.0.0: The restriction of $\mathcal Q$ to $\ker B_+$ is positive definite, $\dim H< s^+(\mathcal Q)+\sum_{i=1}^{m^+} \dim H_i$ and $B_+$ is not surjective.
By Lemma \ref{lem:nonsurjective1}(i), $\min J=0$. Proposition \ref{prop:nonzeroinf-for-centered-gaussians} ensures that 
$\inf_{\mathcal{CG}}J=0$. 

\smallskip
Case 1.0.1:  $\mathcal Q$ is positive definite on  $\ker B_+$, $\dim H< s^+(\mathcal Q)+\sum_{i=1}^{m^+} \dim H_i$ and $B_+$ is surjective. Proposition \ref{prop:nonzeroinf} gives $\inf_{\mathcal G} J=0$.
However, in this case the value of $\inf_{\mathcal CG} J$ is not always 0. We will give examples later.

\smallskip
Case 1.1: $\mathcal Q$ is positive definite on  $\ker B_+$  and  $\dim H\ge s^+(\mathcal Q)+\sum_{i=1}^{m^+} \dim H_i$.
This is our last case, and in a sense the only non-degenerate one. Dealing with it is the main part of the work.
We postpone the proof of the following statement to the next section, in order to discuss its consequences first.

\begin{theorem}\label{theo:gauss-mini}
If $\mathcal Q$ is positive definite on $\ker B_+$ and 
 $$ \dim H\ge s^+(\mathcal Q)+\sum_{i=1}^{m^+} \dim H_i,$$
 then $\inf J=\inf_{\mathcal{CG}} J$.
\end{theorem}

So under the above hypothesis, centered Gaussian functions allow to compute the optimal constant in 
inverse Brascamp-Lieb inequalities.
When the hypothesis of the theorem is not verified, $\inf J$ can  only be 0 or $+\infty$.
 
\begin{remark} \label{rem:dec-Q}
Assume \eqref{eq:any-decomposition-of-Q}, \eqref{eq:ker-B0-Bm1-H}
and the notation \eqref{eq:any-decomposition-of-Q-as-sefl-adj-map}. Then  
\[\mathcal Q(x)=\pi c_0\scalar{A_0 B_0 x, B_0 x} + \pi c_{m+1}\scalar{A_{m+1} B_{m+1} x, B_{m+1} x},\]  which ensures that 
\[ \mathrm{ker}(B_0,\ldots,B_{m^+})\subset \big\{ x \in H \colon \mathcal Q(x) \le 0\big\} \cap \bigcap_{i=1}^{m^+} \ker B_i.\]
 Hence $(B_0,\ldots,B_{m^+})$ is injective when $\mathcal Q$ is positive on $\ker B_+$. Together with \eqref{eq:signature-dim-H0-Hm1}, this implies
that  $ \dim H \le s^+(\mathcal Q) + \dim H_1 + \cdots + \dim H_{m^+} $. Since the hypotheses of the above theorem provide the converse inequality, they imply that  $ \dim H =s^+(\mathcal Q) + \dim H_1 + \cdots + \dim H_{m^+}$, and that $(B_0,\ldots,B_{m^+})$ is a bijection.
\end{remark}

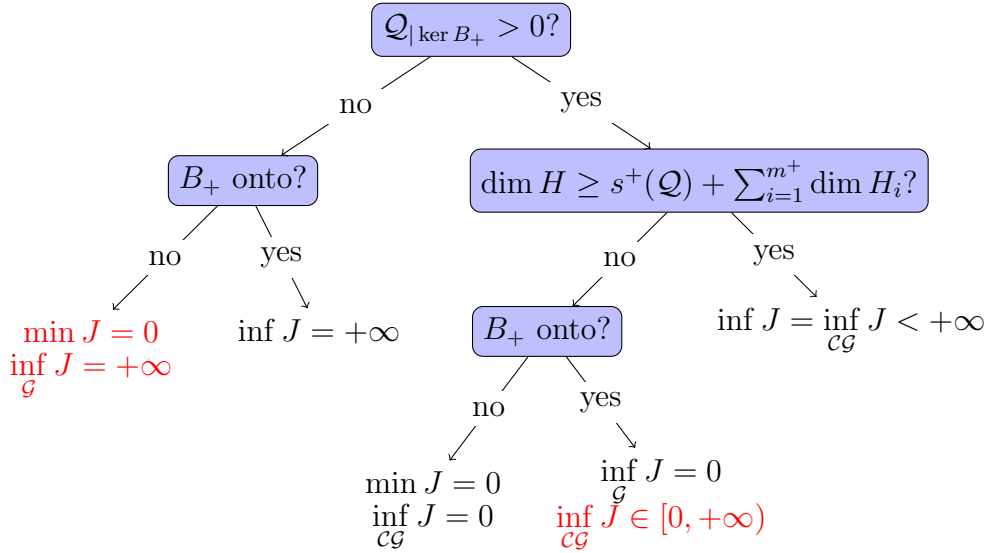
\begin{figure}

\begin{tikzpicture}
    %node styles (ie boxes)
	\tikzstyle{question}=[rectangle,draw,rounded corners=4pt,fill=blue!25]
	\tikzstyle{conclusion-}=[rectangle,text=red]
	\tikzstyle{conclusion}=[rectangle]
	%setting the nodes
	\node[question] (Q) at (0,4) {$\mathcal Q_{|\ker B_+}>0$?};
	\node[question] (B1) at (-3,2) {$B_+$ onto?};
	\node[question] (D) at (3,2) {$\dim H\ge s^+(\mathcal{Q}) +\sum_{i=1}^{m^+} \dim H_i$?};
	\node[question] (B2) at (1,0) {$B_+$ onto?};
	\node[conclusion-] (00a) at (-5,0) {  $\;\; \min J=0 \;\; $  };
	\node[conclusion-] (00b) at (-5,-0.6) {$\displaystyle \inf_{\mathcal G}J=+\infty$};
	\node[conclusion] (01) at (-2,0) {  $\inf J=+\infty$  };
	\node[conclusion] (100a) at (-0.5,-2) {  $ \min J=0$  };
	\node[conclusion] (100b) at (-0.5,-2.6) { $\displaystyle \inf_{\mathcal{CG}}J=0$  };
	\node[conclusion] (101a) at (2.5,-2) { $\displaystyle \inf_{\mathcal{G}}J=0$ };
	\node[conclusion-] (101b) at (2.5,-2.6) { $\displaystyle \inf_{\mathcal{CG}}J\in [0,+\infty)$  };
	\node[conclusion] (11) at (5,0) {$\displaystyle \inf J=\inf_{\mathcal{CG}}J<+\infty$};
	%arrow style
	\tikzstyle{suite}=[->]
	%setting the arrows
	\draw[suite] (Q) -- (B1) node[midway,fill=white]{no};
	\draw[suite] (Q) -- (D) node[midway,fill=white]{yes};
	\draw[suite] (B1) -- (00a) node[midway,fill=white]{no};
	\draw[suite] (B1) -- (01) node[midway,fill=white]{yes};
	\draw[suite] (D) -- (B2) node[midway,fill=white]{no};
	\draw[suite] (D) -- (11) node[midway,fill=white]{yes};
	\draw[suite] (B2) -- (100a) node[midway,fill=white]{no};
	\draw[suite] (B2) -- (101a) node[midway,fill=white]{yes};
\end{tikzpicture}
\caption{Summary of the case analysis}
\end{figure}

Let us mention variants of the above theorem, which consist in grouping a bit differently the various possible cases.
A first variant is Theorem \ref{thm:main-result}, as stated in the introduction. Another one is given next. It means
that under the assumption that the functional $J$ is finite for some Gaussian functions,  the optimal constant
can be computed using non-centered Gaussian functions only.

\begin{theorem}\label{theo:gauss-mini-noncentered}
If $\mathcal Q$ is positive definite on $\ker B_+$ 
 then $\inf J=\inf_{\mathcal{G}} J$.
\end{theorem}

Next we provide examples of the cases when the Gaussian minimizers principle fails.

\begin{example} 
Consider the very simple case of the functional 
$$J(f,g):=\frac{\int_{\R^2} f(x)g(x) \,dx\,dy}{ \int_{\R} f \times \int_{\R} g}\cdot$$
Here $m^+=m=2$, $c_1=c_2=1$, $\mathcal Q=0$ and 
$B_1(x,y)=B_2(x,y)=x$.
The map $B_+\colon \R^2\to \R^2$ is given by $B_+(x,y)=(x,x)$. It is not surjective, and
 $\mathcal{Q}$ is not positive definite on $\ker B_+=\{0\}\times \R$. So we are in the setting
 of Case 0.0 above.  
 
By Fubini $J(f,g)=+\infty \times \int_{\R} f(x)g(x)\, dx/ (\int f \times \int g)$ which is equal 
to 0 if the supports of $f$ and $g$ are disjoint, and is equal to $+\infty$ for Gaussian functions. 
\end{example}

\begin{example}[Reversed hypercontractivity]
Borell's reverse Gaussian hypercontractivity~\cite{borell-1982} states that for any $p,q\in(-\infty,1)$ the operators of the Ornstein-Uhlenbeck semigroup $P_t f(x) = \int_{\mathbb R} f(e^{-t} x + \sqrt{1-e^{-2t}} y) \gamma(dy)$, where $\gamma$ is a standard Gaussian measure, satisfy
\[
  \| P_t f \|_{L^q(\gamma)} \ge \|f\|_{L^p(\gamma)}
\]
for all \emph{positive} functions $f \in L^1(\gamma)$ if and only if $e^{-2t} \le \frac{1-p}{1-q}$.
%For negative indices $p,q$, the above inequality quantifies a property of positivity improvement 
%along the semigroup.
Excluding the case when either $p$, $q$ or $t$ is $0$ and using the fact that for $q\in (-\infty,1)$  and 
$h \in L^q$, $\|h\|_{L^q} = \inf \{ \int h k \colon k > 0, \int k^{q'} = 1\}$ where $q'=q/(q-1)$, the above estimate can be restated as follows: let $n=2$, $m=2$, $n_1 = n_2 = 1$, $B_1(x_1,x_2) = x_1$, $B_2(x_1,x_2)=x_2$, $c_1 = 1/p\in \R\setminus [0,1]$, $ c_2 = 1/q' \in \R\setminus [0,1]$, $t > 0$ and
\[
  Q = \frac{1}{2\pi(1-e^{-2t})} \left( \begin{array}{cc} 1 - (1-e^{-2t}) c_1 & -e^{-t} \\ -e^{-t} & 1 - (1-e^{-2t}) c_2 \end{array} \right).
\]
Then for the corresponding functional
\[
  J(f,g) = \int_{\R^2} e^{-\pi \scalar{Qx,x}} f^{c_1}(x_1) g^{c_2}(x_2) \, dx_1 dx_2 \, \Big(\int f\Big)^{-c_1} \Big(\int g\Big)^{-c_2},
\]
we have $\inf J = (2\pi)^{1 - \frac{c_1+c_2}{2}} \sqrt{1-e^{-2t}}$ if and only if $c_1 c_2 \det Q \ge 0$.

 Now, let us focus on the  a specific example $c_1 = c_2 = 2$. In this case $B_+(x_1,x_2)=(x_1,x_2)$. Hence $B_+$ is surjective
 and $Q$ is positive definite on $\ker B_+=\{0\}$.
 Condition~\eqref{eq:surjectivity-condition} is violated if and only if $s^+(Q) > 0$, which is equivalent to
 \[ \tr (Q) > 0 \quad \mathrm{or} \quad  \det( Q )< 0.\] 
 Actually, in our case $\tr (Q) > 0$ implies $\det (Q) < 0$. A simple calculation shows that  $s^+(Q) > 0$ holds if and only if $e^{-2t} > 1/4$.
 Thus, whenever $e^{-2t} > 1/4$, we are in Case 1.0.1 above, and   $\inf J = \inf_{\mathcal G} J = 0$. Besides, Borell's result asserts that $\inf J = (2\pi)^{-1} \sqrt{1-e^{-2t}}$ provided $e^{-2t} \le 1/4$. Next we claim that
\[
  \inf_{\mathcal{CG}} J = 
   \begin{cases}
   (2\pi)^{-1} \sqrt{1-e^{-2t}} &\text{if } e^{-2t} \in \big(\frac14, \frac12\big],
\\
   0        &\text{if } e^{-2t} \in \big(\frac12, 1\big).
   \end{cases}
\]
This is an illustration of Case 1.0.1 above: $\inf J=\inf_{\mathcal{G}}J=0$ but $\inf_{\mathcal{CG}}J$ can be 0 in some cases,
and positive in some other cases.

It remains to prove the claim. Put $f(x) = e^{-a x^2 /2}$ and $g(x) = e^{-b x^2 /2}$ for some $a,b>0$. Then 
\[
  J(f,g)^2 = \begin{cases}
    \frac{(2\pi)^{2-(c_1+c_2)} a^{c_1} b^{c_2} (1 - e^{-2t})^2}
           {\det \left( \begin{array}{cc} 1+(1-e^{-2t}) c_1 (a-1) & -e^{-t} \\[1ex]
                                                           -e^{-t} & 1+(1-e^{-2t}) c_2 (b-1) \end{array} \right)} & \text{if $\det > 0,$} \\
           +\infty & \text{otherwise.}
  \end{cases}
\]
Restricting our attention to the case $c_1 = c_2 =2$,
\[
  J(f,g)^2 = \begin{cases}
    (2\pi)^{-2} (1-e^{-2t}) \frac{a^2 b^2}{4(1+ab)(1-e^{-2t}) - 3 + 2(2e^{-2t}-1) (a+b)} &
      \text{if the denominator is positive} \\[1ex]
    +\infty & \text{otherwise.}
  \end{cases}
\]
In the case $e^{-2t} \in (1/2,1)$ we show that $\inf_{a,b>0} J(f,g)^2 = 0$ by checking that
\begin{align*}
 \sup_{a,b>0} \frac{4(1+ab)(1-e^{-2t}) - 3 + 2(2e^{-2t}-1) (a+b)}{a^2 b^2} \\[1ex]
  \ge 
 \sup_{a>0, b=1/a} 8(1-e^{-2t}) - 3 + 2(2e^{-2t}-1) (a+1/a) = +\infty.
\end{align*}
In the case $e^{-2t} \in (1/4, 1/2]$ we will have $\inf_{a,b>0} J(f,g) = (2\pi)^{-1} \sqrt{1-e^{-2t}}$ if we show 
\begin{equation}\label{eq:sup1}
 \sup_{a,b>0} \frac{4(1+ab)(1-e^{-2t}) - 3 + 2(2e^{-2t}-1) (a+b)}{a^2 b^2} = 1.
\end{equation}
Put $\lambda = 2 - 4e^{-2t} \in [0,1)$. Since $a+b$ is multiplied by the coefficient $2(2e^{-2t}-1) = -\lambda\le 0$, we can use the inequality $a+b \ge 2 \sqrt{ab}$ to calculate the supremum in~\eqref{eq:sup1} as follows:
\begin{align*}
  \sup_{a,b>0} \frac{\lambda-1 - \lambda(a+b) + (\lambda+2)ab}{a^2 b^2} 
  = \sup_{x = (ab)^{-1/2} > 0} (\lambda-1)x^4 - 2\lambda x^3 + (\lambda+2) x^2 =: \sup_{x>0} \varphi(x).
\end{align*}
Since
\[
  \varphi'(x) = -4(1-\lambda) x(x-1)\Big(x - \frac{\lambda + 2}{2(\lambda-1)} \Big),
\]
$\varphi$ is increasing on $(0,1]$ and decreasing on $[1,\infty)$ and hence $\sup_{x>0} \varphi(x) = \varphi(1) = 1$.
\end{example}

We conclude this section with the analysis of degenerate and non-degenerate cases for the inverse convolution inequality.

\begin{example}[Reverse Young inequality]
As mentioned in the introduction: for $p,q,r\in (0,1]$ such that $1+1/r=1/p+1/q$,  and positive functions on $\mathbb R^n$, Brascamp and Lieb have proved that 
\[ \|f*g\|_r\ge \left(\frac{C_pC_q}{C_r}\right)^n \|f\|_p \|g\|_q\] holds where $C_t=|t|^{1/t}/|t'|^{1/t'}$, and the constant is  optimal. Our goal here is to discuss extensions to negative exponents.

Using a duality type argument, a change of functions and the fact that $C_{p'}=1/C_p$ we can reformulate the above result as follows:
if $p,q\in (0,1]$ and $r'\in (-\infty,0)$ verify $1/p+1/q+1/r'=2$ then for all
positive integrable functions $f,g,h$,
\[\int_{(\mathbb R^n)^2}  f(x-y)^{\frac1p} g(y)^{\frac1q} h(x)^{\frac{1}{r'}} \,dx\,dy\ge (C_pC_qC_{r'})^n \left(\int f\right)^{\frac1p} \left(\int g\right)^{\frac1g} \left(\int h\right)^{\frac{1}{r'}}.\] 
Simple changes of variables as 
$\int F(x-y)G(y)H(x) \,dx\,dy= \int F(z)G(x-z)H(x) \,dx\,dz$ show that $p,q$ and $r'$
play symmetric roles. Therefore the convolution inequality is also true when 
$1/p+1/q+1/r'=2$ and among the three numbers $p,q,r'$, two are in $(0,1]$ and one is in $(-\infty,0)$,
(which is more general than $p,q,r\in (0,1]$).
However no non-trivial inequality holds beyond this range of indices, as we show next.

 The 
condition $1/p+1/q+1/r'=2$ is necessary (applying the inequality to $f(\lambda \cdot),
g(\lambda \cdot)$, $ h(\lambda \cdot)$ for $\lambda> 0$ and changing variables $(x,y)=\lambda^{-1}(X,Y)$ gives it).
Consider the three surjective maps from $\mathbb R^{2n}$ to $\mathbb R^n$ defined by $B_1(x,y)=x-y$, $B_2(x,y)=y$, $B_3(x,y)=x$, and the numbers
$c_1=1/p$, $c_2=1/q$ and $c_3=1/r'$.
The above analysis of degenerate cases shows that one should focus on 
the map $B_+=(B_i)_{i\colon c_i>0}$.
If $p,q,r'$ are positive then $B_+$ is not surjective and the only possible constant 
in the convolution inequality is 0. If only one among the three number $p,q,r'$ is 
positive, then $B_+$ is surjective but not injective and we are in Case 0.1, meaning
that the functional under study never takes finite values. 

\end{example}

\section{Proof of Theorem \ref{theo:gauss-mini}}\label{sec:proof-main}

\subsection{Decomposition of the kernel $\exp(-\mathcal Q)$}\label{sec:decomposition-of-Q}
%We shall first investigate the impact of the kernel $\exp(-\mathcal Q)$ on our problem. In fact, 
The positive and negative parts of a quadratic form $\mathcal Q$ play different roles, as do the functions $f_i$ with $i \le m^+$ and the functions $f_j$ with $j > m^+$. Although there is no canonical decomposition of $H$ into subspaces on which $\mathcal Q$ is, respectively, positive and negative definite, Condition~\eqref{eq:injectivity-condition} provides a natural candidate for a subspace on which $\mathcal Q$ is positive definite.
This leads to the following result:

\begin{lemma}\label{lem:decompose-Q}
The following two assertions are equivalent:
\begin{enumerate}
	\item \textup{(i)} $\mathcal Q$ is positive definite on $\ker B_+$ and \textup{(ii)} $\dim H \ge s^+(\mathcal Q)+ \sum_{i=1}^{m^+} \dim H_i $.
	\item There exist vector spaces $H_0$, $H_{m+1}$,  surjective linear maps $B_0 \colon H\to H_0$ and $B_{m+1} \colon H\to H_{m+1}$, and positive definite quadratic forms $\mathcal Q_+$ on $H_0$ and $\mathcal Q_-$ on $H_{m+1}$ such that:
	\begin{align}
		\bullet\; & (B_0,B_+) \colon H\to H_0\times\cdots \times H_{m^+} \mbox{ is bijective}, \label{eq:isomorphism-condition} \\
		\bullet\; &  \ker B_+\subset \ker B_{m+1}, \label{eq:kerBplus-contained-in-kerBm1} \\
		\bullet\; & \mbox{for all } x\in H, \quad \mathcal Q(x)=\mathcal Q_+(B_0x)-\mathcal Q_-(B_{m+1}x). \nonumber
 	\end{align}		
\end{enumerate}
\end{lemma}
\begin{remark} The above decomposition of $\mathcal Q$  is more specific than the ones introduced in Subsection~\ref{subsec:translated-gaussians} in~\eqref{eq:any-decomposition-of-Q}. Although we have used for convenience the same notation $B_0$ and $B_{m+1}$ there, they do not necessarily represent the same maps as in the above lemma. However no confusion will be possible, since from now on we will only use the decomposition of Lemma~\ref{lem:decompose-Q}.
\end{remark}

\begin{proof}
We start with $(2)\Longrightarrow (1)$:
For any $x\in \ker B_+\subset \ker B_{m+1}$, it holds
$$ \mathcal Q(x)=\mathcal Q_+(B_0x)-\mathcal Q_-(B_{m+1}x)=\mathcal Q_+(B_0x)\ge 0.$$
Moreover, if $\mathcal Q(x)=0$, using that $\mathcal Q_+$ is definite positive, we deduce that
$B_0x=0$. It follows that $x$ belongs to $\ker B_0\cap \ker B_+$, which is equal to $\{0\}$ by hypothesis.
Thus we have shown that $\mathcal Q$ is positive definite on $\ker B_+$.

It remains to prove (1)(ii). By hypothesis, $(B_0,B_+)$ is
a linear isomorphism, which implies that
$$\dim H=\sum_{i=0}^{m^+} \dim H_i.$$
Therefore, it is enough to show that $s^+(\mathcal Q)\le \dim H_0$.
Since $\mathcal Q_-$ is positive definite, $\mathcal Q$ is negative semi-definite on $\ker B_0$, and hence
\[s^+(\mathcal Q)\le \dim H-\dim \ker B_0=\dim H_0. \]

\medskip
Now we prove that $(1)\Longrightarrow(2)$. Set
\[
  H_0 = \ker B_+ = \bigcap_{i=1}^{m^+} \ker B_i.
\]
Note that (1) implies $\dim H_0 = s^+(\mathcal Q)$. Indeed,
\[
  \dim H_0 = \dim \ker B_+ = \dim H - \dim \im B_+ \ge \dim H - \sum_{i=1}^{m^+} \dim H_i \ge s^+(\mathcal Q),
\]
where the last inequality follows from~(1)(ii). The converse inequality follows from Sylvester's theorem since $\mathcal Q$ is positive definite on $H_0$.

Now consider the subspace
\[
  H_0^{\perp_{\mathcal Q}} = \{ x \in H \colon \forall {y \in H_0}, \ \mathcal Q(x,y) = 0 \},
\]
where we also denote by $\mathcal{Q}(\cdot, \cdot)$ the symmetric bilinear form associated with the quadratic form $\mathcal Q$.
Since $\mathcal Q$ is positive definite on $H_0$, we have
\[
  H_0 \cap H_0^{\perp_{\mathcal Q}} 
  %= \{ x \in H_0 \colon \forall {y \in H_0}, \ \mathcal Q(x,y) = 0 \} 
  \subseteq \{ x \in H_0 \colon \mathcal Q(x, x) = 0 \} = \{0\}.
\]
As a general fact, $\dim H_0^{\perp_{\mathcal Q}} \ge \dim H - \dim H_0$, therefore
\begin{equation}\label{eq:direct-product-H0}
  H_0 \oplus H_0^{\perp_{\mathcal Q}} = H.
\end{equation}

Consider the projection $P \colon H \to H$ onto $H_0$ with kernel $H_0^{\perp_{\mathcal Q}}$. Then $\Id - P \colon H \to H$ is the projection onto $H_0^{\perp_{\mathcal Q}}$ with $\ker (\Id - P) = H_0$ and 
\begin{equation}\label{eq:decomposition-Q-prep-step}
  \mathcal Q(x) = \mathcal Q(Px) + \mathcal Q((\Id - P)x).
\end{equation}

Next, note that $\mathcal Q$ is negative semi-definite on $H_0^{\perp_{\mathcal  Q}}$. Indeed, suppose that for some $0 \neq x \in H_0^{\perp_{\mathcal  Q}}$, $\mathcal Q(x) > 0$. Then for all $\lambda \in \R$ and $y \in H_0$,
\[
  \mathcal Q(\lambda x + y) = \lambda^2 \mathcal Q(x) + 2\lambda \mathcal Q(x,y) + \mathcal Q(y) = \lambda^2 \mathcal Q(x) + \mathcal Q(y) > 0
\]
whenever $\lambda \neq 0$ or $y \neq 0$, which thanks to~\eqref{eq:direct-product-H0} is equivalent to $\lambda x + y \neq 0$. In this way $\mathcal Q$ would be positive definite on the subspace $H_0 \oplus \textup{span}\{x\}$ which has dimension strictly larger than $s^+(\mathcal Q)$ and thus it contradicts Sylvester's theorem.

Further on, as a general fact, the \emph{radical} of $\mathcal Q$
\[
  \rad \mathcal Q = \{ x \in H \colon \forall{y \in H}, \ \mathcal Q(x,y) = 0\}
\]
is a subspace in $H_0^{\perp_{\mathcal Q}}$. Consider
\[
  H_{m+1} = \sfrac{H_0^{\perp_{\mathcal Q}}}{\rad \mathcal Q}
\]
and the maps $ B_0 \colon H \to H_0$ and $ B_{m+1} \colon H \to H_{m+1} $ defined for $x\in H$ by 
\begin{eqnarray*}
  &  & B_0(x) = P(x), \\
  & & B_{m+1}(x) = \pi_{H_0^{\perp_{\mathcal Q}} \to H_0^{\perp_ {\mathcal Q}}/{\rad \mathcal Q} } \big((\Id - P)(x)\big),
\end{eqnarray*}
where $\pi_{H_0^{\perp_{\mathcal Q}} \to {H_0^{\perp_{\mathcal Q}}}/{\rad \mathcal Q}}$ is the  natural quotient map from $H_0^{\perp_{\mathcal Q}}$ to $\sfrac{H_0^{\perp_{\mathcal Q}}}{\rad \mathcal Q}$.
Finally, consider the following positive definite quadratic forms
\[ \begin{split}
  & \mathcal Q_+ = \mathcal Q|_{H_0} \colon H_0 \to \R, \\
  & \mathcal Q_- \colon H_{m+1} \to \R, \quad \mathcal Q_-(x + \rad \mathcal Q) = -\mathcal Q(x) \textup{ for $x \in H_0^{\perp_{\mathcal Q}}$}.
\end{split} \]
Then the decomposition~\eqref{eq:decomposition-Q-prep-step} becomes
%\begin{equation}\label{eq:decomposition-Q}
\[
 \mathcal Q(x) = \mathcal Q_+(B_0 x) - \mathcal Q_-(B_{m+1} x).
\]
%\end{equation}

Next, let us establish the claimed properties of the linear maps which appear
in the  above decomposition of $\mathcal Q$. %, which are actually properties of the maps $B_0$ and $B_{m+1}$.

The non-degeneracy conditions~(1)(i) and~(1)(ii) imply that the map
\[
B_{0+} := (B_0, B_1, \ldots, B_{m^+}) \colon H \to H_0 \times H_1 \times \cdots \times H_{m^+} \textup{ is a linear isomorphism.}
\]
Indeed $\ker B_0 \cap \ker B_+=H_0^{\perp Q}\cap H_0=\{0\}$, hence $B_{0+}$ is injective. The dimension condition (1)(ii),
once rewritten as $\dim H\ge \sum_{i=0}^{m^+} \dim H_i$, shows that $B_{0+}$
is an isomorphism.

Note also that $\ker B_{m+1} = H_0 + \rad \mathcal Q$ and therefore
\[
  \ker B_+ =H_0\subseteq \ker B_{m+1}.
\]
This concludes the proof of the claimed properties. 
\end{proof}

The simple lemma stated below establishes the following consequence of~\eqref{eq:kerBplus-contained-in-kerBm1}:
\begin{equation}\label{eq:compact-image}
\textup{if $F \subseteq H_1 \times \cdots \times H_{m^+}$ is compact, then $B_{m+1}(B_+^{-1}(F))$ is compact.}
\end{equation}
\begin{lemma}\label{lem:compact-image}
Let $S \colon X \to Y$ and $T \colon X \to Z$ be linear maps between finite dimensional linear spaces $X, Y, Z$ and $F \subseteq Y$ be a compact set. If $\ker S \subseteq \ker T$ then $T(S^{-1}(F))$ is a compact subset of $Z$.
\end{lemma}
\begin{proof}
Put $V = \ker S$ and let $\pi \colon X \to \sfrac{X}{V}$ be the natural quotient map. Define $\tilde{S} \colon \sfrac{X}{V} \to Y$ and $\tilde{T} \colon \sfrac{X}{V} \to Z$ as
\[
  \tilde{S}(x + V) = S x, \qquad \tilde{T}(x + V) = T x
\]
(these definitions are correct since the kernels of $S$ and $T$ contain $V$), i.e. $\tilde S \circ \pi = S$ and $\tilde{T} \circ \pi = T$. Note that $\ker \tilde{S}$ is trivial, hence $\tilde{S}$ is a linear an isomorphism onto its range, and thus $G = \tilde{S}^{-1}(F)$ is a compact subset of $\sfrac{X}{V}$. Finally write
\[
  T(S^{-1}(F)) = \tilde{T}(\pi(\pi^{-1}(\tilde{S}^{-1}(F)))) = \tilde{T}(\tilde{S}^{-1}(F)) = \tilde{T}(G)
\]
and use that $\tilde{T}(G)$ is also compact.
\end{proof}

\subsection{More on quadratic forms}

We start with recalling a simple inequality, which appears in the transportation proof of the Brascamp-Lieb inequalities.

\begin{lemma}\label{lem:quad-direct}
Let $I$ be a finite set, and for each $i\in I$ let $d_i>0$, $L_i\colon H\to H_i$ be linear and onto 
and $K_i\colon H_i\to H_i$ be a linear symmetric definite positive map. 
Assume that $K:=\sum_i d_i L_i^*K_iL_i>0$.  Let $w\in H$, then for all $y_i\in H_i$ verifying 
$w=\sum_i d_i L_i^* y_i$, the following holds
$$\langle K^{-1}w,w \rangle \le \sum_i d_i \langle K_i^{-1} y_i,y_i\rangle.$$
There is equality if one chooses $y_i:=K_iL_iK^{-1}w$.
\end{lemma}
\begin{proof}
This is a direct application of the Cauchy-Schwarz inequality:
\begin{eqnarray*}
\langle K^{-1}w,w \rangle &=& \sum_i d_i  \langle K^{-1}w,L_i^*y_i\rangle
=   \sum_i d_i  \langle K_i^{1/2}L_i K^{-1}w,K_i^{-1/2} y_i\rangle\\
&\le & \left(  \sum_i d_i  \langle K_i^{1/2}L_i K^{-1}w,  K_i^{1/2}L_i K^{-1}w \right)^{\frac12}\left(  \sum_i d_i  \langle K_i^{-1/2} y_i,K_i^{-1/2} y_i\rangle\right)^{\frac12}\\
&=& \left(  \Big\langle (\sum_i d_i  L_i^* K_iL_i) K^{-1}w, K^{-1}w \Big\rangle\right)^{\frac12}\left(  \sum_i d_i  \langle K_i^{-1} y_i, y_i\rangle\right)^{\frac12}\\
&=&    \langle  K^{-1}w,w\rangle ^{\frac12}\Big(  \sum_i d_i  \langle K_i^{-1} y_i, y_i\rangle\Big)^{\frac12}
\end{eqnarray*}
\end{proof}

When proving Proposition \ref{prop:nonzeroinf}, we have shown that the quadratic inequality stated as its first conclusion implies that the map $(B_0,\ldots,B_{m^+})$ is onto. Our next task is to prove a converse statement, for further use.

\begin{lemma}\label{lem:quad-inverse}
Let $c_1,\ldots,c_{m^+}>0>c_{m^++1},\ldots, c_m$.  For $k=1,\ldots,m$, let
 $B_k \colon H\to H_k$ be a linear surjective map, and let $A_k \colon H_k\to H_k$ be symmetric definite positive operator.
Assume  that $(B_1,\ldots,B_{m+})\colon H\to H_1\times\cdots\times H_{m^+}$ is onto. 

If $A:=\sum_{k=1}^m c_k B_k^*A_kB_k>0$ and $y=\sum_{k=1}^m c_k B_k^* y_k$ for some $y_k\in H_k$, then 
$$\langle A^{-1}y,y \rangle \ge \sum_{k=1}^m c_k \langle A_k^{-1} y_k,y_k\rangle.$$
There is equality if one chooses $y_k:=A_kB_kA^{-1}y$.
\end{lemma}

\begin{proof}
The statement is derived from the former lemma, after  rearranging the terms.
By the surjectivity hypothesis, there exits $z\in H$ such that for all $i\le m^+$, $y_i=A_iB_iz$. 

The relationship $y=\sum_{k=1}^m c_k B_k^* y_k$ can be rewritten as $$y+\sum_{m^+<j\le m} |c_j| B_j^* y_j=
\Big(\sum_{i\le m^+} c_iB_i^*A_iB_i\Big) z.$$
If we set $K:=\sum_{i\le m^+} c_i B_i^*A_iB_i$, $w:=Kz$, $H_{m+1}:=H$, $B_{m+1}:=\Id_H$, $A_{m+1}:=A$, $y_{m+1}:=y$ and $c_{m+1}=1$, we obtain that 
\begin{equation} \label{eq:quad1}
w=\sum_{m^+<j\le m+1} |c_j| B_j^* y_j.
\end{equation}
 With this notation, we may also rewrite the relationship 
  $A=\sum_{k=1}^m c_k B_k^*A_kB_k$ as 
 \begin{equation}\label{eq:quad2}
 K=\sum_{i\le m^+} c_i B_i^*A_iB_i=A+\sum_{m^+<j\le m} |c_j| B_j^*A_jB_j
 =\sum_{m^+<j\le m+1} |c_j| B_j^*A_jB_j.
 \end{equation}
 Observe that $K>0$ holds, as a consequence of $A>0$. Therefore \eqref{eq:quad1} and  \eqref{eq:quad2} allow to apply Lemma~\ref{lem:quad-direct} and to get 
 \begin{equation}\label{eq:quad3}
\sum_{m^+<j\le m+1} |c_j| \langle A_j^{-1} y_j,y_j\rangle \ge \langle K^{-1} w,w\rangle.
\end{equation}
By our definitions for $w$,  $K$ and $z$,
 $$ \langle K^{-1} w,w\rangle= \langle K z,z\rangle=\Big\langle \sum_{i\le m^+} c_iB_i^*A_iB_iz,z\Big\rangle= \sum_{i\le m^+} c_i \langle A_iB_iz, B_i z\rangle
 =\sum_{i\le m^+} c_i \langle y_i, A_i^{-1} y_i\rangle.$$
Since $|c_{m+1}| \langle A_{m+1}^{-1} y_{m+1},y_{m+1}\rangle=\langle A^{-1}y,y\rangle$
and $c_j<0$ for $m^+<j\le m$, the statement of~\eqref{eq:quad3} gives the claimed inequality. The case of equality is easily verified.
\end{proof}

\subsection{Preliminaries and general strategy of the proof of Theorem~\ref{theo:gauss-mini}}

The very first step in the proof of Theorem~\ref{theo:gauss-mini} is to consider a decomposition of the Gaussian kernel $\exp(-\mathcal Q)$ as explained before. Namely, thanks to the hypothesis of Theorem~\ref{theo:gauss-mini} the assertion (2) from Lemma~\ref{lem:decompose-Q} holds true. Therefore we will consider the quadratic forms $\mathcal{Q}_+$ and $\mathcal{Q}_-$ together with the maps $B_0 \colon H \to H_0$ and $B_{m+1} \colon H \to H_{m+1}$ whose existence is ensured by that assertion. Further consider self-adjoint maps $Q_+ \colon H_0 \to H_0$ and $Q_- \colon H_{m+1} \to H_{m+1}$ which represent the respective quadratic forms, i.e.
\[ \begin{split}
  \mathcal{Q}_+(x) &= \pi \scalar{Q_+ x, x} \quad \textup{for $x \in H_0$,} \\
  \mathcal{Q}_-(x) &= \pi \scalar{Q_- x, x} \quad \textup{for $x \in H_{m+1}$.}
\end{split} \]

For $k = 1,2,\ldots,m$ fix measurable functions $f_k \colon H_k \to [0, \infty]$ of integral one. We will deal with the Gaussian kernel $\exp(-\mathcal Q)$ as two additional functions, namely
\[
  \exp(-\mathcal{Q}(x)) = \sqrt{\frac{\det Q_-}{\det Q_+}} f_0^{c_0}(B_0 x) f_{m+1}^{c_{m+1}}(B_{m+1} x),
\]
where $f_0 \colon H_0 \to [0, \infty]$ and $f_{m+1} \colon H_{m+1} \to [0, \infty]$ are defined as
\[ \begin{split}
  f_0(x) &= \sqrt{\det Q_+} \exp(-\pi \scalar{Q_+ x, x}), \\
  f_{m+1}(x) &= \sqrt{\det Q_-} \exp(-\pi \scalar{Q_- x, x})
\end{split} \]
and $c_0 = 1$, $c_{m+1} = -1$. With this notation we have
\begin{equation}\label{eq:J-as-product-of-f}
  J(f_1, \ldots, f_m) = \sqrt{\frac{\det Q_-}{\det Q_+}} \int_H \prod_{k=0}^{m+1} f_k^{c_k}(B_k x) \, dx.
\end{equation}

The general strategy is similar to the one of the proof of the direct and reverse Brascamp-Lieb inequality from~\cite{barthe-inventiones}. Namely we will consider a tuple of centered Gaussian functions $g_k$ on $H_k$ ($k = 0,1,\ldots,m+1$) of integral one and optimal transport maps $H_k \ni x \mapsto y = T_k(x) \in H_k$ which push forward the density $f_k(x) \, dx$ onto the density $g_k(y) \, dy$. Starting from the maps $T_k$, we will build a change of variable map $\theta \colon H \to H$ which will allow us to pass from $J(f_1, \ldots, f_k)$ in the form of~\eqref{eq:J-as-product-of-f} to the integral over $H$ involving a Gaussian function only. However, since we aim to bound~\eqref{eq:J-as-product-of-f} from below, it is crucial that the map $\theta$ is \emph{surjective}. This point is a substantial technical difficulty which is not present in the transportation proof of the Brascamp-Lieb inequality with positive exponents.

In order to make the above strategy work, we need to restrict $f_k$ to carefully chosen classes of functions. There are two reasons behind this: first, we need to ensure existence of optimal transport maps; moreover, it will be convenient to have some regularity of these maps and to have the Monge-Amp{\`e}re equation satisfied in the classical sense. 
The second reason is that we need to ensure surjectivity of the map $\theta$. This can be done by appropriate choice of supports of the test functions. Finally, the inequality for any integrable functions $f_k$ will be obtained via an approximation argument.

\bigskip

We close this subsection with notation and facts concerning convex functions on Euclidean spaces. A standard reference is~\cite{rockafellar-convex-analysis}. In the sequel we write $\interior A$, $\cl A$, $\bd A$ for the interior, closure and boundary  of $A$.

Let $\varphi \colon \R^n \to \R \cup \{+\infty\}$ be a convex function. The \emph{domain} of $\varphi$ is
\[
  \dom \varphi = \{ x \in \R^n \colon \varphi(x) < +\infty \}.
\]
We say that $\varphi$ is \emph{proper} if $\dom \varphi \neq \emptyset$. We say that $\varphi$ is \emph{closed} if the epigraph of $\varphi$, i.e. the set $\{ (x, y) \in \R^{n+1} \colon x \in \dom \varphi, y \ge \varphi(x) \}$, is a closed subset of $\R^n$. A convex function $\varphi$ is closed if and only if it is lower semi-continuous (or, equivalently, for every $\alpha \in \R$, $\{ x \in \R^n  \colon \varphi(x) \le \alpha \}$ is a closed subset of $\R^n$).

For $x \in \R^n$, the \emph{subdifferential} $\partial \varphi(x)$ is the set of all vectors $x^\ast \in \R^n$, called \emph{subgradients}, which satisfy
\[
  f(y) \ge f(x) + \scalar{x^\ast, y-x} \quad \textup{for all $y \in H$.}
\]
For $A \subseteq \R^n$,
\[
  \partial \varphi(A) = \bigcup_{x \in A} \partial \varphi(x).
\]
Note that $\partial \varphi(x) \neq \emptyset$ for all $x \in \interior \dom \varphi$ (actually for all $x$ in the relative interior of $\dom \varphi$) and that if $\varphi$ is proper then $\partial \varphi(x) = \emptyset$ for all $x \not\in \dom \varphi$. If $\varphi$ is differentiable at $x \in \R^n$ then $\partial \varphi(x)$ contains exactly one vector $\nabla \varphi(x)$. The converse statement is also true: for a convex function having a unique subgradient at a given point implies differentiability at that point (see~\cite[Theorem 25.1]{rockafellar-convex-analysis}).

If $\varphi$ is proper, we define the Legendre conjugate of $\varphi$ as
\[
  \varphi^\ast(y) = \sup_{x \in \dom \varphi} \scalar{x, y} - \varphi(x),
\]
which is a proper closed convex function on $\R^n$. If $\varphi$ is proper and closed then $(\varphi^\ast)^\ast$ coincides with $\varphi$.

If $\varphi$ is proper and closed then the multi-valued maps $\partial \varphi$ and $\partial \varphi^\ast$ are inverses of each other, i.e.
\begin{equation}\label{eq:subgrad-and-legendre}
  y \in \partial \varphi(x) \quad \iff \quad x \in \partial \varphi^\ast(y)
\end{equation}
(see~\cite[Corollary 23.5.1]{rockafellar-convex-analysis}). In particular, $y \in \partial \varphi(\R^n)$ if and only if $\partial \varphi^\ast(y) \neq \emptyset$ which readily implies that if $\varphi$ is proper and closed then
\begin{equation}\label{eq:inclusion-of-domains}
  \interior \dom \varphi^\ast \subseteq \partial \varphi(\R^n) \subseteq \dom \varphi^\ast.
\end{equation}

\subsection{Optimal transport map}
Here we present a result (formulated as Corollary~\ref{cor:transport-maps}) on existence of smooth solutions to the Monge-Amp{\`e}re equation related to certain class of optimal transport problems. Although the result is most probably well-known to specialists in the theory of the Monge-Amp{\`e}re equation, we were not able to find a reference where it is explicitly stated. For this reason we explain below how the result can be derived from well-established results in optimal transport and regularity theory of the Monge-Amp{\`e}re equation.

Let us begin our discussion with the following result of McCann~\cite{McCann-1995}, which is a refinement of an earlier result of Brenier~\cite{Brenier-1987} (see also references in~\cite{McCann-1995} for related developments):
\begin{theorem}[Brenier, McCann]\label{thm:McCann}
Let $\mu$ and $\nu$ be Borel probability measures on $\R^n$.
\begin{enumerate}
\item[(i)] There exist a Borel probability measure $\gamma$ on $\R^n \times \R^n$ whose marginals are $\mu$ and $\nu$ (namely, for any Borel set $A \in \R^n$, $\mu(A) = \gamma(A \times \R^n)$ and $\nu(A) = \gamma(\R^n \times A)$) and a closed convex function $\varphi \colon \R^n \to \R \cup \{+\infty\}$ such that the measure $\gamma$ is supported on the graph of the subdifferential of $\varphi$, i.e. the set
\[
  \{(x, y) \in \R^n \times \R^n \colon y \in \partial \varphi(x)\}.
\]
\item[(ii)] In addition to (i), if $\mu$ vanishes on Borel subsets of $\R^n$ Hausdorff dimension $n-1$, then $\varphi$ is differentiable $\mu$-a.e. and the map $\nabla \varphi$ (defined in points where $\varphi$ is differentiable) pushes $\mu$ forward to $\nu$, i.e. for every Borel subset $B \subseteq \R^n$,
\begin{equation}\label{eq:push-forward-by-gradient}
  \nu(B) = \mu\big((\nabla \varphi)^{-1}(B)\big)
\end{equation}
(and in fact, the map $(\id, \nabla \varphi) \colon \R^n \to \R^n \times \R^n$ pushes $\mu$ forward to $\gamma$). Moreover, the map $\nabla \varphi$ satisfying~\eqref{eq:push-forward-by-gradient} is uniquely determined $\mu$-a.e. among gradients of convex functions on $\R^n$.
\end{enumerate}
\end{theorem}
The map $\nabla \varphi$ from Theorem~\ref{thm:McCann}(ii) is called the \emph{Brenier map}.
\begin{remark}
\begin{enumerate}
\item[(i)] For a closed convex function the graph of its subdifferential is a closed subset of $\R^n \times \R^n$.
\item[(ii)] From Theorem~\ref{thm:McCann}(i) it follows that
\begin{equation}\label{eq:supp-mu-in-cl-dom-varphi}
  \supp(\mu) \subseteq \cl \{ x \in \R^n \colon \partial \varphi(x) \neq \emptyset \} = \cl \dom \varphi
\end{equation}
and
\begin{equation}\label{eq:supp-nu-in-cl-partial-varphi}
  \supp(\nu) \subseteq \cl \partial \varphi(\R^n).
\end{equation}
\item[(iii)] The assumption on $\mu$ in Theorem~\ref{thm:McCann}(ii) is satisfied if $\mu$ is absolutely continuous with respect to the Lebesgue measure on $\R^n$.
\item[(iv)] Using~\eqref{eq:push-forward-by-gradient} and continuity of the (sub)gradient of a convex function (see e.g.~\cite[Corollary 24.5.1]{rockafellar-convex-analysis}) one can show that
\begin{equation}\label{eq:supports-mu-nu}
  \supp(\nu) = \cl \nabla \varphi(\supp(\mu))
\end{equation}
(see e.g. the proof of Theorem 2.12 in~\cite{villani-tot} for details).
\item[(v)] By the above item (ii), the exterior of the domain of $\varphi$ has measure $\mu$ zero, and the boundary of $\dom \varphi$ (as the boundary of a convex set in $\R^n$) has Hausdorff dimension at most $n-1$. Therefore, $\mu$-a.e. differentiability of $\varphi$ follows from the result of Anderson and Klee~\cite{Anderson-Klee} which says that a convex function on $\R^n$ is differentiable everywhere in its domain except for a set of Hausdorff dimension at most $n-1$.
\end{enumerate}
\end{remark}

From now on we assume that $\mu$ is a probability measure on $\R^n$ with a density $f > 0$, and $\nu$ is a probability measure on $\R^n$ with a density $g$ which is positive in an open bounded convex set $\Omega$ and $g \equiv 0$ in $\Omega^c$. Thanks to Theorem~\ref{thm:McCann} we consider a closed convex function $\varphi$ for which $\nabla \varphi$ is the Brenier map which pushes $\mu$ forward to $\nu$. Let us discuss some basic properties of $\varphi$ in this context:
\begin{enumerate}
\item[(i)] Since $\supp(\mu) = \R^n$, it follows from~\eqref{eq:supp-mu-in-cl-dom-varphi} that $\dom \varphi = \R^n$.
\item[(ii)] By the hypothesis that $f > 0$, the Lebesgue measure on $\R^n$ is absolutely continuous with respect to $\mu$ and hence $\varphi$ is differentiable a.e. in $\R^n$.
\item[(iii)] By~\eqref{eq:supports-mu-nu} it is clear that
\begin{equation}\label{eq:nabla-varphi-in-cl-omega}
  \nabla \varphi(x) \in \cl \Omega \textup{ for all $x \in \R^n$ for which $\nabla \varphi(x)$ is defined.}
\end{equation}
Moreover, by~\eqref{eq:push-forward-by-gradient}, $\mu((\nabla \varphi)^{-1}(\bd \Omega)) = \nu(\bd \Omega) = 0$, hence the set $(\nabla \varphi)^{-1}(\bd \Omega)$ has zero Lebesgue measure. Therefore
\begin{equation}\label{eq:nabla-varphi-in-omega}
  \textup{$x$-a.e. the following holds: $\varphi$ is differentiable at $x$ and $\nabla \varphi(x) \in \Omega.$}
\end{equation}
\end{enumerate}

Thanks to the regularity theory of the Monge-Amp{\`e}re equation it is known that under some additional assumptions on the densities $f$ and $g$, the Brenier map $\nabla \varphi$ is defined everywhere on $\R^n$ and is a $\mathcal{C}^1$ diffeomorphism onto $\Omega$. In such case, the change of variable formula justifies that~\eqref{eq:push-forward-by-gradient} is equivalent to the fact that $\varphi$ is a solution to the Monge-Amp{\`e}re equation
\begin{equation}\label{eq:monge-ampere-general}
  \det \Hess \varphi(x) = \frac{f(x)}{g(\nabla \varphi(x))}.
\end{equation}
To this end we follow the argument of Caffarelli as presented in the paper~\cite{alesker-dar-milman}.

First note that due to~\eqref{eq:nabla-varphi-in-omega} the right hand side of~\eqref{eq:monge-ampere-general} is defined $x$-a.e. Second, we use the result of Caffarelli~\cite{caffarelli-1992} (see also Theorems 4.8 and 4.10 in~\cite{villani-tot}): since $\mu$ and $\nu$ are absolutely continuous with respect to the Lebesgue measure and the support of $\nu$ is convex, we have that $\varphi$ satisfies~\eqref{eq:monge-ampere-general} in the Aleksandrov sense, i.e. the \emph{Hessian measure} $ {\det}_H \Hess \varphi$  associated to $\varphi$, defined by
\[
  {\det}_H \Hess \varphi(A) = \textup{vol}_n(\partial \varphi(A)) \quad \textup{for any Borel set $A \subseteq \R^n,$}
\]
is absolutely continuous with respect to the Lebesgue measure and its density coincides almost everywhere with the right-hand side of~\eqref{eq:monge-ampere-general}. For a proof of this result the following relation is crucial:
\begin{equation}\label{eq:partial-varphi-cl-omega}
  \partial \varphi(\R^n) \subseteq \cl \Omega.
\end{equation}
To see~\eqref{eq:partial-varphi-cl-omega}, use~\eqref{eq:nabla-varphi-in-cl-omega} and a general result on a subdifferential of a (closed) convex function~\cite[Theorem 25.6]{rockafellar-convex-analysis} which in our case says that for any $x \in \R^n$, $\partial \varphi(x)$ lies in the closure of the convex hull of limits of $\nabla \varphi(x_k)$, where $(x_k)$ runs through all sequences of points of differentiability of $\varphi$ which converge to $x$. Since $\cl \Omega$ is already convex and closed, $\partial \varphi(x) \subseteq \cl \Omega$ for any $x \in \R^n$.

Further on, we assume additionally that the functions $f$ and $1/f$ are bounded on compact subsets of $\R^n$ and $g$ is bounded and bounded away from zero on $\Omega$. Then for any $R > 0$ there exists $0 < c(R) < C(R) < \infty$ such that the right hand side of~\eqref{eq:monge-ampere-general} is between $c(R)$ and $C(R)$ almost everywhere in the ball $B(0,R)$. Since $\varphi$ satisfies~\eqref{eq:monge-ampere-general} in the Aleksandrov sense, it follows that ${\det}_H \Hess \varphi$ has a density which is bounded and bounded away from zero on compact sets. With these a priori bounds on ${\det}_H \Hess \varphi$ we apply a geometric lemma of Caffarelli~\cite[Theorem 1]{caffarelli-localization-1990} (see also~\cite[Chapter 5]{guttierez-book}) which will allow us to prove that $\varphi$ is strictly convex.
\begin{lemma}[{Caffarelli~\cite[Theorem 1]{caffarelli-localization-1990}}]\label{lem:caffarelli-strict-convexity}
Let $\Gamma \subseteq \R^n$ be an open bounded convex set and $\psi \colon \Gamma \to \R$ be a non-negative convex function. Suppose for some constants $0 < c < C < \infty$,
\[
  c \,\textup{vol}_n(A) \le \textup{vol}_n(\partial \psi(A)) \le C \,\textup{vol}_n(A) \quad \textup{for any Borel set $A \subseteq \Gamma.$}
\]
If the (convex) set $\{x \in \Gamma \colon \psi(x) = 0 \}$ is non-empty and contains more than one point, then it has no extremal points.
\end{lemma}
\begin{corollary}
$\varphi$ is strictly convex.
\end{corollary}
\begin{proof}
Assume the opposite. Then there exists $x_0 \in \R^n$ and a supporting hyperplane $l$ of $\varphi$ at $x_0$ such the closed convex set that $F = \{ x \in \R^n \colon \varphi(x) = l(x) \}$ contains some other point $z \neq x_0$.

First suppose that $F$ contains an extreme point, say $x_1$ (not necessarily distinct from $x_0$), and take $\Gamma$ to be an (open) ball $B(0,R)$ large enough to contain $x_0, z$ and the extreme point $x_1$. Let $\psi$ be the non-negative convex function $\varphi - l$ restricted to $\Gamma$. By translation invariance of the Lebesgue measure, $\textup{vol}_n(\partial \psi(A)) = \textup{vol}_n(\partial \varphi(A))$ for any Borel set $A \subset \Gamma$ and thus we can apply Lemma~\ref{lem:caffarelli-strict-convexity}. The set $\{ x \in \Gamma \colon \psi(x) = 0 \}$ coincides with $F \cap \Gamma$ and since $F \cap \Gamma$ contains two distinct points $x_0$ and $z$, it follows from the lemma that $x_1$ cannot be an extreme point of $F \cap \Gamma$. This clearly contradicts the fact that $x_1$ is an extreme point of $F$.

Therefore $F$ is a non-empty closed convex set which has no extreme points. Hence it must contain a line (see e.g.~\cite[Corollary 18.5.3]{rockafellar-convex-analysis}). In consequence, the graph of $\varphi$ must contain a line and hence all hyperplanes supporting $\varphi$ are parallel to that line. This means that $\partial \varphi(\R^n)$ is contained in an affine subset of $\R^n$ of dimension at most $n-1$, which clearly contradicts~\eqref{eq:supp-nu-in-cl-partial-varphi}.
\end{proof}

Having established strict convexity of $\varphi$ we can conclude with a stronger statement than~\eqref{eq:partial-varphi-cl-omega}, namely that $\partial \varphi(\R^n) = \Omega$. It is based on a convexity argument. 
\begin{lemma}
Let $\psi$ be a convex function on $\R^n$ with $\dom \psi = \R^n$. If $\psi$ is strictly convex then $\partial \psi(\R^n) = \interior \dom \psi^\ast$.
\end{lemma}
\begin{proof}
In the view of~\eqref{eq:inclusion-of-domains} it is enough to prove that $\partial \psi(\R^n)$ is disjoint from $\bd \dom \psi^\ast$. To this end suppose $x \in \R^n$ and $y \in \partial \psi(x) \cap \bd \dom \psi^\ast$. Since $y \in \bd \dom \psi^\ast$, by convexity of $\dom \psi^\ast$ there exists $u \in \R^n$ such that
\begin{equation}\label{eq:outer-normal-to-dom-varphi-ast}
  \scalar{u, v - y} \le 0 \quad \textup{for all $v \in \dom \psi^\ast.$}
\end{equation}

On the other hand, consider $z = x + \lambda u$ with any $\lambda > 0$. By strict convexity of $\psi$,
\[ \begin{split}
  \psi(z) - \psi(x) &> \scalar{z - x, y} \\
  \psi(x) - \psi(z) &> \scalar{x - z, v},
  \end{split}
\]
where $v$ is an arbitrary vector from $\partial \psi(z)$. Adding up the two above inequalities yields $\scalar{z - x, y - v} = \lambda \scalar{u, y - v} < 0$ which means that $\scalar{u, v - y} > 0$. Since by~\eqref{eq:inclusion-of-domains} the vector $v \in \partial \psi(z)$ belongs to $\dom \psi^\ast$, it contradicts~\eqref{eq:outer-normal-to-dom-varphi-ast}.
\end{proof}
Applying the above lemma to $\varphi$ we get that $\partial \varphi(\R^n)$ is an open convex subset of $\R^n$. The fact that $\partial \varphi(\R^n)$ is open combined with~\eqref{eq:partial-varphi-cl-omega} implies that $\partial \varphi(\R^n) \subseteq \interior \cl \Omega = \Omega$.
On the other hand, since $\partial \varphi(\R^n)$ is open and convex, then \eqref{eq:supp-nu-in-cl-partial-varphi} yields $\Omega \subseteq \interior \cl \partial \varphi(\R^n) = \partial \varphi(\R^n)$. Therefore the three open convex sets
\begin{equation}\label{eq:partial-varphi-equals-int-dom-varphi-ast-equals-omega}
  \partial \varphi(\R^n) = \interior \dom \varphi^\ast = \Omega
\end{equation}
coincide.

Finally, we use the following
\begin{theorem}[Caffarelli {\cite[Theorem 1.3]{alesker-dar-milman}}]\label{thm:caffarelli}
Let $\mu(dx)=f(x) dx$ and $\nu(dx)=g(x)dx$ be two probability measures on $\R^n$.
Assume that $f$ is locally H{\"o}lder and strictly positive on $\R^n$. Assume that the restriction of  $g$ to an open bounded convex set $\Omega$  is locally H{\"o}lder, bounded and bounded away from zero, and that $g \equiv 0$ in $\Omega^c$. Then any convex function $\varphi$ on $\R^n$ that induces the Brenier map $\nabla \varphi$ which pushes $\mu$ forward to $\nu$ belongs locally to the H{\"o}lder class $\mathcal{C}^{2,\alpha}$ for some $\alpha > 0$ and satisfies~\eqref{eq:monge-ampere-general} for all $x \in \R^n$.
\end{theorem}

We will also need a $\mathcal{C}^2$ convex function whose gradient pushes forward $\nu$ to $\mu$. Clearly, a natural candidate is $\varphi^\ast$. In the corollary below we state the final result we will use in the sequel.
\begin{corollary}\label{cor:transport-maps}
Assume $f$ and $g$ are as in Theorem~\ref{thm:caffarelli}.
\begin{enumerate}
\item[(i)]
There exists a strictly convex function $\varphi \in \mathcal{C}^2(\R^n)$ with $\Hess \varphi$ positive definite everywhere whose gradient $\nabla \varphi$ maps $\R^n$ onto $\Omega$, pushes $\mu$ forward to $\nu$ and thus satisfies the Monge-Amp{\`e}re equation~\eqref{eq:monge-ampere-general}.
\item[(ii)] For $\varphi$ as in (i), the Legendre conjugate $\varphi^\ast$ has $\interior \dom \varphi^\ast = \Omega$, belongs to $\mathcal{C}^2(\Omega)$, $\partial \varphi^\ast(y) = \emptyset$ for all $y \not\in \Omega$. Moreover $\nabla \varphi^\ast$ pushes $\nu$ forward to $\mu$, $\Hess \varphi^\ast$ is positive definite everywhere in $\Omega$ and $\varphi^\ast$ satisfies
\begin{equation}\label{eq:monge-ampere-general-inversed}
\det \Hess \varphi^\ast(y) = \frac{g(y)}{f(\nabla \varphi^\ast(y))} \quad \textup{for all $y \in \Omega.$}
\end{equation}
\end{enumerate}
\end{corollary}
\begin{proof}
Consider $\varphi$ as in Theorem~\ref{thm:caffarelli}. In order to complete the proof of (i) it is enough to note that \eqref{eq:partial-varphi-equals-int-dom-varphi-ast-equals-omega} becomes $\nabla \varphi(\R^n) = \interior \dom \varphi^\ast = \Omega$.

(ii) Strict convexity of $\varphi$ allows us to conclude that $\nabla \varphi$ is a $\mathcal{C}^1$ bijection from $\R^n$ onto $\Omega$. Combining it with~\eqref{eq:subgrad-and-legendre} shows that $\partial \varphi^\ast(y) = \{ (\nabla \varphi)^{-1}(y) \}$ for $y \in \Omega$ and $\partial \varphi^\ast(y) = \emptyset$ for $y \not\in \Omega$. The uniqueness of subgradient of $\varphi^\ast$ at each point of $\Omega$ implies that $\varphi^\ast$ is differentiable everywhere in $\interior \dom \varphi^\ast = \Omega$ and the map $\nabla \varphi^\ast \colon \Omega \to \R^n$ is the inverse map of $\nabla \varphi$. This is already sufficient to justify that $\nabla \varphi^\ast$ pushes $\nu$ forward to $\mu$. In order to show that $\varphi^\ast$ is in fact $\mathcal{C}^2(\Omega)$ and satisfies~\eqref{eq:monge-ampere-general-inversed}, it is enough to use that the Jacobian of the map $\nabla \varphi$ (i.e. $\det \Hess \varphi$) does not vanish and thus use the inverse function theorem to obtain that the map $\nabla \varphi^\ast$ is $\mathcal{C}^1(\Omega)$ and its derivative $\Hess \varphi^\ast(y)$ equals $(\Hess \varphi(x))^{-1}$ where $y \in \Omega$ and $x = \nabla \varphi^\ast(y) \in \R^n$. Thus~\eqref{eq:monge-ampere-general-inversed} follows from~\eqref{eq:monge-ampere-general}.
\end{proof}

\subsection{Classes of test functions}
For each $k=1,2,\ldots,m$ fix a measurable function $f_k \colon H_k \to [0, \infty]$ of integral one such that
\begin{itemize}
  \item for $1 \le i \le m^+$, $f_i$ is locally Lipschitz, bounded and bounded away from zero on some bounded open convex subset of $H_i$, and vanishes outside this set;
  \item for $m^+ < j \le m$, $f_j$ is locally Lipschitz and strictly positive in its whole domain $H_j$.
\end{itemize}
The target functions are chosen as follows. Fix $R > 0$ and consider any $m$-tuple $(A_1, \ldots, A_m) \in \Lambda$ (for the definition of $\Lambda$ refer to  Subsection~\ref{subsec:centered-gaussians}) and also put $A_0 = Q_+$ and $A_{m+1} = Q_-$. For $k = 0,1,\ldots, m+1$ define
\[
  g_k(y) = (\det A_k)^{-1/2} \exp(-\pi \scalar{-A_k^{-1} y, y}).
\]
The target functions will be $\tilde{g}_k$ defined as
\begin{align*}
  \tilde{g}_i &= g_i \quad \textup{for $1 \le i \le m^+,$} \\
  \tilde{g}_0 &= g_0, \\
  \tilde{g}_{m+1} &= g_{m+1}, \\
  \tilde{g}_j &= \lambda_j g_j \Ind{B_{H_j}(0,R)} \quad \textup{for $m^+ < j \le m,$}
\end{align*}
where $\lambda_j > 1$ is a normalizing constant (such that $\int_{H_j} \tilde{g}_j = 1$).

\subsection{Transportation argument}

For each $i=1,2,\ldots,m^+$ let $\varphi_i \colon H_i \to \R \cup \{ +\infty \}$ be the function $\varphi^\ast$ from Corollary~\ref{cor:transport-maps}(ii) for probability measures $\mu$ and $\nu$ on $H_i$ having the densities $\tilde{g}_i$ and $f_i$ respectively. Each $\varphi_i$ belongs to $\mathcal{C}^2(\interior \dom \varphi_i)$, is lower semi-continuous (as the Legendre transform of a convex function) and satisfies
\begin{equation}\label{eq:vp-i-has-empty-subdifferential-outside-int-dom}
  \partial \vp_i(x) = \emptyset \quad \textup{for all $x \not\in \interior \dom \vp_i$.}
\end{equation}
For each $j=1+m^+, \ldots, m$ let $\varphi_j \colon H_j \to \R$ be the function $\varphi$ from Corollary~\ref{cor:transport-maps}(i) for probability measures $\mu$ and $\nu$ of $H_j$ having the densities $f_j$ and $\tilde{g}_j$ respectively. Each $\varphi_j$ is $\mathcal{C}^2(H_j)$ and
\begin{equation}\label{eq:vp-j-to-B-0-R}
  \nabla \varphi_j(x) \in B_{H_j}(0,R) \quad \textup{for all $x \in H_j.$}
\end{equation}
Additionally put $\varphi_0(x) = \frac12 \scalar{Q_+ x, x}$ and $\varphi_{m+1}(x) = \frac12 \scalar{Q_- x, x}$.

For $k=0,1,2,\ldots, m+1$, put $T_k = \nabla \varphi_k$ and note that by Corollary~\ref{cor:transport-maps}, for all $x \in \interior \dom \varphi_k$,
\begin{equation}\label{eq:monge-ampere}
  f_k(x) = \tilde{g}_k(T_k(x)) \det dT_k(x)
\end{equation}
and
\begin{equation}\label{eq:dTk-symmetric-positive-definite}
  dT_k(x) = \Hess \vp_k(x) \textup{ is symmetric positive definite}
\end{equation}
for all $x \in \interior \dom \vp_k$.

For $x \in H$ put
\begin{align}
  \varphi_+(x) &= \sum_{1\le i\le m^+} c_i \varphi_i(B_i x),  \nonumber\\
  \varphi_-(x) &= \sum_{m^+ < j \le m} (-c_j) \varphi_j(B_j x), \nonumber \\
  \varphi(x) &= \varphi_0(B_0 x) + \varphi_+(x) - \varphi_-(x) - \varphi_{m+1}(B_{m+1} x)=\sum_{k=0}^{m+1} c_k \varphi_k(B_k x). \label{def:phi}
\end{align} 

On the open domain $S = \bigcap_{i=1}^{m^+} B_i^{-1}(\interior \dom \vp_i) \subset H$, which is non-empty thanks to surjectivity of $B_+$, define the change of variable map $\theta \colon S \to H$ by
\begin{align*}
  \theta(x) = \nabla \vp(x) = \sum_{k=0}^{m+1} c_k B_k^\ast T_k(B_k x).
\end{align*}
This map is $\mathcal{C}^1$ and its differential equals
\begin{equation}\label{eq:dtheta}
\begin{split}
  d\theta(x) &= \Hess \vp(x) = \sum_{k=0}^{m+1} c_k B_k^\ast dT_k(B_k x) B_k \\
  &= B_0^\ast Q_+ B_0 - B_{m+1}^\ast Q_- B_{m+1} + \sum_{k=1}^{m} c_k B_k^\ast dT_k(B_k x) B_k \\
  &= Q + \sum_{k=1}^{m} c_k B_k^\ast dT_k(B_k x) B_k.
\end{split}
\end{equation}
Combining the above with~\eqref{def:BL-constant} and~\eqref{eq:dTk-symmetric-positive-definite} we obtain that for $x \in S$,
\begin{equation}\label{eq:det-dtheta-bound}
  \det d\theta(x) \le D \prod_{k=1}^m \big(\det dT_k(B_k)\big)^{c_k}
\end{equation}
whenever $d\theta(x)$ is positive definite. Since~\eqref{eq:det-dtheta-bound} remains true for $d\theta(x)$ being positive semi-definite, we will consider the subdomain of $S$,
\[
  S_+ = \{ x \in S \colon d\theta(x) \textup{ is positive semi-definite} \}
\]
on which~\eqref{eq:det-dtheta-bound} is valid. Note that by continuity of the map $S \ni x \mapsto d\theta(x)$, $S_+$ is a closed subset of $S$ and in particular $S_+$ is a measurable subset of $H$.
%[One can define $S_+'$ as an open set $\{ x \in S \colon d\theta(y) > 0 \}$ and apply Sard's theorem to say that $\theta(S_+' \setminus S_+)$ is of zero measure, so does not affect the change of variables below.]

As announced above, the following lemma is crucial in  our argument. We defer its proof to  Subsection~\ref{subsec:surjectivity-of-theta}.
\begin{lemma}\label{lem:theta-is-surjective}
The map $\theta_{| S_+} \colon S_+ \to H$ is surjective.
\end{lemma}

Now we are in the position to establish a sharp lower bound on $J(f_1, \ldots, f_m)$.
Starting from~\eqref{eq:J-as-product-of-f} and using the Monge-Amp\`ere equations \eqref{eq:monge-ampere} we get
\begin{align*}
  J(f_1, \ldots, f_m) &\ge \sqrt{\frac{\det Q_-}{\det Q_+}} \int_{S_+} \prod_{k=0}^{m+1} f_k^{c_k}(B_k x) \, dx 
\\
&= \sqrt{\frac{\det Q_-}{\det Q_+}} \int_{S_+} \prod_{k=0}^{m+1} \big( \tilde{g}_k(T_k(B_k x)) \det dT_k(B_k x))\big)^{c_k} \, dx \\
&= \sqrt{\frac{\det Q_+}{\det Q_-}} \int_{S_+} \left( \prod_{k=0}^{m+1} \tilde{g}_k^{c_k}(T_k(B_k x))\right) \left(\prod_{k=1}^m (\det dT_k(B_k x))^{c_k} \right) \, dx \\
     &\ge D^{-1} \sqrt{\frac{\det Q_+}{\det Q_-}} \int_{S_+} \left(\prod_{k=0}^{m+1} \tilde{g}_k^{c_k}(T_k(B_k x))\right) \det d\theta(x) \, dx \\
       &= (*),
\end{align*}
where the latter inequality comes from~\eqref{eq:det-dtheta-bound}.
Setting $\lambda = \prod_{j=m^+ + 1}^m \lambda_j^{c_j}$ and using the point-wise estimate $\tilde{g}_j^{c_j} \ge \lambda_j^{c_j} g_j^{c_j}$ for $m^+ < j \le m$ we continue with the bound
\begin{align}
  (*) &\ge 
  \lambda D^{-1} \sqrt{\frac{\det Q_+}{\det Q_-}} 
  \int_{S_+} \left(\prod_{k=0}^{m+1} g_k^{c_k}(T_k(B_k x))\right) \det d\theta(x) \, dx  \nonumber\\
  &\ge \lambda D^{-1} \sqrt{\frac{\det Q_+}{\det Q_-}} \int_{S_+} \left(\inf_{\theta(x) = \sum_{k=0}^{m+1} c_k B_k^\ast y_k} \prod_{k=0}^{m+1} g_k^{c_k}(y_k)\right) \det d\theta(x) \, dx  \nonumber\\
  &\ge \lambda D^{-1} \sqrt{\frac{\det Q_+}{\det Q_-}} \int_H \inf_{z = \sum_{k=0}^{m+1} c_k B_k^\ast y_k} \prod_{k=0}^{m+1} g_k^{c_k}(y_k) \, dz  \label{eq:dual-inverse-proof}\\
  &= \lambda D^{-1} \left( \prod_{k=1}^m (\det A_k)^{-c_k/2} \right)
  \int_H \exp\left(-\pi \sup_{z = \sum_{k=0}^{m+1} c_k B_k^\ast y_k} \sum_{k=0}^{m+1} c_k \scalar{A_k^{-1} y_k, y_k}\right) \, dz\nonumber\\
  &= (**), \nonumber
\end{align}
where $\lambda = \prod_{j=m^+ + 1}^m \lambda_j^{c_j}$ and the last inequality follows from the area formula for $\mathcal{C}^1$ maps \cite[Theorem 3.2.5]{federer}
and the fact that the map $\theta|S_+ \colon S_+ \to H$ is surjective (Lemma~\ref{lem:theta-is-surjective}). Using Lemma~\ref{lem:quad-inverse}, we finish the above estimate with
\begin{align*}
  (**) &=
  \lambda D^{-1} \left(\prod_{k=1}^m (\det A_k)^{-c_k/2} \right) \int_H \exp\left(-\pi \biggscalar{\Big(\sum_{k=0}^{m+1} c_k B_k^\ast A_k B_k \Big)^{-1} z, z} \right) \, dz \\
  &= \lambda D^{-1} \left( \frac{\det\big(Q + \sum_{k=1}^{m} c_k B_k^\ast A_k B_k \big)}{\prod_{k=1}^m (\det A_k)^{c_k}} \right)^{1/2}.
\end{align*}
Thanks to \eqref{def:BL-constant}, taking supremum over $(A_k)_{k=1}^m \in \Lambda$ and the limit $R \to \infty$ (which results in $\lambda \uparrow 1$) yields the desired inequality, i.e. $J(f_1, \ldots, f_m) \ge D^{-1/2}$.

\subsection{Surjectivity of the map $\theta$}\label{subsec:surjectivity-of-theta}

\begin{lemma}\label{lem:difference-of-convex-attains-infimum}
Let $f \colon H \to \R \cup \{+\infty\}$ be convex, lower semi-continuous and $g \colon H \to \R$ be convex. Assume that $\dom f \neq \emptyset$ and $\partial f(x) = \emptyset$ at every $x \in \bd \dom f$. If $f(x) - g(x) \to + \infty$ as $|x| \to +\infty$ then $f-g$ attains its infimum at a point in $\interior \dom f$.
\end{lemma}
\begin{proof}
Note that $f-g$ is lower semi-continuous which combined with the hypothesis that $f-g \to +\infty$ at infinity implies that the sets $A_r = \{ x \in H \colon f(x)-g(x) \le r\}$ are compact for all $r \in \R$. Since $\dom f \neq \emptyset$, $A_r$ is non-empty for $r$ large enough. Therefore $f-g$ attains its infimum at some point $x \in \dom f$. Suppose $x \in \bd \dom f$. Take any $x^\ast \in \partial g(x)$ (note that the subdifferential of $g$ is everywhere non-empty). By hypothesis, $\partial f(x) = \emptyset$ hence we can find $y \in H$ such that
\begin{align}\label{ineq:empty-differential}
  f(y) < f(x) + \scalar{x^\ast, y-x}.
\end{align}
On the other hand we have
\[
  g(y) \ge g(x) + \scalar{x^\ast, y-x},
\]
which combined with~\eqref{ineq:empty-differential} gives
\[
  f(y) - g(y) < f(x) - g(x) = \inf (f-g)
\]
and hence contradicts the assumption that $x \in \bd \dom f$.
\end{proof}

We will need one more lemma, about the function $\varphi$, defined in \eqref{def:phi}.
\begin{lemma}\label{lem:vp-superlinear}
The function $\vp$ is superlinear, i.e.
\[
  \lim_{|x| \to \infty} \frac{\vp(x)}{|x|} = +\infty.
\]
\end{lemma}
\begin{proof}
Consider the compact set
\[
  F = \prod_{i=1}^{m^+} \cl \dom \vp_i.
\]
Obviously
\[
  \dom \vp_+ = \bigcap_{i=1}^{m^+} B_i^{-1}(\dom \vp_i) \subseteq B_+^{-1}(F).
\]
Using~\eqref{eq:compact-image} we get
\begin{equation}\label{eq:vp-mplus1-bounded}
  \sup_{x \in \dom \vp_+} \vp_{m+1}(B_{m+1} x) = C_1 < \infty.
\end{equation}
For each $1 \le i \le m^+$, since $\dom \vp_i$ is bounded, we have $\inf_{H_i} \vp_i > -\infty$ and thus $\inf_{H} \vp_+ > -\infty$. Therefore for some constant $C_2 < \infty$ we have
\begin{equation}\label{eq:vp-plus-at-least-quadratic}
  \vp_+(x) \ge |B_+ x|^2 - C_2
\end{equation}
for all $x \in H$ (it is enough to ensure this inequality on the set $B_+^{-1}(F)$, inside which $|B_+ x|^2 - C_2$ has finite supremum).

Combining~\eqref{eq:vp-plus-at-least-quadratic} with the fact that $\vp_0$ is a positive definite quadratic function on $H_0$, we get
\[
  \vp_0(B_0 x) + \vp_+(x) \ge \varepsilon |B_0 x|^2 + |B_+ x|^2 - C_2 \ge \varepsilon' |x|^2 - C_2,
\]
for some $\varepsilon, \varepsilon' > 0$, where in the last inequality we used injectivity of the map $B_{0+} = (B_0, B_+)$. Combining the above estimate with~\eqref{eq:vp-mplus1-bounded} we get
\[
  \varphi(x)+\varphi_-(x)=\varphi_0(B_0 x) + \varphi_+(x) - \varphi_{m+1}(B_{m+1} x) \ge \varepsilon' |x|^2 - C_1 - C_2,
\]
hence the function $\vp + \vp_-$ is superlinear.

Due to~\eqref{eq:vp-j-to-B-0-R} the function $\vp_-$ is Lipschitz. Therefore $\vp$ is also superlinear.
\end{proof}

Now we are ready to establish our claim about the map $\theta$.
\begin{proof}[Proof of Lemma~\ref{lem:theta-is-surjective}]
Consider the function $f \colon H \to \R \cup \{+\infty\}$ defined by
\begin{equation}\label{eq:def-function-f}
  f(x) = \vp_0(B_0 x) + \vp_+(x)=\sum_{i=0}^{m^+} c_i \varphi_i(B_ix).
\end{equation}
Clearly $f$ is convex and lower semi-continuous. Note also that
\[
  \dom f = \bigcap_{1 \le i \le m^+} \dom(\vp_i \circ B_i) = \bigcap_{1 \le i \le m^+} B_i^{-1}(\dom \vp_i)
\]
and
\[
  \interior \dom f = \bigcap_{1 \le i \le m^+} B_i^{-1}(\interior \dom \vp_i),
\]
i.e. $\interior \dom f$ coincides with the domain $S$.

Using Theorems 23.8 and 23.9 from~\cite{rockafellar-convex-analysis} we express the subdifferential of $f$ in terms of subdifferentials of $\vp_0$ and $\vp_i$ with $1 \le i \le m^+$. Namely for all $x \in H$ we have
\[
  \partial f(x) = \sum_{0 \le i \le m^+} c_i B_i^\ast \partial \vp_i(B_i x)
\]
where the summation means the Minkowski sum of sets (in $H$). (The above formula is an equality rather than merely an inclusion $\supset$ if e.g. the set $\bigcap_{i=0}^{m^+} \dom (\vp_i \circ B_i) \neq \emptyset$; this follows from the fact that the domain $S$ is non-empty).
Combining the above with~\eqref{eq:vp-i-has-empty-subdifferential-outside-int-dom} we obtain that if $x \not\in S$ (i.e. for some $1 \le i \le m^+$, $B_i x \not\in \interior \dom \vp_i$) then $\partial f(x) = \emptyset$.

Fix any $y_0 \in H$. We claim that the $\mathcal{C}^2$ function $S \ni x \mapsto \vp(x) - \scalar{x, y_0}$ attains a local minimum at, say, $x_0 \in S$. This will allow to establish the lemma. Indeed, since $S$ is an open set, the gradient of this function vanishes at $x_0$, i.e. $\nabla \vp(x_0) - y_0 = 0$ and $\Hess \vp(x_0)$ is positive semi-definite, which means that $\theta(x_0) = y_0$ and $x_0 \in S_+$.

Eventually, let us prove the claim that $\vp(\cdot) - \scalar{\cdot, y_0} \colon S \to H$ attains its infimum. Beside the function $f$ defined in~\eqref{eq:def-function-f}, consider a convex function $g \colon H \to \R$,
\[
  g(x) = \vp_-(x) + \vp_{m+1}(B_{m+1} x) + \scalar{x, y_0}.
\]
Obviously 
\begin{equation}\label{eq:f-minus-g-as-vp-and-scalar-prod}
  f(x) - g(x) = \vp(x) - \scalar{x, y_0}.
\end{equation}
By Lemma~\ref{lem:vp-superlinear} and the fact that $\scalar{\cdot, y_0}$ is Lipschitz we obtain that $f - g$ is superlinear at infinity, so in particular $f(x) - g(x) \to \infty$ as $|x| \to \infty$. Since the subdifferential of $f$ is empty outside $S = \interior \dom f$, we can use Lemma~\ref{lem:difference-of-convex-attains-infimum} in order to conclude.
\end{proof}

\subsection{Approximation argument}

For non-negative, integrable functions $f_i$ ($1 \le i \le m^+$) and $f_j$ ($m^+ < j \le m$) denote
\[
  \mathcal{I}((f_i), (f_j))(x) = e^{-\mathcal{Q_+}(B_0 x)} e^{\mathcal{Q_-}(B_{m+1} x)} \prod_{i=1}^{m^+} f_i^{c_i}(B_i x) \prod_{j=1+m^+}^m f_j^{c_j}(B_j x) \quad\textup{for $x \in H$.}
\]
We proved that under the hypothesis of Theorem~\ref{theo:gauss-mini},
\begin{equation}\label{eq:gauss-mini-for-class-of-functions}
  \int_H \mathcal{I}((f_i), (f_j)) \ge K \prod_{1\le i\le m^+} \left(\int_{H_i} f_i\right)^{c_i} \prod_{m^+<j\le m} \left(\int_{H_j} f_j\right)^{c_j},
\end{equation}
for all $f_i \in \mathcal{F}_i^0$ ($1 \le i \le m^+$) and for all $f_j \in \mathcal{F}_j^0$ ($m^+ < j \le m$), where $K = \inf_{\mathcal{CG}} J$ and
\begin{itemize}
\item $\mathcal{F}_i^0$ is the class of non-negative functions on $H_i$ which are locally Lipschitz, bounded and bounded away from zero on an open bounded convex subset of $H_i$, and vanish outside this set,
\item $\mathcal{F}_j^0$ is the class of strictly positive and locally Lipschitz functions on  $H_j$. 
%and on every compact subset of $H_j$ they are bounded and bounded away from zero.
\end{itemize}
We proceed in three steps $s = 1,2,3$. In each step we consider different classes of functions $\mathcal{F}_i^s$ and $\mathcal{F}_j^s$ for which we prove~\eqref{eq:gauss-mini-for-class-of-functions} to be valid. At the final step $s=3$, the classes $\mathcal{F}_i^3$, $\mathcal{F}_j^3$ will consist of all non-negative, integrable functions.

\paragraph{\bf{Step 1.}} Fix $f_i \in \mathcal{F}_i^1$  ($1 \le i \le m^+$) and $f_j \in \mathcal{F}_j^1$ ($m^+ < j \le m$) where
\begin{itemize}
\item $\mathcal{F}_i^1$ is the class of non-negative bounded measurable functions on $H_i$ with compact support,
\item $\mathcal{F}_j^1$ is the class of positive bounded Lipschitz functions $f_j$ on $H_j$ for which $f_j(y)^{-1}$ is bounded from above by a polynomial in $|y|$.
\end{itemize}
Note that $f_j$ belongs to $\mathcal{F}_j^0$ as well. For each $i$, take $R_i > 0$ such that the ball $B_{H_i}(0,R_i)$ contains the support of $f_i$. Consider the sequence of functions
\[
  f_{i,n} = f_i \ast \phi_{i,n} + \frac{1}{n} \Ind{B_{H_i}(0,R_i + 1)},
\]
where $\phi_{i,n}(x) = c_i n^{\dim H_i} \dist(x, H_i \setminus B_{H_i}(0, 1/n))$ and $c_i$ is such that $\int_{H_i} \phi_{i,n} = 1$. 

Since $\phi_{i,n}$ are bounded, Lipschitz and $f_i$ are non-negative and measurable with compact support, $f_{i,n} \in \mathcal{F}_i^0$. Moreover $\int_{H_i} f_{i,n} \to \int_{H_i} f_i$ and by the Lebesgue differentiation theorem, $f_{i,n} \to f_i$ a.e.. In other words the set $\Omega_i\subset H_i$ where the latter convergence holds has a negligible 
complement. Then the convergence $\mathcal{I}((f_{i,n}), (f_j)) \to \mathcal{I}((f_i), (f_j))$ holds for all points
in the set $\Omega:=\bigcap_{i=1}^{m^+} B_i^{-1}(\Omega_i)$. Since the maps $B_i$ are surjective, the complement of $\Omega$
is negligible. Hence  $\mathcal{I}((f_{i,n}), (f_j)) \to \mathcal{I}((f_i), (f_j))$ a.e.

In order to verify~\eqref{eq:gauss-mini-for-class-of-functions} it is enough to ensure that
\begin{equation}\label{eq:convergence-of-J}
\lim_{n \to \infty} \int_H \mathcal{I}((f_{i,n}), (f_j)) = \int_H \mathcal{I}((f_i), (f_j)).
\end{equation}
To this end we find an integrable function on $H$ which dominates $\mathcal{I}((f_{i,n}), (f_j))$ uniformly in $n$ and then apply the Lebesgue dominated convergence theorem.

First, since $f_i$ are bounded, all $f_{i,n}$ are bounded uniformly in $n$ and thus for some constant $C > 0$ and a compact set $F \subseteq H_1 \times \cdots \times H_{m^+}$,
\[
  \prod_{i=1}^{m^+} f_{i,n}^{c_i}(B_i x) \le C \Ind{F}(B_+ x)\quad\textup{for all $x \in H$.}
\]
Second, by~\eqref{eq:compact-image}, $B_{m+1}(B_+^{-1}(F))$ is compact, hence after adjusting the constant $C$, we also have
\begin{equation}\label{eq:domination-of-integrand-1}
  e^{\mathcal{Q}_-(B_{m+1} x)} \prod_{i=1}^{m^+} f_{i,n}^{c_i}(B_i x) \le C \Ind{F}(B_+ x)\quad\textup{for all $x \in H$.}
\end{equation}
Since the map $(B_0, B_+)$ is a linear isomorphism, $\mathcal{I}((f_{i,n}), (f_j))$ would be compactly supported if only the function $e^{-\mathcal{Q_+}}$ was compactly supported. Obviously it is not (unless $\mathcal{Q_+}$ is trivial), but we can still use a compactness argument by decomposing $e^{-\mathcal{Q_+}}$ into slices, namely
\[ \begin{split}
  e^{-\mathcal{Q_+}(y)} &= \int_0^1 \Ind{\{e^{-\mathcal{Q_+}} \ge u\}}(y) \, du
  = \int_0^\infty 2t e^{-t^2} \Ind{\{\mathcal{Q_+} \le t^2\}}(y) \, dt \\
  &= \int_0^\infty 2 t e^{-t^2} \Ind{t \{\mathcal{Q_+} \le 1\}}(y) \, dt.
\end{split} \]
Combining the above with~\eqref{eq:domination-of-integrand-1}, we can bound $\mathcal{I}((f_{i,n}), (f_j))(x)$ pointwise and uniformly in $n$ by a constant times
\[ \begin{split}
  \int_0^\infty t e^{-t^2} \Ind{t \{\mathcal{Q_+} \le 1\}}(B_0 x) \Ind{F}(B_+ x) \, dt
  \prod_{j = 1+m^+}^m f_j^{c_j}(B_j x) \\
  \le
  \int_0^\infty t e^{-t^2} \Ind{(t+1)B_H(0,R)}(x) \, dt
  \prod_{j = 1+m^+}^m f_j^{c_j}(B_j x) 
\end{split} \]
with $R > 0$ large enough. Since $f_j(y)^{-1}$ is bounded from above by a polynomial in $|y|$, we finally obtain a pointwise upper bound
\[
  \mathcal{I}((f_{i,n}), (f_j))(x) \le \int_0^\infty (C_1 t^\alpha + C_2) e^{-t^2} \Ind{(t+1)B_H(0,R)}(x) \, dt
\]
with some constants $C_1, C_2, \alpha > 0$, which is clearly an integrable function of $x \in H$.

\medskip
\paragraph{\bf{Step 2.}} Fix $f_i \in \mathcal{F}_i^2$ ($1\le  i \le m^+$) and $f_j \in \mathcal{F}_j^2$ ($m^+ < j \le m$) where
\begin{itemize}
\item $\mathcal{F}_i^2 = \mathcal{F}_i^1$,
\item $\mathcal{F}_j^2$ is the class of non-negative integrable functions on $H_j$.
\end{itemize}
Let $\phi(u) = \frac{1}{\pi(1+u^2)}$ be the density of the standard Cauchy distribution and for each $j = 1+m^+, \ldots, m$ and $\lambda > 0$ put
\[
  \phi_{j,\lambda}(x) = \lambda^{\dim H_j} \prod_{l=1}^{\dim H_j} \phi(\lambda x_l).
\]
Fix $\varepsilon > 0$ and for each $j$ and $n$ put
\[
  f_{j,n} = (f_j + \varepsilon \phi_{j,1}) \ast \phi_{j,n}.
\]
For each $j$ and $n$, $\phi_{j,n}$ is bounded and Lipschitz and $f_j$ is integrable, hence $f_{j,n}$ is also bounded and Lipschitz. From the estimate
\[
f_{j,n} \ge \varepsilon \phi_{j,1} \ast \phi_{j,n} = \varepsilon \phi_{j,\frac{n}{n+1}} %%%\ge \varepsilon \left(\frac{n+1}{n}\right)^{\dim H_j} \phi_{j,1}\left(\frac{n}{n+1}\,\cdot\,\right) 
\ge \varepsilon 2^{-\dim H_j} \phi_{j,1}
\]
(the equality above follows from the fact that the Cauchy distribution is $1$-stable), we obtain that
\begin{equation}\label{eq:polynomial-decay}
f_{j,n}(y)^{-1} \quad\text{is bounded from above by a polynomial in $|y|$ uniformly in $n$.}
\end{equation}
Hence we proved that $f_{j,n} \in \mathcal{F}_j^1$.

By the result of Step 1 (i.e.~\eqref{eq:gauss-mini-for-class-of-functions} for $(f_i), (f_{j,n})$),
\[ \begin{split}
  \int_H \mathcal{I}((f_i), (f_{j,n})) &\ge K \prod_{1\le i\le m^+} \left( \int_{H_i} f_i \right)^{c_i} \prod_{m^+ < j\le m} \left( \int_{H_j} f_{j,n} \right)^{c_j} \\
  &= K \prod_{1\le i\le m^+} \left( \int_{H_i} f_i \right)^{c_i} \prod_{m^+ < j\le m} \left( \varepsilon + \int_{H_j} f_j\right)^{c_j},
\end{split} \]
where the last equality follows from $\int_{H_j} f_{j,n} = \varepsilon + \int_{H_j} f_j$. Obviously
\[
  \int_H \mathcal{I}((f_i), (f_j)) \ge \int_H \mathcal{I}((f_i), (f_j + \varepsilon \phi_{j,1})),
\]
so proving
\begin{equation}\label{eq:convergence-of-J-2}
  \lim_{n \to \infty} \int_H \mathcal{I}((f_i), (f_{j,n})) = \int_H \mathcal{I}((f_i), (f_j + \varepsilon \phi_{j,1}))
\end{equation}
would yield
\[
  \int_H \mathcal{I}((f_i), (f_j)) \ge K \prod_{1\le i\le m^+} \left( \int_{H_i} f_i \right)^{c_i} \prod_{m^+ < j\le m} \left( \int_{H_j} f_j + \varepsilon \right)^{c_j}
\]
and in consequence~\eqref{eq:gauss-mini-for-class-of-functions} by letting $\varepsilon \to 0$.
 
Since $f_{j,n} \to f_j + \varepsilon \phi_{j,1}$ a.e., we have $\mathcal{I}((f_i), (f_{j,n})) \to \mathcal{I}((f_i), (f_j + \varepsilon \phi_{j,1}))$ a.e. In the view of~\eqref{eq:polynomial-decay} and the fact that $f_i$ are bounded with compact support, we can proceed as in Step 1 to find an integrable function on $H$ which dominates $\mathcal{I}((f_i), (f_{j,n}))$ for all $n$ and conclude with~\eqref{eq:convergence-of-J-2}.

\medskip
\paragraph{\bf{Step 3.}} Fix $f_i \in \mathcal{F}_i^3$ ($1\le i \le m^+$) and $f_j \in \mathcal{F}_j^3$ ($m^+ < j \le m$) where
\begin{itemize}
\item $\mathcal{F}_i^3$ is the class of non-negative integrable functions on $H_i$,
\item $\mathcal{F}_j^3 = \mathcal{F}_j^2$.
\end{itemize}
We approximate $f_i$ with $f_{i,n} = \min(f_i, n) \Ind{B_{H_i}(0,n)}$ which belong to $\mathcal{F}_i^2$. The convergence as in~\eqref{eq:convergence-of-J} follows from the monotone convergence theorem. We conclude with~\eqref{eq:gauss-mini-for-class-of-functions} by using the result of Step 2 for the functions $(f_{i,n}), (f_j)$.

\section{Geometric Brascamp-Lieb inequality}\label{sec:geometric}

We study specific non-degenerate situations for  which $\inf J=1$  and some extremizing functions can be identified.  They are related to 
geometric Brascamp-Lieb inequalities and the decomposition of the identity~\eqref{eq:decomposition-identity-1d}. More precisely, they are characterized by the following conditions:
\begin{align}
  \label{cond:geometric-condition-0}
  B_k B_k^\ast &= \Id_{H_k} \quad \text{for $k = 1,\ldots, m,$} \\
  \label{cond:geometric-condition}
  Q + \sum_{k=1}^m c_k B_k^\ast B_k &= \Id_{H}.
\end{align}

\subsection{Finding $\inf_{\mathcal{CG}} J$}
The aim of this subsection is to prove that if the non-degeneracy condition~\eqref{eq:non-degeneracy2} and the geometric conditions ~\eqref{cond:geometric-condition-0} and~\eqref{cond:geometric-condition} hold then the infimum of $J$ on centered Gaussian functions is  equal to 1, and is achieved when for all $k$,  $f_k(\cdot)=\exp(-\pi|\,\cdot\,|^2$). The crucial result here is Proposition~\ref{prop:concavity-of-log-det} which establishes a concavity property of a function related to Formula~\eqref{eq:J-on-Gaussian-input}.

First, let us put forward two useful facts and a lemma.
\begin{fact}[{see e.g.~\cite[Theorem 7.7.6]{horn_johnson_matrix_analysis}}]\label{fact:schur-complement}
Let $A$ be $n \times n$ real symmetric matrix and C be $m \times m$ real symmetric matrix and $B$ be an $n \times m$ real matrix. Let $X = \left( \begin{array}{cc} A & B \\ B^\ast & C \end{array} \right)$. If $C > 0$ then
\begin{enumerate}
  \item[(i)] $X \ge 0$ if and only if $A - BC^{-1}B^\ast \ge 0$;
  \item[(ii)] $X > 0$ if and only if $A - BC^{-1}B^\ast > 0$.
\end{enumerate}
\end{fact}
Below $\I_n$ denotes the $n \times n$ identity matrix.
\begin{fact}[Woodbury formula]\label{fact:woodbury-formula}
For an $m \times n$ matrix $A$,
\[
  A^\ast (\I_m + A A^\ast)^{-1} A = \I_n - (\I_n + A^\ast A)^{-1}.
\]
\end{fact}
\begin{proof}
Direct calculation.
\end{proof}
\begin{lemma}\label{lem:equivalence-of-positive-definiteness}
Let $R$ be a $p \times n$ real matrix and $S$ be an $r \times n$ real matrix. Consider the $(p+r) \times (p+r)$ matrix
\[
  M = \left( \begin{array}{cc}
    -\I_p + R R^\ast & R S^\ast \\
    S R^\ast & \I_r + S S^\ast
    \end{array} \right).
\]
\begin{enumerate}
\item[(i)] $M \ge 0$ if and only if $R(\I_n + S^\ast S)^{-1} R^\ast \ge \I_p$.
\item[(ii)] If $n = p$ and $R$ is invertible then $M \ge 0$ if and only if $R^\ast R - S^\ast S \ge \I_n$.
\end{enumerate}
\end{lemma}
\begin{proof}
Applying Fact~\ref{fact:schur-complement}(i) we obtain that $M \ge 0$ is equivalent to
\[
  -\I_p + R R^\ast - R S^\ast (\I_r + S S^\ast)^{-1} S R^\ast \ge 0.
\]
Using Fact~\ref{fact:woodbury-formula} for $A = S$ the above can be rephrased as
\[
  -\I_p + R R^\ast - R \big(\I_n - (\I_n + S^\ast S)^{-1}\big) R^\ast \ge 0,
\]
which finishes the proof of (i).

If $n=p$ and $R$ is invertible, then $M \ge 0$ is also equivalent to
\[
  (\I_n + S^\ast S)^{-1} \ge R^{-1} R^{-\ast} = (R^\ast R)^{-1}
\]
which in turn is equivalent to
\[
  \I_n + S^\ast S \le R^\ast R.
\]
\end{proof}

In the context of Brascamp-Lieb inequalities the following easy consequence of the Cauchy-Binet formula is useful, see Proposition 6 of \cite{barthe-inventiones}:
if $d\le n$ and $U$ is an $n\times d$ matrix, then the map 
$$x\in \mathbb R^n \mapsto  \log \det\big(U^\ast \diag((e^{x_i})_{i \le n}) U\big)$$
is convex. The next property is a counterpart for inverse Brascamp-Lieb inequalities.
\begin{prop}\label{prop:concavity-of-log-det}
  Let $m \ge n \ge 1$ and $U$ be an invertible $n \times n$ real matrix and $V$ be a real $(m-n) \times n$ matrix. Let
  \[
    \Omega = \big\{(x_1, \ldots, x_m) \in \R^m \colon U^\ast \diag((e^{x_i})_{i \le n}) U - V^\ast \diag((e^{x_j})_{j > n}) V > 0 \big\}.
  \]
Then $\Omega$ is convex and the map $\phi \colon \Omega \to \R$,
\[
  \phi(x_1, \ldots, x_m) = \log \det\big(U^\ast \diag((e^{x_i})_{i \le n}) U - V^\ast \diag((e^{x_j})_{j > n}) V\big)
\]
is concave.
\end{prop}
\begin{proof}
First we show that $\Omega$ is convex. Take any $x = (x_1, \ldots, x_m) \in \Omega$, $y = (y_1, \ldots, y_m) \in \Omega$ and $\lambda \in (0,1)$. Let $X = \diag((e^{x_k}/2)_{1 \le k \le m})$, $Y = \diag((e^{y_k/2})_{1 \le k \le m})$ and $X_+, X_-$ be the diagonal blocks of $X$ of size $n \times n$ and $(m-n) \times (m-n)$ (resp.), similarly $Y_+$, $Y_-$.

From $x \in \Omega$ it follows that $U^\ast X_+^2 U > V^\ast X_-^2 V$. By invertibility of $U$,
\begin{align*}
  \I_n > X_+^{-1} U^{-\ast} V^\ast X_-^2 V U^{-1} X_+^{-1} = (X_- V U^{-1} X_+^{-1})^\ast (X_- V U^{-1} X_+^{-1}),
\end{align*}
which is equivalent to $\| X_- V U^{-1} X_+^{-1} \| < 1$ ($\|\cdot\|$ denotes the operator norm).
Similarly, $y \in S$ implies $\| Y_- V U^{-1} Y_+^{-1} \| < 1$.

Put $A = Y_- V U^{-1} X_+^{-1}$ and $B = X_- Y_-^{-1}$ and $C = X_+ Y_+^{-1}$. Then we have
\begin{align*}
  \|B A\| < 1, \qquad \|A C\| < 1.
\end{align*}
Now use~\cite[Corollary IX.5.3]{Bhatia-Matrix-analysis} which asserts that
\[
  \|B^\lambda A C^{1-\lambda}\| \le \|B A\|^\lambda \|A C\|^{1-\lambda}
\]
to obtain
\begin{align*}
  \I_n > (B^\lambda A C^{1-\lambda})^\ast (B^\lambda A C^{1-\lambda}) = Y_+^{\lambda-1} X_+^{-\lambda} U^{-\ast} V^\ast X_-^{2\lambda} Y_-^{2(1-\lambda)} V U^{-1} X_+^{-\lambda} Y_+^{\lambda-1},
\end{align*}
which is equivalent to
\begin{align}\label{ineq:to-be-in-omega}
  U^\ast (X_+^2)^\lambda (Y_+^2)^{1-\lambda} U > V^\ast (X_-^2)^\lambda (Y_-^2)^{1-\lambda} V.
\end{align}
Since $(X^2)^\lambda (Y^2)^{1-\lambda}$ is the diagonal matrix with the entries $e^{\lambda x_k + (1-\lambda) y_k}$, \eqref{ineq:to-be-in-omega} ensures that $\lambda x + (1-\lambda) y \in \Omega$.

Next we establish concavity of $\phi$ on $\Omega$. For $x = (x_1, \ldots, x_m) \in \R^m$, set $x_+ = (x_1, \ldots, x_n) \in \R^n$ and $x_- = (x_{n+1}, \ldots, x_m) \in \R^{m-n}$. Let $A \colon \R^m \to \R^{n \times n}$ be defined as
\[
  A(x) = U^\ast e^{\diag(x_+)} U - V^\ast e^{\diag(x_-)} V.
\]
Then $\phi(x) = \log \det A(x)$ for $x \in \Omega$.
Since $\phi$ is a smooth function, we can analyze the Hessian of $\phi$. To this end, we will use the following formulas:
\begin{align}
  \label{eq:d-log-det}
  \partial \log \det X &= \tr(X^{-1} \partial X), \quad \textup{for $X > 0,$} \\
  \label{eq:d-x-inverse}
  \partial X^{-1} &= -X^{-1} (\partial X) X^{-1}, \quad \textup{for $X > 0,$} \\
  \label{eq:d-A-of-x}
  \partial A(x) &= U^\ast e^{\diag(x_+)} \diag((\partial x)_+) U - V^\ast e^{\diag(x_-)} \diag((\partial x)_-) V.
\end{align}
Specialization of~\eqref{eq:d-A-of-x} to partial derivatives gives
\begin{align}
  \label{eq:di-A-of-x}
  \partial_i A(x) &= e^{x_i} U^\ast e_i e_i^\ast U
  \quad \textup{for $i \le n,$} \\
  \label{eq:dj-A-of-x}
  \partial_j A(x) &= -e^{x_j} V^\ast f_{j-n} f_{j-n}^\ast V
  \quad \textup{for $j > n,$}
\end{align}
where for $i\le n$, $e_i$ is a column matrix with $n$ rows, a coefficient $1$ in the $i$-th row and all other coefficients equal to 0. Similarly, for  $\ell\le m-n$, $f_\ell$
is a column matrix with $m-n$ rows, with a 1 in its $\ell$-th row and zeroes elsewhere.
Fix $x \in \Omega$. Using~\eqref{eq:d-log-det} and~\eqref{eq:di-A-of-x}, for $i \le n$ we obtain
\[
  \partial_i \phi(x) = \tr \big( A^{-1}(x) \partial_i A(x) \big) = e^{x_i} e_i^\ast U A^{-1}(x) U^\ast e_i,
\]
Similarly, for $j > n$,
\[
  \partial_j \phi(x) = -e^{x_j} e_{j-n}^\ast V A^{-1}(x) V^\ast e_{j-n}.
\]
In order to calculate second order partial derivatives, we use~\eqref{eq:d-x-inverse} combined with~\eqref{eq:di-A-of-x} or~\eqref{eq:dj-A-of-x}. For $i_1, i_2 \le n$ and $i_1 \neq i_2$ we have
\[ \begin{split}
  \partial^2_{i_1 i_2} \phi(x) &= -e^{x_{i_1}} e_{i_1}^\ast U A^{-1}(x) \partial_{i_2} A(x) A^{-1} U^\ast e_{i_1} \\
  &= -e^{x_{i_1} + x_{i_2}} e_{i_1}^\ast U A^{-1}(x) U^\ast e_{i_2} e_{i_2}^\ast U A^{-1}(x) U^\ast e_{i_1}
\end{split} \]
Denoting $$R = e^{\diag(x_+)/2} U A^{-1/2}(x)$$ we can write the above second order mixed partial derivative in a more compact way
\[
  \partial^2_{i_1 i_2} \phi(x) = -(R R^\ast)_{i_1 i_2} (R R^\ast)_{i_2 i_1} = -(R R^\ast)^2_{i_1 i_2},
\]
and for $i \le n$ we have
\[
  \partial^2_{ii} \phi(x) = \partial_i \phi(x) - (R R^\ast)^2_{ii} = (R R^\ast)_{ii} - (R R^\ast)^2_{ii}.
\]
Combining the two above formulas we can write that for any $i_1, i_2 \le n$,
\[
  \partial^2_{i_1 i_2} \phi(x) = (R R^\ast)_{i_1 i_2} (\I_n - R R^\ast)_{i_1 i_2}.
\]
If we denote $$S = e^{\diag(x_-)/2} V A^{-1/2}(x),$$ then by similar calculations we get that for $j_1, j_2 > n$,
\[
  \partial^2_{j_1, j_2} \phi(x) = -(S S^\ast)_{j_1-n, j_2-n} (\I_{m-n} + S S^\ast)_{j_1-n, j_2-n}.
\]
Lastly, for $i \le n$ and $j > n$,
\[ \begin{split}
  \partial^2_{ij} \phi(x) &= -e^{x_i} e_i^\ast U A^{-1}(x) \partial_j A(x) A^{-1}(x) U^\ast e_i \\
  &= e^{x_i + x_j} e_i^\ast U A^{-1}(x) V^\ast f_{j-n} f_{j-n}^\ast V A^{-1}(x) U^\ast e_i \\
  &= (R S^\ast)_{i, j-n} (S R^\ast)_{j-n, i} = (R S^\ast)_{i, j-n}^2 = (S R^\ast)^2_{j-n, i}.
\end{split} \]
As a result,
\[ \begin{split}
  \Hess \phi(x) &= -\left( \begin{array}{cc}
    (R R^\ast) \circ (-\I_n + R R^\ast) & -(R S^\ast) \circ (R S^\ast) \\
    -(S R^\ast) \circ (S R^\ast) & (S S^\ast) \circ (\I_{m-n} + S S^\ast)
  \end{array} \right) \\
  &= -\underbrace{\left( \begin{array}{cc}
      R R^\ast & -R S^\ast \\
     -S R^\ast &  S S^\ast
  \end{array} \right)}_{M} \circ
  \underbrace{\left( \begin{array}{cc}
    -\I_n + R R^\ast & R S^\ast \\
    S R^\ast         & \I_{m-n} + S S^\ast
  \end{array} \right)}_{N},
\end{split} \]
where $A \circ B$ denotes the Hadamard product (i.e. entry-wise product) of $A$ and $B$. 

Note that $M$ is positive semi-definite. Indeed, $M = \left( \begin{array}{c} R \\ -S \end{array} \right) (R^\ast, -S^\ast) \ge 0.$ Now we argue that also $N \ge 0$. From definitions of the matrices $A(x)$ and $R$ and $S$ it follows immediately that
\[
  R^\ast R - S^\ast S = \I_n.
\]
Since $U$ is invertible, so does $R$ and Lemma~\ref{lem:equivalence-of-positive-definiteness}(ii) implies $N \ge 0$.

Now it is enough to apply the Schur product theorem (see e.g.~\cite[Theorem 7.5.3]{horn_johnson_matrix_analysis}) which asserts that if $M$ and $N$ are positive semi-definite then $M \circ N$ is also positive semi-definite. Therefore $\Hess \phi(x)$ is negative semi-definite at any $x \in \Omega$ and hence $\phi$ is concave.
\end{proof}

The next theorem uses the notation from Subsection~\ref{subsec:centered-gaussians}.
\begin{theorem}\label{thm:D-in-extremisable-case}
  Assume the non-degeneracy condition~\eqref{eq:non-degeneracy2} holds. Let $(A_1, \ldots, A_m) \in \Lambda$. Put $A = Q + \sum_{k=1}^m c_k B_k^\ast A_k B_k>0$.
  Then the supremum in~\eqref{def:BL-constant} is attained at $(A_1, \ldots, A_m)$, i.e.
  \begin{align}\label{eq:gaussian-extremisable}
    D = \frac{\det A}{\prod_{i=k}^m (\det A_k)^{c_k}}
  \end{align}
  if and only if
  \begin{align}\label{eq:generalized-geometric-condition}
    A_k^{-1} - B_k A^{-1} B_k^\ast = 0 \quad\text{for all $k=1,\ldots,m$ for which $c_k \neq 0.$}
  \end{align}
  In particular, if~\eqref{cond:geometric-condition-0} and~\eqref{cond:geometric-condition} hold then $D=1$.
\end{theorem}
\begin{remark}
If \eqref{eq:generalized-geometric-condition} holds, then $\tilde Q:=A^{-1/2}QA^{-1/2}$ and $\tilde{B}_k:=A_k^{1/2}B_kA^{-1/2}$ satisfy the generalized geometric conditions \eqref{cond:geometric-condition-0}
and \eqref{cond:geometric-condition}. This allows to show that up to linear isomorphisms (by $A^{1/2}$ on $H$ and $A_k^{1/2}$ on $H_k$)
the situations where $\inf_{\mathcal{CG}}J$ is achieved are equivalent to the geometric situations. This follows exactly what happens for direct Brascamp-Lieb inequalities, see \cite{BCCT-structure}.
\end{remark}

\begin{proof}[Proof of Theorem \ref{thm:D-in-extremisable-case}]
Consider the function $\Phi \colon \Lambda \to \R$,
\[
  \Phi(M_1, \ldots, M_m) = \log \det \left(Q + \sum_{k=1}^m c_k B_k^\ast M_k B_k\right) - \sum_{k=1}^m c_k \log \det M_k.
\]
Note $\Phi$ is smooth and $\sup_\Lambda \Phi = \log D$.

Fix any self-adjoint operators $X_k \colon H_k \to H_k$ (for $k=1,\ldots,m$). Using Formula~\eqref{eq:d-log-det} the directional derivative of $\Phi$ at $(A_1,\ldots,A_m)$ in the direction of $(X_1, \ldots, X_m)$ is
\begin{align*}
  & \partial_{(X_1, \ldots, X_m)} \Phi(A_1, \ldots, A_m) \\
  &=
  \lim_{t \to 0} \frac{1}{t} \big(\Phi(A_1 + tX_1, \ldots, A_m + tX_m) - \Phi(A_1,\ldots,A_m)\big) \\
  &= \tr\left(A^{-1} \Big(\sum_{k=1}^m c_k B_k^\ast X_k B_k\Big)\right) - \sum_{k=1}^m c_k \tr(A_k^{-1} X_k) \\
  &= \sum_{k=1}^m c_k \tr(A^{-1} B_k^\ast X_k B_k) - \sum_{k=1}^m c_k \tr(A_k^{-1} X_k) =
  \sum_{k=1}^m c_k \tr\big((B_k A^{-1} B_k^\ast - A_k^{-1}) X_k \big).
\end{align*}
The condition~\eqref{eq:gaussian-extremisable} implies that the derivative must be $0$. Using the fact that a self-adjoint operator $Y$ is zero if and only if $\tr (Y X) = 0$ for all self-adjoint operators $X$, \eqref{eq:generalized-geometric-condition} follows by considering $X_1, \ldots, X_{k-1}, X_{k+1}, \ldots, X_m$ being zero and $X_k$ being arbitrary for each $k=1,\ldots,m$ such that $c_k \neq 0$.

For the converse implication assume that~\eqref{eq:generalized-geometric-condition} holds. Then the above calculation shows that the derivative of $\Phi$ at $(A_1,\ldots,A_m)$ is zero. In order to conclude that $\Phi$ has a global maximum at this point, we prove below that $\Phi$ enjoys a concavity type property along well chosen curves.

Fix any self-adjoint operators $Y_k \colon H_k \to H_k$ (for $k=1,\ldots,m$) and for any real $t$ put
\begin{equation}\label{eq:exponential-parametrization}
  \mathbf{A}(t) = \big(A_1^{1/2} \exp(t Y_1) A_1^{1/2}, \ldots, A_m^{1/2} \exp(t Y_m) A_m^{1/2}\big).
\end{equation}
For $t \in \R$ for which $\mathbf{A}(t) \in \Lambda$ consider the function
\[
  \vp(t) = \Phi\big(\mathbf{A}(t)\big).
\]
Since $\Lambda$ is an open set and the function $t \mapsto \mathbf{A}(t)$ is continuous, the domain of $\vp$ is an open subset of $\R$. The domain contains $0$ and, as $\vp$ is smooth, $\vp'(0)$ vanishes.

For each $k = 1,\ldots,m$ take an orthogonal transformation $U_k \in O(H_k)$ such that $U_k^\ast Y_k U_k$, when identified with its matrix in the standard basis $(e_l^{H_k})_l$ in $H_k$, is a diagonal matrix. Denote the diagonal entries by $y_{k 1}, \ldots, y_{k n_k}$, where $n_k = \dim H_k$. Then
\begin{equation}\label{eq:rep-expY}
  B_k^\ast A_k^{1/2} \exp(tY_k) A_k^{1/2} B_k = (U_k A_k^{1/2} B_k)^\ast \diag\big((e^{t y_{k l}})_{l \le n_k}\big) U_k A_k^{1/2} B_k.
\end{equation}
Thanks for the non-degeneracy condition~\eqref{eq:non-degeneracy2} we can use the decomposition of the Gaussian kernel $\exp(-\mathcal{Q})$ as asserted by Lemma~\ref{lem:decompose-Q}. Beside the maps $B_0 \colon H \to H_0$ and $B_{m+1} \colon H \to H_{m+1}$ consider also $A_0 > 0$ on $H_0$ and $A_{m+1} > 0$ on $H_{m+1}$ such that
\[
  Q = \sum_{k \in \{0, m+1\}} c_k B^\ast_k A_k B_k,
\]
with $c_0 = 1$ and $c_{m+1} = -1$. For the sake of consistency with~\eqref{eq:rep-expY}, for $k \in \{0, m+1\}$ put $Y_k = 0$ (a zero map on $H_k$), $U_k = \Id_{H_k}$ and $y_{kl} = 0$ for all $l \le n_k = \dim H_k$.

Let
\[ \begin{split}
  U &= (\sqrt{c_i} U_i A_i^{1/2} B_i)_{0 \le i \le m^+} \colon H \to H_0 \times \cdots \times H_{m^+} \\
  V &= (\sqrt{-c_j} U_j A_j^{1/2} B_j)_{m^+ < j \le m++1} \colon H \to H_{m^+ +1} \times \cdots \times H_{m+1}.
\end{split} \]
Considering the diagonal matrix $D_+(t) = \diag\big((e^{t y_{i l}})_{0 \le i \le m^+, l \le n_i}\big)$ as an operator acting on $H_0 \times \cdots \times H_{m^+}$ and the diagonal matrix $D_-(t) = \diag\big((e^{t y_{j l}})_{m^+ < j \le m+1, l \le n_j}\big)$ as an operator acting on $H_{m^+ +1} \times \cdots \times H_{m+1}$ we can write
\begin{align*}
  \vp(t) = \log\det\left(U^\ast D_+(t) U - V^\ast D_-(t) V \right) - \sum_{k=1}^m c_k \log\det A_k - t \sum_{k=1}^m c_k \tr Y_k,
\end{align*}
where used the formula $\log\det(\exp(Y)) = \tr Y$ for a self-adjoint $Y$.

It follows from Assertion~\eqref{eq:isomorphism-condition}  that $U$ is a linear isomorphism. Therefore we can apply Proposition~\ref{prop:concavity-of-log-det}, which tells us  that the domain of $\vp$ must be an open interval and that $\vp$ is concave. Since $\vp'(0) = 0$, $\vp$ attains its global maximum at $t=0$.

Since for any $(X_1, \ldots, X_m) \in \Lambda$ there exist $t$ and self adjoint operators $Y_k$ on $H_k$ (for $k=1,\ldots,m$) such that $(X_1, \ldots, X_m)$ is of the form~\eqref{eq:exponential-parametrization} (e.g. take $t=1$ and $Y_k = \log(A_k^{-1/2} X_k A_k^{-1/2}$), we actually showed that $(A_1,\ldots,A_m)$ is a global maximum of $\Phi$ and thus~\eqref{eq:gaussian-extremisable} holds.
\end{proof}

\subsection{Geometric version of Inverse Brascamp-Lieb inequalities}

\begin{theorem}\label{th:geometric-IBL}
For $k=1,\ldots,m$, let $c_k\in \mathbb R$ and  let $B_k:H\to H_k$ be linear surjective maps such that 
$B_kB_k^*=\mathrm{Id}{H_k}$. Let $Q:H\to H$ be a symmetric operator.
Assume that 
\[
 Q+\sum_{k=1}^m c_k B_k^*B_k=\mathrm{Id}_H \quad \mathrm{and}\quad
 \dim H\ge s^+(Q)+\sum_{k:\, c_k>0} \dim H_k.
\]
Then for all non-negative integrable functions $h_k:H_k \to [0,+\infty]$ with 
$\int h_k >0$, it holds
\[
  \int_{H} \exp(-\pi \scalar{x, Qx}) \prod_{k=1}^m h_k^{c_k}(B_k x) \, dx \ge \prod_{k=1}^m \Big( \int_{H_k} h_k \Big)^{c_k}.
\]
There is equality when for all $k$ and all $y\in H_k$, $f_k(y)=\exp(-\pi|y|^2)$.
\end{theorem}
\begin{proof}
We may assume without loss of generality that $c_1,\ldots, c_{m^+}>0>c_{1+m^+},\ldots, c_m$. The above decomposition of the identity implies that
\[ Q+\sum_{i=1}^{m^+} c_i B_i^*B_i=\mathrm{Id}_H+\sum_{j>m^+} |c_j| B_j^*B_j>0.\] 
Hence the restriction of $Q$ to $\ker B_+=\bigcap_{i=1}^{m^+} \ker B_i$ is positive definite.   The non-degeneracy conditions \eqref{eq:non-degeneracy2} are verified and we may apply  Theorem \ref{theo:gauss-mini} to conclude $\inf J=\inf_{\mathcal {CG}} J$. Then Theorem \ref{thm:D-in-extremisable-case} ensures that $\inf_{\mathcal {CG}} J=D^{-\frac12}=1$.
\end{proof}

\subsection{Relation with the results of Chen, Dafnis and Paouris}

The reverse Gaussian correlation inequality by Chen, Dafnis and Paouris, presented here in Theorem~\ref{theo:CDP}, turns out to be the geometric version of our main result. To see this, let us consider a slight reformulation of the second inequality from Theorem~\ref{theo:CDP}, which appears explicitly in~\cite{chen-dafnis-paouris}:

\begin{theorem}[{\cite[Theorem 3(ii)]{chen-dafnis-paouris}}]\label{thm:CDP-thm-3}
Let $\gamma_E$ stand for the standard Gaussian measure on a Euclidean space $E$.
Let $B_k \colon H \to H_k$ (for $k=1,\ldots,m$) be linear maps satisfying $B_kB_k^*=\mathrm{Id}_{H_k}$. 
 Denote
\[
  B = (B_1, \ldots, B_m) \colon H \to H_1 \times \cdots \times H_m
\]
and let $C \colon H_1 \times \cdots \times H_m \to H_1 \times \cdots \times H_m$ be the block diagonal operator defined as
\[
  C = \diag\big(c_1 \Id_{H_1}, \ldots, c_m \Id_{H_m}\big).
\]
If
\begin{equation}\label{cond:CDP-algebraic-condition}
  B B^\ast \ge C^{-1}
\end{equation}
then for any non-negative functions $f_k \in L^1(H_k, \gamma_{H_k})$ ($k=1,\ldots,m$),
\[
  \int_{H} \prod_{k=1}^m f_k^{c_k}(B_k x) \, d\gamma_H(x) \ge \prod_{k=1}^m \Big( \int_{H_k} f_k \, d\gamma_{H_k} \Big)^{c_k},
\]
\end{theorem}
By setting $h_k(x) = f_k(\sqrt{2\pi} x) e^{-\pi |x|^2}$ we can rewrite the above inequality in terms of integrals with respect to the Lebesgue measure:
\[
  \int_{H} \exp(-\pi \scalar{x, Qx}) \prod_{k=1}^m h_k^{c_k}(B_k x) \, dx \ge \prod_{k=1}^m \Big( \int_{H_k} h_k \Big)^{c_k},
\]
where $Q = \Id_H - \sum_{k=1}^m c_k B_k^\ast B_k$. Hence the geometric condition~\eqref{cond:geometric-condition} is obviously satisfied.
% In what follows we show that~\eqref{cond:CDP-algebraic-condition} yields the non-degeneracy condition~\eqref{eq:non-degeneracy2} (with $\mathcal{Q}(x) = \pi \scalar{x, Qx}$). 
In order to deduce Theorem~\ref{thm:CDP-thm-3} from Theorem~\ref{th:geometric-IBL}, we need to establish the dimension condition 
 $\dim H\ge s^+(Q)+\sum_{k:\, c_k>0} \dim H_k$. This is what we do next.

Assume as usual that $c_1, \ldots, c_{m^+} > 0$ and $c_{m^+ +1}, \ldots, c_m < 0$. Recall that $B_+ = (B_1, \ldots, B_{m^+})$ and set $B_- = (B_{m^+ +1}, \ldots, B_m)$ and $H_+ = H_1 \times \cdots \times H_{m^+}$, $H_- = H_{m^+ +1} \times \cdots \times H_m$. The condition~\eqref{cond:CDP-algebraic-condition} is equivalent to
\[
  |C|^{1/2} B B^\ast |C|^{1/2} \ge \left( \begin{array}{cc} \Id_{H_+} & \\ & -\Id_{H_-} \end{array} \right).
\]
Introducing $\tilde{B}_k = |c_k|^{1/2} B_k$ for $k=1,\ldots,m$ and defining $\tilde{B}_+$ and $\tilde{B}_-$ correspondingly, the above condition can be rewritten as
\begin{equation}\label{cond:CDP-algebraic-condition-2}
  \left( \begin{array}{cc}
    -\Id_{H_+} + \tilde{B}_+ \tilde{B}_+^\ast & \tilde{B}_+ \tilde{B}_-^\ast \\[1ex]
    \tilde{B}_- \tilde{B}_+^\ast & \Id_{H_-} + \tilde{B}_- \tilde{B}_-^\ast
  \end{array} \right) \ge 0.
\end{equation}
Since the upper-left corner of the above matrix is semi-definite positive, we know  that $\tilde{B}_+ \tilde{B}_+^\ast$ is positive definite and hence $\tilde{B}_+$ is surjective, so does $B_+$. 
%Since $J((h_k)_{k=1}^m)$ for the functions $h_k(x) = e^{-\pi |x|^2}$ (which corresponds to $f_k \equiv 1$) assumes a finite value (more precisely, $J((h_k)_{k=1}^m) = 1$), we can apply Proposition~\ref{prop:finitevalues2} to conclude that $\mathcal{Q}$ is positive definite on $\ker B_+$.
%Secondly, we verify the dimension condition~\eqref{eq:surjectivity-condition}. 
Moreover, from Lemma~\ref{lem:equivalence-of-positive-definiteness}(i) we get that~\eqref{cond:CDP-algebraic-condition-2} is equivalent to 
\[
  \tilde{B}_+ \big(\Id_H + \tilde{B}_-^\ast \tilde{B}_-\big)^{-1} \tilde{B}_+^\ast \ge \Id_{H_+},
\]
which in view of the identity $Q = \Id_H - \tilde{B}_+^\ast \tilde{B}_+ + \tilde{B}_-^\ast \tilde{B}_-$ implies
\[
  \tilde{B}_+^\ast \tilde{B}_+ \big(\tilde{B}_+^\ast \tilde{B}_+ + Q\big)^{-1} \tilde{B}_+^\ast \tilde{B}_+ \ge \tilde{B}_+^\ast \tilde{B}_+.
\]
Thanks to the lemma below we conclude that $s^+(Q) \le \dim \ker \tilde{B}_+$, which coincides with the dimension condition~\eqref{eq:surjectivity-condition} since $\dim \ker \tilde{B}_+ = \dim H - \dim H_+$ due to surjectivity of $\tilde{B}_+$.
\begin{lemma}
Let $A, B$ be real symmetric matrices of size $d$ such that $A \ge 0$ and $A + B > 0$. If $A (A+B)^{-1} A \ge A$ then $s^+(B) \le \dim \ker A$.
\end{lemma}
\begin{proof}
Observe that the statement of this lemma is invariant under congruency (i.e. under replacing $A$ with $C^\ast A C$  and $B$ with $C^\ast B C$).
Since $A+B>0$ there exists an invertible matrix such that $C^*(A+B)C$ and $C^*AC$ are both diagonal. By subtraction we get that $C^*BC$ is diagonal too.
Hence, we may assume without loss of generality that $A=\diag\big((a_i)_{i=1}^d\big)$ and $B=\diag\big((b_i)_{i=1}^d\big)$ with for all $i$, $a_i\ge 0$ and $a_i+b_i>0$. The hypothesis $A (A+B)^{-1} A \ge A$ reads as $a_i^2/(a_i+b_i)\ge a_i$,
which is equivalent to $0\ge a_ib_i$, for all $i$. Since $a_i\ge 0$, we may deduce that for all $i$
$$ b_i>0 \Longrightarrow a_i=0.$$
The matrices being  diagonal, this implication means that $s^+(B)\le \dim \ker A$.
\end{proof}

Let us also comment on the Lebesgue version of the inverse Brascamp-Lieb inequalities presented in~\cite{chen-dafnis-paouris} as Theorem 2(ii). Applying suitable linear transformation in the Euclidean spaces $H$ and $H_1, \ldots, H_m$ one can formulate that result as follows:
\begin{theorem}[{\cite[Theorem 2(ii)]{chen-dafnis-paouris}}]
In the settings of Theorem~\ref{thm:CDP-thm-3}, if
\begin{equation}\label{cond:homogeneity}
  \dim H = \sum_{k=1}^m c_k \dim H_k
\end{equation}
and $BB^*\ge C^{-1}$ then for any non-negative integrable functions $f_k \in H_k \to [0,\infty)$,
\begin{equation}\label{eq:CDP-Lebesgue-BL}
  \int_{H} \prod_{k=1}^m f_k^{c_k}(B_k x) \, dx \ge \prod_{k=1}^m \Big( \int_{H_k} f_k \Big)^{c_k}.
\end{equation}
\end{theorem}

Let us explain how to reprove this results from what we already did, and settle a question on existence of cases of equalities that
was left open in \cite{chen-dafnis-paouris}.
Recall that $BB^*\ge C^{-1}$ is equivalent to~\eqref{cond:CDP-algebraic-condition-2} and implies that $B_+$ (equivalently $\tilde{B}_+$) is surjective. There are two possible cases:
\begin{enumerate}
\item[Case 1:] $B_+$ (equivalently $\tilde{B_+}$) is injective. In this case we can apply Lemma~\ref{lem:equivalence-of-positive-definiteness}(ii) to get that $\eqref{cond:CDP-algebraic-condition-2}$ is equivalent to
$\tilde{B}_+^\ast \tilde{B}_+ - \tilde{B}_-^\ast \tilde{B}_- \ge \Id_H$ or simply
\begin{equation}\label{cond:CDP-algebraic-condition-3}
  \sum_{k=1}^m c_k B_k^\ast B_k \ge \Id_H.
\end{equation}
Using the $B_kB_k^*=\mathrm{Id}_{H_k}$ and~\eqref{cond:homogeneity} we see the maps on both sides of the above inequality have the same trace. Hence there must be equality in~\eqref{cond:CDP-algebraic-condition-3}, i.e. the geometric condition~\eqref{cond:geometric-condition} holds. In particular for the functions $f_k(x) = \exp(-\pi |x|^2)$ we get equality in~\eqref{eq:CDP-Lebesgue-BL}. The decomposition of the identity also allows to deduce \eqref{eq:CDP-Lebesgue-BL} from Theorem \ref{cond:CDP-algebraic-condition} applied to the functions $f_k(\cdot)\exp(|\cdot|^2/2)$.
% Moreover, in the view of absence of the kernel ($Q=0$), the non-degeneracy condition~\eqref{eq:non-degeneracy2} follows immediately from the hypothesis that $B_+$ is a linear isomorphism, as considered in this case.
\item[Case 2:] $B_+$ has a non-trivial kernel. Since $B_+$ is surjective and $Q=0$, we are in the degenerate case 0.1 of the case analysis made in Subsection~\ref{subsection:case-analysis}, in which the left-hand side of~\eqref{eq:CDP-Lebesgue-BL} is always infinite and thus~\eqref{eq:CDP-Lebesgue-BL} does not admit extremizers.
\end{enumerate}

\section{Dual form of inverse Brascamp-Lieb inequalities}\label{sec:dual-inverse}
The transportation technique that we have used in Section \ref{sec:proof-main} in order to prove inverse Brascamp-Lieb inequalities follows the one used by the first named author in \cite{barthe-inventiones}. In this reference, the method is actually proved to establish two inequalities:
\begin{itemize}
\item The classical multilinear Brascamp-Lieb inequality, of the form
\[ \int_{ H} \prod_{i=0}^m f_i(B_ix)^{c_i} dx \le C_{BL} \prod_{i=0}^m \left( \int_{H_i} f_i \right) ^{c_i},\]
\item The ``dual'' Brascamp-Lieb inequality,
\[ \int_{ H}^* \sup_{\sum_i c_i B_i^*x_i=x}\prod_{i=0}^m f_i(x_i)^{c_i} dx \ge C_{DBL} \prod_{i=0}^m \left( \int_{H_i} f_i \right) ^{c_i}.\]
\end{itemize}
For both inequalities, the optimal constant is obtained by inspecting centered Gaussian functions. Also the only relevant indices are $c_i\in (0,1]$. Moreover it is possible to introduce a kernel by fixing $c_0=1$ and $f_0$ to be a specific Gaussian functions, and then to consider the best constant for arbitrary non-negative integrable functions $f_1,\ldots,f_m$.

Let us reproduce the proof of Theorem \ref{theo:gauss-mini} of the inverse Brascamp Lieb inequality 
\[ \int_{ H} \prod_{i=0}^{m+1} f_i(B_ix)^{c_i} dx \ge C_{IBL} \prod_{i=0}^{m+1} \left( \int_{H_i} f_i \right) ^{c_i},\]
but choosing  the functions $g_1,\ldots, g_m$ to be arbitrary  (we omit here to repeat the regularity and support assumptions, that can be achieved by approximation. Recall that $c_0=1=-c_{m+1}$ and that $f_0, g_0, f_{m+1},g_{m+1}$ are specific Gaussian functions, which model Gaussian kernels).
With our notation,
\[J(f_1,\ldots,f_m)=\frac{\int_H e^{-\pi \langle Q_+B_0x,B_0x\rangle+\pi \langle Q_-B_{m+1}x,B_{m+1}x\rangle} \prod_{k=1}^m f_k(B_kx)^{c_k}dx}{ \prod_{k=1}^m \left( \int_{H_k} f_k\right)^{c_k}}.\]
In view of \eqref{eq:dual-inverse-proof} we also set ($*$ standing for inner integral):
\[K(g_1,\ldots,g_m)=\frac{ \int_{*,H} \inf_{\sum c_k B_k^*y_k=y} e^{-\pi \langle Q_+^{-1}y_0,y_0\rangle+\pi \langle Q_-^{-1}y_{m+1},y_{m+1}\rangle} \prod_{k=1}^m g_k(y_k)^{c_k}dy}{ \prod_{k=1}^m \left( \int_{H_k} g_k\right)^{c_k}}.\]
The above transportation argument, up to \eqref{eq:dual-inverse-proof}
 yields $ J(f_1,\ldots,f_m)\ge D^{-1} K(g_1,\ldots,g_m)$ for all functions, hence
 $\inf J\ge D^{-1} \sup K$. 
 However, we have seen in \eqref{def:BL-constant} that $\inf_{\mathcal{CG}}J=D^{-\frac12}$. The conclusion of the argument after \eqref{eq:dual-inverse-proof} can be rephrased as $\sup_{\mathcal{CG}}K=D^{\frac12}$. Therefore 
 \[ \sqrt D = D \inf_{\mathcal{CG}}J\ge D \inf J \ge \sup K\ge \sup_{\mathcal{CG}}K=\sqrt D.\]
 In particular $\sup K =\sup_{\mathcal{CG}}K$. This means that under the non-degeneracy hypothesis of Theorem \ref{theo:gauss-mini}, the best constant 
in the following inequality (which can be called dual inverse Brascamp-Lieb)
is obtained by inspecting centered Gaussian functions only: for all $g_1,\ldots,g_m$,
\[ \int_{*,H} \inf_{\sum c_k B_k^*y_k=y}  \prod_{k=0}^{m+1} g_k(y_k)^{c_k}dy\le C_{DIBL} \prod_{k=1}^m \left( \int_{H_k} g_k\right)^{c_k}.\]
Let us state the simplest examples of these four inequalities: no kernel, two functions, all maps being the identity: for all $f,g:\mathbb R^n\to \mathbb R^+$ with $\int f\in (0,+\infty)$:

If $\lambda\in (0,1)$,
 \[ \int f(x)^\lambda g(x)^{1-\lambda}dx \le \left(\int f\right)^\lambda  \left(\int g\right)^{1-\lambda}\le \int^* \sup_{\lambda a+(1-\lambda) b=x}  f(a)^\lambda g(b)^{1-\lambda}\;dx \]

If $\lambda\in \mathbb R\setminus [0,1]$,
\[\int f(x)^\lambda g(x)^{1-\lambda}dx \ge \left(\int f\right)^\lambda  \left(\int g\right)^{1-\lambda}\ge \int_* \inf_{\lambda a+(1-\lambda) b=x}  f(a)^\lambda g(b)^{1-\lambda}\; dx 
\]
The reader has recognized the inequalities of H\"older, Pr\'ekopa-Leindler and the inverse H\"older inequality. The fourth inequality seems novel. In this very simple situation, the inequalities for $\lambda \in \mathbb R\setminus [0,1]$ can be deduced from the ones for $\lambda\in (0,1)$ by rearranging the terms.

\section{Interpolation}\label{sec:interpolation}

We have proved that the best constant in inverse Brascamp-Lieb inequalities can be computed using centered Gaussian functions, apart from some degenerate situations. In the rest of the paper, we address the question of positivity
of this optimal constant. More precisely, given a quadratic form $\mathcal Q$ and the geometric data $B=(B_k)_{k=1}^m$, our aim is to characterize exponents $c=(c)_{k=1}^m$ for which a non-trivial inverse Brascamp-Lieb inequality holds, meaning $\inf J_{\mathcal Q,B,c}>0$ where
$$
 J_{\mathcal Q,B,c}(f_1, \ldots, f_m) = \frac{\int_H e^{-\mathcal Q (x)} \prod_{i=k}^m f_k^{c_k}(B_k x) \, dx}{\prod_{k=1}^m \left(\int_{H_k} f_k \right)^{c_k}}.
 $$ 
The analogous question for direct Brascamp-Lieb inequality was solved in full generality by Bennett, Carbery, Christ and Tao \cite{BCCT-structure,BCCT-finiteness}. They gave a description of the set $\mathcal F$ of exponents $c$ for which $\sup J_{\mathcal Q,B,c}<+\infty$. It turns out that this set $\mathcal F$ is convex, which is a simple instance of interpolation of Lebesgue spaces. Actually, this may be proved by mere application of the Cauchy-Schwarz inequality: if $t\in [0,1]$,
 \begin{align*}
 \int  e^{-\mathcal Q } \prod_{k=1}^m f_k^{tc_k+(1-t)d_k}\circ B_k 
 & \le \left(\int e^{-\mathcal Q} \prod_{k=1}^m f_k^{c_k}\circ B_k  \right)^t \left(\int e^{-\mathcal Q } \prod_{k=1}^m f_k^{d_k}\circ B_k  \right)^{1-t} \\
 &\le  (\sup J_{\mathcal Q,B,c})^t (\sup J_{\mathcal Q,B,d})^{1-t}  \prod_{k=1}^m \left(\int_{H_k} f_k \right)^{tc_k+(1-t)d_k}.
\end{align*}

In the setting of inverse inequalities, we did not find a simple interpolation argument as above. Nevertheless
the convexity of the set of non-trivial exponents is still valid, provided one prescribes their signs.
\begin{prop}\label{prop:interpolation}
Let $0\le m^+\le m$, linear surjective maps $B_k \colon H\to H_k$, $1\le k\le m$ and a quadratic form $\mathcal Q \colon H\to \R$.
Assume that $\mathcal Q$ is positive definite on $\ker B_+$ and 
$$ \dim H\ge s^+(\mathcal Q)+\sum_{i=1}^{m^+} \dim H_i.$$
Let $c,d\in (0,+\infty)^{m^+}\times (-\infty,0]^{m-m^+}$ satisfy $\inf J_{\mathcal Q,B,c}>0$ and $\inf J_{\mathcal Q,B,d}>0$. 
Then for any $t\in [0,1]$,
 $$\inf J_{\mathcal Q,B,tc+(1-t) d}>0.$$
\end{prop}

\begin{proof}
 We use Theorem \ref{theo:gauss-mini} (the infimum of $J$ can be computed on centered Gaussians) and the
 explicit calculations on centered Gaussian functions of Subsection \ref{subsec:centered-gaussians}. Let $Q$
 be a self-adjoint linear map such that for all $x\in H$, $\mathcal Q(x)=\pi \langle x,Qx \rangle$. 
 Let 
 $$\mathcal P=\mathcal P_{\mathcal Q,B,m^+}=\big\{ x\in  (0,+\infty)^{m^+}\times (-\infty,0]^{m-m^+}; \inf J_{\mathcal Q,B,x}>0\big\}$$
 denote the set of exponents with prescribed signs, for which a non-trivial inequality holds.
 Given  $c\in (0,+\infty)^{m^+}\times (-\infty,0]^{m-m^+}$, 
 Equation \eqref{def:BL-constant} ensures that $c\in \mathcal P$ 
if and only if 
$$ \sup_{A_k>0} \frac{\det \left( \Big(Q+\sum_k c_k B_k^* A_k B_k \Big)_+ \right)}{\prod_k (\det A_k)^{c_k}} <+\infty.$$
where the supremum is on $k$-tuples of definite positive self-adjoint operators $A_k$ on $H_k$ and for a self-adjoint operator $A$ we denote
\begin{equation}\label{eq:def-A-plus}
  (A)_+ = \begin{cases}
    A & \textup{if $A$ is positive semi-definite,} \\
    0 & \textup{otherwise.}
  \end{cases}
\end{equation}

When $c_k\neq 0$ (which is true at least for $k\le m^+$), we make a change of variables $M_k=|c_k| A_k$ (which is still positive definite). Hence $c\in \mathcal P$ is equivalent to 
\begin{equation}\label{eq:finite-sup1}
\sup_{M_k>0} \frac{\det \left( \Big(Q+\sum_{i\le m^+}  B_i^* M_iB_i-\sum_{j>m^+; \, c_j\neq 0}  B_j^* M_j B_j \Big)_+ \right)}{\prod_k (\det M_k)^{c_k}}<+\infty.
\end{equation}
We claim that the latter is equivalent to 
\begin{equation}\label{eq:finite-sup2}
\sup_{M_k>0} \frac{\det \left( \Big(Q+\sum_{i\le m^+}  B_i^* M_iB_i-\sum_{j>m^+}  B_j^* M_j B_j \Big)_+ \right)}{\prod_k (\det M_k)^{c_k}}<+\infty.
\end{equation}
The fact that \eqref{eq:finite-sup1} implies \eqref{eq:finite-sup2} is easy: $A\mapsto (A)_+$
is a non-decreasing map on self-adjoint operators and the operator in the determinant of \eqref{eq:finite-sup2} differs from the one of \eqref{eq:finite-sup1} by additional
negative definite terms.
 
 To show that \eqref{eq:finite-sup2} implies \eqref{eq:finite-sup1}, it is sufficient to let $M_j>0$ tend to $0$  for all indices $j>n$ such that $c_j=0$ (hence $(\det M_j)^{c_j}=1$).
 This requires continuity properties of the numerator. Observe that $A\mapsto (A)_+$ is not
 continuous (if $A$ is not positive definite but is semi-definite positive, then $(A)_+=A$
 but for $\varepsilon>0$, $A - \varepsilon \Id$ is not positive semi-definite, so that $\lim_{\varepsilon \to 0^+}  (A - \varepsilon \Id)_+=0$).
Fortunately, we may conclude by using the continuity of $A\mapsto \det(A_+)$ (which is 
easy to verify: let  $(A_n)$ be self-adjoint operators tending to $A$. If $A$ is positive definite, then so is $A_n$ for $n$ large enough, and continuity of the determinant allows to conclude.
If there is $v\neq 0$ with $\langle Av,v\rangle <0$ then this is eventually true for $A_n$
and $(A)_+=(A_n)_+=0$. If $A$ is semi-definite positive but not definite positive, then 
$\det A=0$. If $A_n$ is not positive
definite  then $\det(A_n)_+ = 0$, while if $A_n$ is positive definite then $\det (A_n)_+=
\det A_n$ tends to $\det A=0$ when $n$ increases).
 
\medskip
We are ready to show the convexity of $\mathcal P$. Let $c,d\in \mathcal P$. Using that \eqref{eq:finite-sup2} characterizes membership to $\mathcal P$, we know that there exists $K_c$ and $K_d$ in $\R^+$ such that
for all $M_k>0$,
$$\det \left( \Big(Q + \sum_{i\le m^+}  B_i^*M_iB_i -\sum_{j>m^+}  B_j^*M_jB_j  \Big)_+ \right)$$
is upper bounded by $K_c \prod_k (\det M_k)^{c_k}$ and by  $K_d \prod_k  (\det M_k)^{d_k}$.
Therefore, for any $t\in[0,1]$,
$$%\det \left( \Big(\sum_{i\le n}  \mu_i \, u_i \otimes u_i-\sum_{j>n}  \mu_j \,u_j\otimes u_j \Big)_+ \right)
\det \left( \Big(Q + \sum_{i\le m^+}  B_i^*M_iB_i -\sum_{j>m^+}  B_j^*M_jB_j  \Big)_+ \right)
\le K_c^{t}K_d^{1-t} \prod_k (\det M_k)^{tc_k+(1-t)d_k}.$$
The above arguments show that this implies that $tc+(1-t)d\in \mathcal P$. Hence $\mathcal P$ is convex.
\end{proof}

%%%%%%%%%%%%%%%%%%%%%%%%%%%%%%%%%%%%%%%%%%%%%%%%%%%%% 
\section{Positivity in the rank one case}
\label{sec:finiteness-1}
We study the positivity of the optimal constant in inverse Brascamp-Lieb inequalities, when for all   $k=1,\ldots,m$,  $\dim H_k=1$, in terms of the coefficients $(c_k)_{k=1}^m$. In this case a very complete solution can be given, based on a rather straightforward argument.  
For concreteness, we may identify  each $H_k$ with $\R$,  
%$H$ to the Euclidean space $\R^n$ ($n=\dim H$) 
and  we may find non-zero vectors $u_k\in H$ such that $B_k x=\langle x, u_k\rangle$ for all $x\in H$.
As before, we consider 
exponents with prescribed signs: given $m^+\le m$ the first $m^+$ coefficients are positive, while the other ones are non-positive.
In addition, we work under the hypotheses of Theorem \ref{theo:gauss-mini}. 

For two integers $k, l$ let $[\![k,l]\!] = \{ k, k+1, \ldots, l\}$ and $]\!]k,l]\!] = \{ k+1, k+2, \ldots, l\}$. 
 
 \subsection{No kernel}
 We start with the case when  $\mathcal Q=0$. In the present setting, the hypotheses of Theorem \ref{theo:gauss-mini} simply mean that the map $B_+$ is one-to-one, that is $(u_1,\ldots, u_{m^+})$ is a basis of $H$ (which can be identified to $\R^{m^+}$).
 For every family $(f_k)_{k=1}^m$ of non-negative integrable functions on $\R$ with
 positive integrals, the functional of interest is
$$J_{u,c}(f_1,\ldots,f_m)= \frac{\int_{H} \prod_{k=1}^m f_k\big(\langle x, u_k\rangle\big)^{c_k} dx}{ \prod_{k=1}^m \left( \int_{\mathbb R} f_k\right)^{c_k}} \cdot$$

\begin{theorem}\label{theo:positivity-rank1}
Let $m\ge m^+\ge 1$ and $u_1,\ldots u_m$ be non-zero vectors in $H$. Assume also that $(u_1,\ldots, u_{m^+})$ is
a basis of $H$. For $i\le m^+<j$, we write that $i\sim j$ if $u_j$ has a non-zero $i$-th coordinate in the 
latter basis.
Consider the positivity domain 
$$\mathcal P_{m^+}((u_k)_{k=1}^m)=\big\{c\in (0,+\infty)^{m^+}\times (-\infty,0]^{m-m^+}; \; \inf J_{u,c}>0\big\}.$$
For any set $S \subseteq [\![1,m]\!]$, denote by $\II_S$ the vector in $\{0,1\}^m$ with $i$-th coordinate equal to 1 if and only if $i\in S$. 
Then
\begin{eqnarray*}
\mathcal P_{m^+}((u_k)_{k=1}^m)&=& \II_{[\![1,m^+]\!]}+\mathrm{Pos}\Big(\big\{\II_{\{i\}}-\II_{\{j\}};\; i\sim j \big\}\Big)\\
&=& \Big\{c\in [1,+\infty)^{m^+}\times (-\infty,0]^{m-m^+};\quad  \sum_k c_k=m^+ \; \mathrm{and} \\
&& \quad   \mathrm{ \, for\,  all} \;
 S\subset [\![1,m^+]\!], \quad \sum_{i\in S} (c_i-1) \le 
  \sum_{j;\; S\sim j} |c_j|\Big\}\\
  &=& \Big\{c\in [1,+\infty)^{m^+}\times (-\infty,0]^{m-m^+}; \quad \sum_k c_k=m^+ \; \mathrm{and}  \\
&& \quad  \mathrm{ \, for\,  all} \;
 T\subset ]\!]m^+,m]\!], \quad \sum_{j\in T} |c_j| \le 
  \sum_{i;\; i\sim T} (c_i-1)\Big\},  
\end{eqnarray*}
where $S\sim j$ means that there exists $i\in S$ with $i\sim j$, and $i \sim T$ means that there exists $j \in T$ with $i \sim j$, and $\mathrm{Pos}(A)$ is the positive hull of $A$.
\end{theorem}

The above description of $\mathcal{P}_{m^+}((u_k))$ as a positive hull can be phrased in terms of mass transportation. The following interpretation will be justified and applied in the course of the proof: consider a bipartite graph $G$ on the sets $I=[\![1, m^+]\!]$ 
and $J=]\!]m^+,m]\!]$, with an edge between $i\in I$ and $j\in J$ if 
$i\sim j$ (i.e. $u_j$ has a non-zero coordinate on $u_i$, when decomposed 
in the basis $(u_1,\ldots, u_{m^+})$). Then $c\in \mathcal{P}_{m^+}((u_k))$ if and only if
one can transport the measure $\sum_{i=1}^{m^+} (c_i-1)\delta_i$ onto 
$\sum_{j=1+m^+}^m |c_j| \delta_j$ by moving the mass along the graph $G$.

\begin{proof}
 For shortness, we  write $\mathcal P$ for the positivity domain $\mathcal{P}_{m^+}((u_k)_{k=1}^m)$.

First part: Let us start with $c\in \mathcal P$ and draw consequences of this fact.
Since $\inf J_{((u_k),(c_k))}>0$, in particular the infimum on centered Gaussian functions
is positive. Using the calculations of the previous section, this can be stated as follows:
there exists $D<+\infty$ such that for all $\lambda_1,\ldots,\lambda_m>0$, it holds
\begin{equation}\label{eq:ineg-gauss-rank1}
 D \prod_k \lambda_k^{c_k} \ge \det \left( \Big(\sum_k c_k \lambda_k \, u_k \otimes u_k \Big)_+ \right),
\end{equation}
where $(\cdot)_+$ is defined as in~\eqref{eq:def-A-plus}. Recall that $(u\otimes u)(x)=\langle x,u\rangle u$.
Our goal is to extract information on $c$ from  \eqref{eq:ineg-gauss-rank1}. Since
it is pointless when the map inside the determinant is non positive definite, we first 
look for values $(\lambda_k)$ for which the inequality has a non-trivial content.

It is convenient to work with $(\lambda_k)$ satisfying the following  inequality
\begin{equation}\label{eq:strong-pos-rank1}
\sum_{i\le m^+} c_i \lambda_i \, u_i \otimes u_i \ge  2 \sum_{j>m^+} |c_j| \lambda_j \, u_j \otimes u_j.
\end{equation}
Indeed, when the above is satisfied and  since $A\ge 2B$ implies $A-B\ge A/2$, it follows that 
$$\sum_k c_k \lambda_k \, u_k \otimes u_k \ge \frac12 \sum_{i\le m^+} c_i \lambda_i u_i\otimes u_i.$$
The map on the right-hand side is positive definite, therefore we may deduce from \eqref{eq:ineg-gauss-rank1}  that 
\begin{equation}\label{eq:ineg-gauss2-rank1}
2^{m^+} D \prod_k \lambda_k^{c_k} \ge \det  \Big(\sum_{i\le m^+} c_i \lambda_i \, u_i \otimes u_i \Big) = \det(u_1,\ldots,u_{m^+})^2 \prod_{i\le m^+} c_i \lambda_i.
\end{equation}
Keeping in mind that \eqref{eq:strong-pos-rank1}$\Longrightarrow\eqref{eq:ineg-gauss2-rank1}$, let us provide numbers $(\lambda_k)$ for which  \eqref{eq:strong-pos-rank1} is verified. For $j>m^+$, we denote by $\alpha_i(j)$ the $i$-th coordinate of $u_j$ in the basis
$(u_1,\ldots, u_n)$. Observe that by definition $i\sim j$ means $\alpha_i(j)\neq 0$. Hence for all $j>m^+$, 
$$ u_j=\sum_{i;\; i\sim j} \alpha_i(j)u_i.$$
For any vector $v\in H$, by Cauchy-Schwarz,
$$ \langle v,u_j\rangle^2 \le \left(\sum_{i;\; i\sim j} \alpha_i(j)^2 \right) \left(\sum_{i;\; i\sim j}  \langle v,u_i\rangle^2\right).$$
Set $K:=\max_j\sum_{i;\; i\sim j} \alpha_i(j)^2 $.
We have proved that for $j>m^+$,
$$u_j\otimes u_j \le K \sum_{i;\; i\sim j} u_i\otimes u_i.$$
Summing upon $j>m^+$ and interchanging summations in $j>m^+$ and in $i\le m^+$ yield
\begin{eqnarray}\label{eq:ineq-tensor-rank1}
2\sum_{j>m^+} |c_j| \lambda_j u_j\otimes u_j &\le & 2K \sum_{i\le m^+} \left(  \sum_{j; \; i\sim j} |c_j|\lambda_j \right) u_i\otimes u_i.
\end{eqnarray}

Let   $q\in \R^+$ and let  $a\in \R^m$ satisfy
\begin{equation}\label{eq:condition-proof-rank1}
 i\sim j \Longrightarrow a_i\ge a_j\,  .
\end{equation}  
Set $K':= 2K \max_i (\sum_{j; \; i\sim j} |c_j|)/c_i$, and 
$$ \lambda_j:= e^{qa_j} \quad\mathrm{for}\quad j>m^+ \qquad\mathrm{and}\qquad
 \lambda_i:= K' e^{qa_i} \quad\mathrm{for}\quad i\le m^+.$$
Then Inequality \eqref{eq:ineq-tensor-rank1} readily implies 
%Then for all $v\in \R^n$, $2\sum_{j>n} |c_j| \lambda_j\langle v,u_j\rangle^2 
 % \le \sum_{i\le n} c_i \lambda_i \langle v,u_i\rangle^2$. This means
   that for this choice of $\lambda$,  \eqref{eq:strong-pos-rank1} is verified. As we have already seen,
  this implies that \eqref{eq:ineg-gauss2-rank1} applies and gives
  $$ 2^{m^+} D \Big(\prod_{i\le m^+} (K')^{c_i}\Big) e^{q \sum_k a_kc_k} \ge  \det(u_1,\ldots,u_{m^+})^2 \Big(\prod_{i\le m^+} (c_iK')\Big)  e^{q\sum_{i\le m^+}a_i}.$$
 For $q$ tending to $+\infty$ the last inequality implies that 
$
 \sum_k a_kc_k \ge \sum_{i\le m^+}a_i.
$ 
 Recall that this was proved assuming \eqref{eq:condition-proof-rank1} (and that $c$ is in the positivity domain of the functional $J$). Summarizing, for $c$ in the positivity domain, and all 
 $a\in \R^m$,
  \begin{equation}\label{eq:forms-rank1}
 \left(i\sim j \Longrightarrow a_i\ge a_j \right) \Longrightarrow \sum_{i\le  m^+} (c_i-1)a_i\ge \sum_{j>m^+} |c_j|a_j.
 \end{equation}
 Let us draw some consequences of this property of coefficients $c$ in the positivity domain,
 by making appropriate choices for $a$:
 \begin{itemize}
 \item  Choosing  vectors $a$ with all equal coordinates  gives $\sum_k c_k=m^+$, known as the homogeneity condition.
\item  For any $i\in [\![1,m^+]\!]$, we may choose  $a=\II_{\{i\}}$ and get  $c_i\ge 1$.
\item For $T\subset  ]\!]m^+,m]\!]$, we may define a vector $a$ as follows:
 $a_j=1$ for $j\in T$, $a_j=0$ for  $j \in ]\!]m^+,m]\!]\setminus T $,
and for $i\in  [\![1,m^+]\!]$, set $a_i=1$ if $i\sim T$ and $a_i=0$ otherwise. It is plain 
that $i\sim j \Longrightarrow a_i\ge a_j$, so this vector is admissible and we can deduce
that 
$ \sum_{j\in T} |c_j| \le 
  \sum_{i;\; i\sim T} (c_i-1)$.
\item In a symmetric way, for $S\subset [\![1,m^+]\!]$, we may define an admissible 
vector $a$ as follows: for $i\le m^+$, $a_i=-1$ if $i\in S$ and $a_i=0$ otherwise;
for $j>m^+$, $a_j=-1$ if $S\sim j$ and $a_j=0$ otherwise.
Plugging this vector in  \eqref{eq:forms-rank1} yields 
$\sum_{i\in S} (c_i-1) \le 
  \sum_{j;\; S\sim j} |c_j|$.
\end{itemize} 
 
 Next give a dual interpretation of \eqref{eq:forms-rank1} (where the right-hand side 
 inequality is taken in the form $\sum_k c_k a_k \ge \sum_{i\le m^+} a_i$): for every 
 vector $a\in \R^m$, if for all $i\le m^+<j$ such that $i\sim j$, it holds $\langle a, \II_{\{i\}}-
  \II_{\{j\}}\rangle \ge 0$, then $\langle a, c-\II_{ [\![1,m^+]\!]}\rangle\ge 0$. Equivalently,
  no linear hyperplane can separate the vector $c-\II_{ [\![1,m^+]\!]} $ from the family 
  of vectors $  (\II_{\{i\}}- \II_{\{j\}})_{i\sim j}$. By the Hahn-Banach theorem, this implies
  that $c-\II_{ [\![1,m^+]\!]}$  belongs to the convex cone generated by the vectors $  (\II_{\{i\}}- \II_{\{j\}})_{i\sim j}$. This concludes the first part of the proof.
  
  \bigskip
  Second part: let us show that $\II_{ [\![1,m^+]\!]}+\mathrm{Pos}( (\II_{\{i\}}- \II_{\{j\}})_{i\sim j})$ is included in the positivity domain $\mathcal{P}$ of the functional $J$. Since Proposition~\ref{prop:interpolation}  ensures that $\mathcal{P}$ is convex, it is enough to show that for all $i\le m^+<j$ such that $i\sim j$, it contains the half-line $\II_{ [\![1,m^+]\!]}+\R^+ (\II_{\{i\}}- \II_{\{j\}})$.
  %  Let us simply write $\mathcal P$ for  $\mathcal{P}((u_k)_{k=1}^m)$.
  
   Let us start with observing that $\II_{[\![1,m^+]\!]}\in \mathcal P$. Indeed, for all measurable
   $f_k:\R\to \R^+$, by the change of variables $y_i:=\langle x,u_i\rangle$, $i\in [\![1,m^+]\!]$
   and Fubini's theorem 
   $$\int_{H} \prod_{i=1}^{m^+} f_i(\langle x,u_i\rangle) \,dx =
    |\det((u_i)_{i=1}^{m^+}) |^{-1} \int_{\R^{m^+}} \prod_{i=1}^{m^+} f_i(y_i) \,dy
    = |\det((u_i)_{i=1}^{m^+}) |^{-1} \prod_{i=1}^{m^+} \int_{\R}  f_i. $$
  Another basic ingredient, which is actually the simplest instance of the reverse inequalities
  we are investigating, is the reverse H\"older inequality: for $\varepsilon\ge 0$ and $f,g$ non-negative measurable functions on $\R$, 
  $$\int  f^{1+\varepsilon} g^{-\varepsilon} \ge \left( \int f\right)^{1+\varepsilon}\left( \int g \right)^{-\varepsilon}.$$
 We are ready to show 
  that for $i_0\sim j$, $\II_{ [\![1,m^+]\!]}+\R^+ (\II_{\{i_0\}}- \II_{\{j\}})\subset \mathcal P$.
  Let $\varepsilon\ge 0$. Then, using that $u_j=\sum_i \alpha_i(j) u_i$ with $\alpha_{i_0}(j)\neq 0$, changing variables by $y_i:=\langle x,u_i\rangle$, $i\in [\![1,m^+]\!]$ as above
  \begin{eqnarray*}
 &&  \int_{H} \Big(\prod_{i\in [\![1,m^+]\!]\setminus \{i_0\} } f_i(\langle x,u_i\rangle) \Big) f_{i_0}(\langle x,u_{i_0}\rangle)^{1+\varepsilon} f_j(\langle x,u_j\rangle)^{-\varepsilon} dx \\
 &=&  \int_{H} \Big(\prod_{i\in [\![1,m^+]\!]\setminus \{i_0\} } f_i(\langle x,u_i\rangle) \Big) f_{i_0}(\langle x,u_{i_0}\rangle)^{1+\varepsilon} f_j\big(\sum_{i\le m^+} \alpha_i(j)\langle x,u_i\rangle\big)^{-\varepsilon} dx \\
  &=&  |\det((u_i)_{i=1}^{m^+}) |^{-1} \int_{\R^{m^+}} \Big(\prod_{i\in [\![1,m^+]\!]\setminus \{i_0\} } f_i(y_i) \Big) f_{i_0}(y_{i_0})^{1+\varepsilon} f_j\big(\sum_{i\le m^+} \alpha_i(j)y_i\big)^{-\varepsilon} dy.
  \end{eqnarray*} 
 Applying inverse H\"older in the variable $y_{i_0}$ and using  $\alpha_{i_0}(j)\neq 0$, we deduce that for any $(y_i)_{i\in [\![1,m^+]\!]\setminus \{i_0\} }$
 $$ \int_{\R}    f_{i_0}(y_{i_0})^{1+\varepsilon} f_j\big(\sum_{i\le m^+} \alpha_i(j)y_i\big)^{-\varepsilon} dy_{i_0} \ge \left( \int f_{i_0}\right)^{1+\varepsilon}  \left(\frac{1}{\alpha_{i_0}(j)} \int f_{j} \right)^{-\varepsilon} $$
  Plugging this estimate in the latter integral over $\R^n$ we arrive at
\begin{eqnarray*}  
 && \int_{H} \Big(\prod_{i\in [\![1,m^+]\!]\setminus \{i_0\} } f_i(\langle x,u_i\rangle) \Big) f_{i_0}(\langle x,u_{i_0}\rangle)^{1+\varepsilon} f_j(\langle x,u_j\rangle)^{-\varepsilon} dx \\
  &\ge &  |\det((u_i)_{i=1}^{m^+}) |^{-1} \alpha_{i_0}(j)^{-\varepsilon} \prod_{i\in [\![1,m^+]\!]\setminus \{i_0\} } \left(\int f_i\right) \times \left( \int f_{i_0}\right)^{1+\varepsilon}  \left(\int f_{j} \right)^{-\varepsilon} .
  \end{eqnarray*}
This inequality proves that   $\II_{ [\![1,m^+]\!]}+\varepsilon (\II_{\{i_0\}}- \II_{\{j\}})\in \mathcal P$.

\bigskip
Third part: Our final task is to show that the descriptions of $\mathcal P$ is terms of inequalities coincides with the one in terms of positive hull. This can be done ``by hands,'' but
we present a neat argument in terms of transportation plans.
We have shown that $c\in \mathcal P$ is equivalent to $c-\II_{[\![1,m^+]\!]}\in \mathrm{Pos}((\II_{\{i\}}-\II_{\{j\}})_{i\sim j}).$ The latter is equivalent to the existence of non-negative 
coefficients $(\gamma_{i,j})_{i\sim j}$ such that  $c-\II_{[\![1,m^+]\!]}=\sum_{i\sim j}\gamma_{i,j}
(\II_{\{i\}}-\II_{\{j\}})$, or in coordinates:
\begin{eqnarray*}
\mathrm{for} \; i\le m^+, && c_i-1=\sum_{j;\; i\sim j} \gamma_{i,j},\\
\mathrm{for} \; j> m^+, && |c_j|=\sum_{i;\; i\sim j} \gamma_{i,j}.
\end{eqnarray*}
This can be interpreted as a coupling, or transportation plan between  measures: $\gamma_{i,j}$ represents
the amount of mass which is transported from $i$ to $j$, and such a shipping is allowed
only if $i\sim j$. Therefore $c$ belongs to $\mathcal P$ if and only if it is possible to transport
the measure $\sum_{i\le m^+} (c_i-1)\delta_i$ to the measure  $\sum_{j>m^+} |c_j| \delta_j$ (or vice-versa), while 
carrying mass only between points which are in relation for $\sim$. This questions of existence of transport with constraints is well know. Its solution is given in the following classical lemma and it allows to complete the proof. Observe that the indices $i\le m^+$ and 
$j>m^+$ play symmetric role for the transportation problem, which leads to two different
description of $\mathcal P$ in terms of inequalities.
\end{proof}

\begin{lemma}\label{lem:transport}
Let $I$ and $J$ be disjoint finite sets. Let $E\subset I\times J$ and consider the bipartite graph $(I,J;E)$.  Let $(\alpha_i)_{i\in I}$ and $(\beta_j)_{j\in J}$ be non-negative numbers.
Then there exists a transportation plan along the graph  between $\sum_{i\in I} \alpha_i \delta_i$ and
 $\sum_{j\in J} \beta_i \delta_i$ if and only if:
\begin{equation}\label{eq:transport-constraint}
 \sum_{i\in I} \alpha_i=\sum_{j\in J} \beta_j  \mbox{ and for all } S\subset I;\; \sum_{i\in S} \alpha_i\le \sum_{j; \; S\sim j} \beta_j\, .
\end{equation}
\end{lemma}

\begin{proof}
The condition means that the origin and target measures have same total mass, and that the mass of  any subset of the origin set is not larger than the mass for the target measure of the 
set of its neighbors.

Showing that the existence of a transport plan implies the above inequalities is straightforward, and actually not the direction we need for the previous theorem, so we omit it.

Assume that \eqref{eq:transport-constraint} is verified. Let us build a weighted graph 
$G$ by enriching $(I,J;E)$ as follows: we assign to every existing edge ($i\sim j$) a
weight $w:=1 + \sum_{i \in I} \alpha_i$; 
we also add a vertex $A$ and  connect it to each $i\in I$ with a weight $\alpha_i$
on the edge; eventually  we add another  vertex $B$ and  connect to each $j\in J$ with a weight $\beta_i$.

  Our goal is to show 
that the maximal flow between $A$ and $B$ is equal to  $\sum_{i\in I} \alpha_i$ (which means  that all the mass from $I$ can be transported to $B$ along the graph, and since
$\sum_i \alpha_i=\sum_j \beta_j$ all the target mass is reached). 
By the Max flow-Min cut theorem (see e.g. \cite{schrijver}), it is enough to show that the minimal weight of a cut
separating $A$ and $B$ is equal to $\sum_{i\in I} \alpha_i$ (with corresponds to cutting
all the edges incident to $A$).
 
 Let us study a minimal cut. First, since the edges between $I$ and $J$ have weight $w>\sum_{i\in I} \alpha_i$, they are not in a minimal cut.
 Such a cut is thus as follows: there are subsets $S\subset I$ and $T\subset J$ such that 
 $S^c=I\setminus S$ and $T^c=J\setminus T$ are not connected, and one cuts the edges
 between $A$ and $I$ and the ones between $B$ and $J$.
The weight of this cut is 
$$\sum_{i\in S} \alpha_i+\sum_{j\in T}\beta_j.$$
Our goal is to bound this weight from below as follows,
$$  \sum_{i\in S} \alpha_i+\sum_{j\in T}\beta_j\ge \sum_{i\in I} \alpha_i.$$
This is equivalent, after canceling the terms appearing twice, to 
 $\sum_{i\in S^c} \alpha_i\le \sum_{j\in T}\beta_j.$
 This is indeed true: by hypothesis $\sum_{i\in S^c} \alpha_i\le \sum_{j; \;S^c\sim j }\beta_j.$
 But since $S^c$ and $T^c$ are not connected, $S^c\sim j$ implies that $j\in T$, and the latter sum is at most $\sum_{j\in T}\beta_j$, as claimed.

\end{proof}

\subsection{With a kernel}
Here we consider in addition a kernel $e^{-\mathcal Q}$, with the restriction that $s^+(\mathcal Q)$, $s^-(\mathcal Q)\le 1$. As above we work under the assumptions $\mathcal Q_{|\ker B_+}$ positive definite and $\dim H\ge s^+(\mathcal Q)+\sum_{i=1}^{m^+} \dim H_i$, for which a convenient equivalent form is given in Lemma~\ref{lem:decompose-Q}.
In our setting, they can be rephrased as follows (we introduce a small twist with respect to the decomposition 
in the lemma, namely a dilation which allows for a more concrete decomposition):
there are vectors $u_0, u_{m+1}\in H$ such that for all $x\in H$,
$$\mathcal{Q}(x)=\pi\langle x,u_0\rangle^2-\pi \langle x,u_{m+1}\rangle^2.$$
Note that these two vectors may be equal to zero (e.g. if $Q$ is non-positive, $u_0=0$).
Moreover setting $B_0x=\langle x,u_0\rangle$ and $B_{m+1}x=\langle x,u_{m+1}\rangle$,
we know that $\ker B_+\subset \ker B_{m+1}$ and that 
$B_{0+}:H\to B_0H\times B_1H\times \cdots \times B_{m^+}H$ is one to one.
The former is equivalent to $\bigcap_{i=1}^{m^+} u_i^\bot \subset u_{m+1}^\bot$, that is 
\begin{equation}\label{hyp:m+1}
u_{m+1}\in \mathrm{vect}\{u_1,\ldots,u_{m^+}\},
\end{equation}
 while the latter means that:
\begin{itemize}
\item either $u_0=0$ and $(u_1, \ldots, u_{m^+})$ is a basis of $H$,
\item or $(u_0, u_1, \ldots, u_{m^+})$ is a basis of $H$.
\end{itemize}
In any of the above cases, we denote by $\mathbb U$ the corresponding basis of $H$.
Given $i\in I=[\![0,m^+]\!]$ and $j\in J = ]\!] m^+, m+1]\!]$, we write $i\sim j$
if $u_j$, once decomposed in the basis $\mathbb U$, has a positive coordinate on the vector $u_i$ of the basis. 
This relation creates a bipartite graph $G$ on $I$ and $J$. Note that  $m+1$ is an isolated
vertex of the graph when $u_{m+1}=0$, and so is $0$ when $u_0=0$. 
The functional of interest is 
$$J_{\mathcal Q,(u_k)_{k=1}^m,c}(f_1,\ldots,f_m)= \frac{\int_{H} e^{-\pi\langle x,u_0\rangle^2+\pi \langle x,u_{m+1}\rangle^2} \prod_{k=1}^m f_k\big(\langle x, u_k\rangle\big)^{c_k} dx}{ \prod_{k=1}^m \left( \int_{\mathbb R} f_k\right)^{c_k}} \cdot$$
Now comes a description of its positivity domain
$$\mathcal P_{m^+}(\mathcal Q,(u_k)_{k=1}^m)=\big\{c\in (0,+\infty)^{m^+}\times (-\infty,0]^{m-m^+}; \; \inf J_{\mathcal Q, (u_k)_{k=1}^m,c}>0\big\}.$$
\begin{theorem}\label{theo:positivity-rank1-Q}
With the above notation and  hypotheses, 
\begin{eqnarray*}
\mathcal P_{m^+}(\mathcal Q,(u_k)_{k=1}^m)
&=& \II_{[\![1,m^+]\!]}+\mathrm{Pos}\Big(
 \big\{\II_{\{i\}};\; i\sim m+1 \big\}\cup \big\{-\II_{\{j\}};\; 0\sim j \big\} \\
&&\qquad \qquad \qquad \cup \big\{\II_{\{i\}}-\II_{\{j\}};\; 1\le i\sim j\le m \big\}
\Big)\\
&=& \Big\{c\in [1,+\infty)^{m^+}\times (-\infty,0]^{m-m^+};  \\
&& \quad  \mathrm{ \, for\,  all} \;
 S\subset [\![1,m^+]\!] \mathrm{ \, with \, } S\not\sim m+1, \quad \sum_{i\in S} (c_i-1) \le 
  \sum_{j;\; S\sim j} |c_j|, \\
  && \quad \mathrm{and  \, for\,  all} \;
 T\subset ]\!]m^+,m]\!] \mathrm{\, with \, } 0\not\sim T, \quad \sum_{j\in T} |c_j| \le 
  \sum_{i;\; i\sim T} (c_i-1)\Big\},\\
&=& \mathrm{Proj}_{\R^{[\![1,m]\!]}} \left( \mathcal P_{1 + m^+}\big(u_0,u_1,\ldots,u_m,u_{m+1}\big)   \right).
\end{eqnarray*}
\end{theorem}
Let us comment on this statement before proving it. The notation of the last line, involving a projection
and an extended use of the notation of the positivity domain in the case of no kernel (if $u_0$ or $u_{m+1}$
is zero, just discard it), means the following: $\inf J_{\mathcal Q,(u_k)_{k=1}^m,c}>0$ if and only if there
exists $c_0\ge 1$ and $c_{m+1}\le 0$ and $\varepsilon>0$ such that for all $f_k:\R\to \R^+$ ($k = 0,1,\ldots,m+1$)
integrable and with positive integral:
$$ \int_H f_0(\langle x,u_0\rangle)^{c_0} f_{m+1}(\langle x,u_{m+1}\rangle)^{c_{m+1}} \prod_{k=1}^m f_k\big(\langle x, u_k\rangle\big)^{c_k} dx \ge \varepsilon 
\prod_{k=0}^{m+1} \left( \int_{H_k} f_k\right)^{c_k} \cdot$$
Here $H_k$ is the image of $H$ by $x\mapsto \langle x,u_k\rangle$. If for instance $u_0=0$ then $H_0=\{0\}$
and the term $f_0(0)=\int_{H_0} f_0>0$ appears on both sides, and can be discarded.
In other words, the positivity of the constant in the inequality with kernel can be deduced from an inequality
without kernel, by specifying one or two functions to be Gaussian.

The description of the positivity domain as a positive convex hull can also be interpreted in terms of a transportation problem: $c$ is in the positivity domain if and only if one can transport the measure
$\sum_{i=1}^{m^+} (c_i-1) \delta_i$ to $\sum_{j=1+m^+}^m |c_j|\delta_j$ along the bipartite graph $G$ defined 
above, with the help of a source at 0 and of a sink at $m+1$. 

\begin{proof}[Proof of Theorem \ref{theo:positivity-rank1-Q}]
The strategy is the same as for Theorem \ref{theo:positivity-rank1}, so we only explain the changes. We simply
write $\mathcal P$ for the positivity domain. Let us denote by $\mathcal P_1$, $\mathcal P_2$ and $\mathcal P_3$
the three sets appearing in the claim (in the same order).

\smallskip
Let $c\in \mathcal P$. By Theorem \ref{theo:gauss-mini} and Gaussian calculations, there exists $D>0$ such that for all $\lambda_1,\ldots,\lambda_m$, 
\begin{equation}\label{eq:ineg-gauss-rank1-Q}
 D \prod_k \lambda_k^{c_k} \ge \det \left( \Big(\sum_{k=0}^{m+1} c_k \lambda_k \, u_k \otimes u_k \Big)_+ \right),
\end{equation}
where we have set $c_0=1, c_{m+1}=-1, \lambda_0=\lambda_{m+1}=1$ in order to include the terms coming from the kernel.
Our first task is to infer that for every $a\in\R^{[\![0,m+1]\!]}$ satisfying $a_0=a_{m+1}=0$,
\begin{equation}\label{eq:forms-rank1-Q}
 \left(i\sim j \Longrightarrow a_i\ge a_j \right) \Longrightarrow \sum_{i\le  m^+} (c_i-1)a_i\ge \sum_{j>m^+} |c_j|a_j.
 \end{equation}
To do this we look for numbers $b_k$ such that for all $q\ge 0$, Inequality \eqref{eq:strong-pos-rank1}
is verified for $\lambda_k=\lambda_k(q)=b_ke^{qa_k}$. Letting $q$ tend to infinity in the determinant inequality
then yields \eqref{eq:forms-rank1-Q}. The main changes in the argument come from the ``boundary'' conditions
$\lambda_0(q)=\lambda_{m+1}(q)=1$ which force $a_0=a_{m+1}=0$ and $b_0=b_{m+1}=1$.
The strategy is again to choose the $\lambda_k$ such that \eqref{eq:ineq-tensor-rank1} holds. Observe that 
 \eqref{eq:ineq-tensor-rank1} is verified  when for all $i\le m^+$,
 $2K \sum_{j;\; i\sim j} |c_j|\lambda_j \le c_i\lambda_i$. Setting $M:=\max\big(1,2K(m+1)\big)$, we get that a sufficient condition
 to ensure the latter is to have:
 \begin{equation*}
  i\sim j \Longrightarrow M|c_j|\lambda_j \le c_i\lambda_i.
 \end{equation*}
 As already mentioned, we look for $\lambda_k=b_ke^{q a_k}$ and $a$ verifies $i\sim j\Longrightarrow a_i\ge a_j$.
 Hence it is enough to choose $b$ such that 
 \begin{equation}\label{cond:proof-positivity-rank1-Q}
  i\sim j \Longrightarrow M|c_j|b_j \le c_i b_i.
 \end{equation}
 Recall that $c_0b_0=|c_{m+1}|b_{m+1}=1$ so the latter inequality may fail, but thanks to \eqref{hyp:m+1} $0\not\sim m+1$. Eventually, if we choose $(b_k)$ such that $b_i=M/c_i$ for $i\in [\![1,m^+]\!]$ and $|c_j|b_j\le 1/M$ for 
 $j\in ]\!]m^+,m]\!]$, then \eqref{cond:proof-positivity-rank1-Q} is verified.
Thus $c\in \mathcal P$ implies \eqref{eq:forms-rank1-Q}. Using the  Hahn-Banach Theorem,  \eqref{eq:forms-rank1-Q} means $c\in \mathcal P_1$. So we have proved that $\mathcal P\subset \mathcal P_1$.

\medskip
Next we show that $\mathcal P_1\subset \mathcal P_2$ by drawing consequences of \eqref{eq:forms-rank1-Q};
 \begin{itemize}
\item  For any $i\in [\![1,m^+]\!]$, we may choose  $a=\II_{\{i\}}$ and get  $c_i\ge 1$.
\item For $T\subset  ]\!]m^+,m]\!]$ with $0\not\sim T$ we may define a vector $a$ as follows:
 $a_j=1$ for $j\in T$, $a_i=1$ if $i\sim T$ and $a_k=0$ otherwise. 
It readily verifies the hypothesis of \eqref{eq:forms-rank1-Q}, so  we can deduce
that 
$ \sum_{j\in T} |c_j| \le 
  \sum_{i;\; i\sim T} (c_i-1)$.
\item In a symmetric way, for $S\subset [\![1,m^+]\!]$ with $S\not \sim m+1$ we may define an admissible 
vector $a$ as follows:  $a_i=-1$ is $i\in S$;
$a_j=-1$ if $S\sim j$ and $a_k=0$ otherwise. This implies 
$\sum_{i\in S} (c_i-1) \le 
  \sum_{j;\; S\sim j} |c_j|$.
\end{itemize} 

\medskip
Now we prove that $\mathcal{P}_2\subset \mathcal{P}_3$. Let $c\in\mathcal{P}_2$, and set $\alpha_i=c_i-1$ 
for $1\le i\le m^+$ and $\beta_j=|c_j|$ for $m^+<j\le m$. Let us consider the bipartite graph $\tilde G$ on
$I=[\![0,m^+]\!]$ and $J=]\!]m^+,m+1]\!]$ obtained
by adding to $G$ an edge between $0$ and $m+1$.
Let us choose two numbers $\alpha_0$ and $\beta_{m+1}$ such that $\sum_{i\in I} \alpha_i=\sum_{j\in J} \beta_j$
and $\alpha_0, \beta_{m+1}> \sum_{i\in [\![1,m^+]\!]} \alpha_i+\sum_{j\in ]\!]m^+,m]\!]} \beta_j$.

Let us show that it is possible to transport along $\tilde G$ the measure $\sum_{i\in I} \alpha_i \delta_i$
to $\sum_{j\in J} \beta_j \delta_j$, by application of Lemma~\ref{lem:transport}.
The equality of masses holds by construction. It remains to prove that for every $S\subset I$,
$\sum_{i\in S} \alpha_i \le \sum_{j\in N(S)} \beta_j$, where $N(S)$ denotes the set of vertices which are connected
to $S$ in $\tilde G$. Let us consider several cases
\begin{itemize}
	\item If $0\not\in S$ and $S\not \sim m+1$ then the inequality comes from the hypothesis that $c\in \mathcal P_2$.
	\item If $0\not\in S$ and $S\sim m+1$, the inequality holds simply because the term $\beta_{m+1}$ is larger than
	  $\sum_{i=1}^{m^+} \alpha_i\ge \sum_{i\in S}\alpha_i$.
	\item If $0\in S$, then by construction $m+1\in N(S)$. Our aim is to show that 
	 $\sum_{i\in S} \alpha_i \le \sum_{j\in N(S)} \beta_j$. Subtracting from the equality of masses condition shows that the inequality is equivalent to $\sum_{i\in I\setminus S} \alpha_i \ge \sum_{j\in J\setminus N(S)} \beta_j$.
Define $T:=J\setminus N(S)$. Observe that $T\subset ]\!]m^+,m]\!]$ since $m+1\in N(S)$. Moreover by construction
$N(T)\subset I\setminus S$ does not contain 0 as $S$ does. So the fact that $c\in \mathcal P_2$ ensures that
$\sum_{j\in T} \beta_j \le \sum_{i\in N(T)} \alpha_i$. Since 	$N(T)\subset I\setminus S$ we obtain
$\sum_{i\in I\setminus S} \alpha_i \ge \sum_{j\in J\setminus N(S)} \beta_j$ as needed.
\end{itemize}
By Lemma~\ref{lem:transport} there is a transport along $\tilde G$. It may ship an amount $\gamma$ of mass
between the vertices $0$ and $m+1$. If we remove this amount from the initial mass at these two points, 
we get two distributions which admit a transport which does not use the edge between $0$ and $m+1$. In other
words if we set $c_0=1+\alpha_0-\gamma\ge 1$ and $c_{m+1}=-(\beta_{m+1}-\gamma)\le 0$, we have shown that 
there is a transportation plan along $G$ from $\sum_{i=0}^{m^+} (c_i-1)\delta_i$ to
$\sum_{j=1+m^+}^{m+1} |c_j|\delta_j$. According to Theorem~\ref{theo:positivity-rank1} and its interpretation in terms of transport, this means that $(c_0,c_1,\ldots,c_{m+1})$ is in the positivity domain $\mathcal P_{1 + m^+}(u_0,u_1,\ldots,u_{m+1})$ of an inverse Brascamp-Lieb inequality without kernel (but with more functions). 

So starting from $c=(c_1,\ldots,c_m)\in \mathcal P_2$ we have expressed it as the projection of a vector
$(c_0,c_1,\ldots,c_{m+1})$ in $\mathcal P_{1 + m^+}(u_0,u_1,\ldots,u_{m+1})$. This concludes the proof of $\mathcal{P}_2\subset \mathcal{P}_3$. The particular cases when $u_0$ (or $u_{m+1}$) is zero is also treated
by this argument because in this case $0$ is isolated in $G$ and thus it is not involved in the transport.

\medskip
The inclusion $\mathcal P_3\subset \mathcal P$ is immediate. It $c\in \mathcal P_3$, then by Theorem~\ref{theo:positivity-rank1} there exists $c_0>0$, $c_{m+1}\le 0$ and $\varepsilon>0$ such that for all non-negative integrable functions $f_k$, $k=0,\ldots,m+1$,
$$ \int_H f_0(\langle x,u_0\rangle)^{c_0} f_{m+1}(\langle x,u_{m+1}\rangle)^{c_{m+1}} \prod_{k=1}^m f_k\big(\langle x, u_k\rangle\big)^{c_k} dx \ge \varepsilon 
\prod_{k=0}^{m+1} \left( \int_{H_k} f_k\right)^{c_k} \cdot$$
It remains to choose adequate Gaussian functions $f_0$ and $f_{m+1}$ so that 
$$f_0(\langle x,u_0\rangle)^{c_0} f_{m+1}(\langle x,u_{m+1}\rangle)^{c_{m+1}}\le e^{-\pi \langle x,u_0\rangle^2 +\pi \langle x,u_{m+1}\rangle^2} =e^{-\mathcal Q(x)}$$ to get a non-trivial inequality for the initial functional.
It is possible to achieve equality when $c_{m+1}<0$; we use an inequality in case $c_{m+1}=0$. 
\end{proof}

\section{Positivity condition in the general case}
\label{sec:finiteness-n}
We turn to a positivity condition in the general case. 
Let $0\le m^+\le m$ and for $k=0,\ldots,m+1$, let $B_k:H\to H_k$ be a surjective linear map.
Recall that  $B_+$  denotes the map $(B_1, \ldots, B_{m^+}) \colon H \to H_1 \times \cdots \times H_{m^+}$. With this notation, $\ker B_+ = \bigcap_{k=1}^{m^+} \ker B_k$. Similarly we define $B_{0+}=(B_0,B_1,\ldots B_{m^+})\colon H \to H_0 \times \cdots \times H_{m^+}$.
Recall also the non-degeneracy conditions~\eqref{eq:injectivity-condition} and~\eqref{eq:surjectivity-condition}, which we assume from now on.

\subsection{Recursive structure of the problem}
Any linear subspace $V \subseteq H$, together with the quotient space  $\sfrac{H}{V}$,  yields a split of $H$, i.e. the following sequence is exact 
\[
\begin{CD} 0 @>>> V @>{i}>> H @>{\pi}>> \sfrac{H}{V} @>>> 0, \end{CD}
\]
where $i \colon V \to H$ is the  natural embedding and $\pi \colon H \to \sfrac{H}{V}$ is the natural quotient map. 

Next, for each $k = 0, \ldots, m+1$ denote
\[
  V_k = B_k V.
\]
We consider a split of $H_k$ induced from the split of $H$ by the map $B_k$, namely
\[
\begin{CD} 0 @>>> V_k @>{i_k}>> H_k @>{\pi_k}>> \sfrac{H_k}{V_k} @>>> 0, \end{CD}
\]
together with surjective linear maps $b_k \colon V \to V_k $ and $\beta_k \colon \sfrac{H}{V} \to \sfrac{H_k}{V_k}$ defined such that the diagram
\[
\begin{CD}
0 @>>>    V     @>{i}>>    H      @>{\pi}>> \sfrac{H}{V}     @>>> 0 \\
@.    @VV{b_k}V      @VV{B_k}V      @VV{\beta_k}V         @.\\
0 @>>>   V_k    @>{i_k}>>   H_k     @>{\pi_k}>> \sfrac{H_k}{V_k} @>>> 0 %\\
%@.     @VVV           @VVV             @VVV               @.\\
%@.       0      @.      0      @.        0           @.   @.
\end{CD}
\]
commutes.
In other words, $b_k$ is the restriction of $B_k$ to $V$, while $\beta_k$ is the quotient of $B_k$ by $V$, which can be defined explicitly by
\[
  \beta_k (x + V) = B_k x + V_k.
\]
Similarly as for the maps $B_k$ we consider the maps
\[ \begin{split}
  b_+ &= (b_1, \ldots, b_{m^+}) \colon V \to V_1 \times \cdots \times V_{m^+}, \\
  b_{0+} &= (b_0, b_1, \ldots, b_{m^+}) \colon V \to V_0 \times V_1 \times \cdots \times V_{m^+}, \\
  \beta_+ &= (\beta_1, \ldots, \beta_{m^+}) \colon \sfrac{H}{V} \to \sfrac{H_1}{V_1} \times \cdots \times \sfrac{H_{m^+}}{V_{m^+}}, \\
  \beta_{0+} &= (\beta_0, \beta_1, \ldots, \beta_{m^+}) \colon \sfrac{H}{V} \to \sfrac{H_0}{V_0} \times \sfrac{H_1}{V_1} \times \cdots \times \sfrac{H_{m^+}}{V_{m^+}}.
\end{split} \]
In the sequel, the above construction of restriction and quotient of maps will be applied recursively to $V$ and the maps $b_k \colon V \to V_k$ as well as to $\sfrac{H}{V}$ and the maps $\beta_k \colon \sfrac{H}{V} \to \sfrac{H_k}{V_k}$. 

Note a simple fact concerning the kernel of a quotient  map.
\begin{lemma}\label{lem:kernel-of-beta}
Let $B \colon H \to H'$ be a linear map (not necessarily surjective), $V \subseteq H$ be a subspace and denote $V' = B V$. Consider the map $\beta \colon \sfrac{H}{V} \to \sfrac{H'}{V'}$ defined as a quotient of $B$, i.e. $\beta (x + V) = B x + V'$. Then its kernel verifies 
\begin{equation}\label{eq:kernel-of-beta}
  \pi^{-1}(\ker \beta) = V + \ker B,
\end{equation}
where $\pi:H\to \sfrac{H}{V}$ is the canonical projection.
\end{lemma}
\begin{proof}
The inclusions $V \subseteq \pi^{-1}(\ker \beta)$ and $\ker B \subseteq \pi^{-1}(\ker \beta)$ are obvious. For the other inclusion, if $x \in \pi^{-1}(\ker \beta)$, i.e. $B x \in V' = B V$, then there exists $v \in V$ such that $B x = B v$ and hence $x = v + (x-v) \in V + \ker B$.
\end{proof}

\medskip

The following notion will be crucial:
\begin{dfn}
Given the maps $B_k \colon H \to H_k$ for $k = 0, 1, \ldots, m^+$ and a linear subspace $V\subset H$, we call the  split $\begin{CD} 0 @>>> V @>>> H @>>> \sfrac{H}{V} @>>> 0\end{CD}$ \emph{admissible}  for $(H,B)$ if the map $b_{0+}$ is a linear isomorphism. For shortness we also say that $V$ is an admissible subspace, and omit to mention $(H,B)$ when there is no ambiguity.
\end{dfn}

The next  lemma  tells that Condition~\eqref{eq:isomorphism-condition} is inherited by  subspaces and  quotients
induced by admissible splits.

\begin{lemma}
\label{lem:admissibility-of-subspace-and-quotient}
Suppose here that the map $B_{0+}$ is a linear isomorphism and consider a split $\begin{CD} 0 @>>> V @>>> H @>>> \sfrac{H}{V} @>>> 0\end{CD}$. Then the map $b_{0+}$ is injective and the map $\beta_{0+}$ is surjective. Moreover, the following assertions are equivalent:
\begin{enumerate}
\item[(i)] the split is admissible (i.e. the map $b_{0+}$ is a linear isomorphism),
\item[(i')] the map $b_{0+}$ is surjective,
\item[(ii)] $\dim V = \sum_{i=0}^{m^+} \dim B_i V$,
%\item[(ii')] $\dim V \ge \sum_{i=0}^{m^+} \dim B_i V$,
\item[(iii)] the map $\beta_{0+}$ is a linear isomorphism,
\item[(iii')] the map $\beta_{0+}$ is injective,
\item[(iv)] $\bigcap_{i=0}^{m^+} (V + \ker B_i) = V$.
\end{enumerate}
\end{lemma}
\begin{proof}
Since $B_{0+}:H\to H_0\times\cdots\times H_{m^+}$ is a linear isomorphism,
it is clear that its restriction $b_{0+}:V\to V_0\times\cdots\times V_{m^+}$ is injective (recall the notation   $V_k=B_kV$).  Hence (i) $\iff$ (i'). 
The fact that  $B_{0+}$ is onto also ensures that
 its quotient $\beta_{0+}:\sfrac{H}{V}\to \sfrac{H_0}{V_0} \times \cdots \times \sfrac{H_{m^+}}{V_{m^+}}$ is surjective. Hence (iii) $\iff$ (iii').
 
Next we use the basic fact that a linear map between finite dimensional vector spaces $L:X\to Y$ is bijective if and only if $\dim X=\dim Y$ and $L$ is injective (which is also equivalent to $\dim X=\dim Y$ and $L$ is surjective). This directly yields (i) $\iff$ (ii).

Since $B_{0+}$ is an isomorphism, it holds $\dim H=\sum_{i=0}^{m^+} \dim H_i$.
Therefore, by subtraction, (ii) is equivalent to 
$$ \dim \sfrac{H}{V} = \sum_{i=0}^{m^+} \dim \sfrac{H_i}{V_i}.$$
Since $\beta_{0+}$ is automatically surjective, we deduce that (ii)$\iff$(iii).

It remains to show that (iii')$\iff$(iv). To do this, we start with observing 
that 
\[
  \pi^{-1}(\ker \beta_{0+})=\pi^{-1}\Big(\bigcap_{i=0}^{m^+} \ker \beta_i \Big) =
   \bigcap_{i=0}^{m^+} \pi^{-1}(\ker \beta_i)=  \bigcap_{i=0}^{m^+} (V+\ker B_i),
\]
where the last equality comes from~\eqref{eq:kernel-of-beta}.  Taking the preimage w.r.t.  the surjective map $\pi$ gives that $\ker \beta_{0+}=\{0\}$ is equivalent to 
$  \pi^{-1}(\ker \beta_{0+} ) = V$.  The equivalence (iii')$\iff$(iv) follows from the above formula.
\end{proof}

Eventually, we show that Condition~\eqref{eq:kerBplus-contained-in-kerBm1} is inherited by the maps induced by admissible splits.
\begin{lemma}\label{lem:kerBplus-inside-kerBm1}
Consider a linear subspace $V\subset H$ and the corresponding split. Suppose $\ker B_+ \subseteq \ker B_{m+1}$. Then
\begin{enumerate}
\item[(i)] $\ker b_+ \subseteq \ker b_{m+1}$,
\item[(ii)] if $b_+$ is surjective then $\ker \beta_+ \subseteq \ker \beta_{m+1}$.
\end{enumerate}
\end{lemma}
\begin{proof}
\par{(i) This part is obvious since $\ker b_+ = V \cap \ker B_+$ and similarly for $\ker b_{m+1}$.}
\par{(ii) By Lemma~\ref{lem:kernel-of-beta} 
\[
  \pi^{-1}(\ker \beta_{m+1}) = V + \ker B_{m+1}.
\]
Applying Lemma~\ref{lem:kernel-of-beta} once again, this time to the map $B := B_+$ and the subspaces $V \subseteq H$ and $V' = B_+ V \subseteq H' := H_1 \times \cdots \times H_{m^+}$ we obtain that the map $\beta \colon \sfrac{H}{V} \to \sfrac{H'}{V'}$ being the quotient of $B$ satisfies
\[
  \pi^{-1}(\ker \beta) = V + \ker B.
\]
Since $b_+$ is surjective, i.e. $V' = B_1 V \times \cdots \times B_{m^+} V$, the map $\beta_+$ coincides with $\varphi \circ \beta$, where $\varphi \colon \sfrac{H'}{V'} = \sfrac{H_1 \times \cdots \times H_{m^+}}{B_1 V \times \cdots \times B_{m^+} V} \to \sfrac{H_1}{B_1V} \times \cdots \times \sfrac{H_{m^+}}{B_{m^+}V}$ is the natural isomorphism, hence $\ker \beta$ and $\ker \beta_+$ coincide. Therefore,
\[
  \pi^{-1}(\ker \beta_+) = V + \ker B_+
\]
and using the hypothesis $\ker B_+ \subseteq \ker B_{m+1}$ we obtain
\[
  \pi^{-1}(\ker \beta_+) \subseteq V + \ker B_{m+1} = \pi^{-1}(\ker \beta_{m+1}).
\]
To conclude, it remains to apply the surjective map $\pi$.}

\end{proof}

\begin{corollary}\label{cor:inheritance-of-kernel-conditions}
Suppose the map $B_{0+}$ is a linear isomorphism and $\ker B_+ \subseteq \ker B_{m+1}$. If a split $\begin{CD} 0 @>>> V @>>> H @>>> \sfrac{H}{V} @>>> 0\end{CD}$ is admissible then
\begin{enumerate}
  \item[(i)] $b_{0+}$ is a linear isomorphism and  $\ker b_+ \subseteq \ker b_{m+1}$,
  \item[(ii)] $\beta_{0+}$ is a linear isomorphism and $\ker \beta_+ \subseteq \ker \beta_{m+1}$.  
\end{enumerate}
\end{corollary}
\begin{proof} This is a direct consequence of Lemma~\ref{lem:kerBplus-inside-kerBm1} 
and of the equivalent forms of the admissibility property given in Lemma~\ref{lem:admissibility-of-subspace-and-quotient}.
\end{proof}

\subsection{Formulation of the characterization result}
Recall $0\le m^+\le m$. In addition to the linear surjective maps $B_k \colon H \to H_k$, $k = 0, \ldots, m+1$, we consider real numbers $c_0 = 1$, $c_1, \ldots, c_{m^+} > 0$, $c_{m^+ +1}, \ldots, c_m \le 0$ and  $c_{m+1} = -1$. In this context, for any positive definite quadratic forms $\mathcal Q_+ \colon H_0 \to \R$ and $\mathcal Q_- \colon H_{m+1} \to \R$, define a functional $J_{\mathcal Q_+, \mathcal Q_-}$ acting on non-negative integrable functions $f_k \colon H_k \to \R$ ($k = 1, \ldots, m$) satisfying $\int_{H_k} f_k > 0$:
\begin{equation}\label{eq:J-Qplus-Qminus}
  J_{\mathcal Q_+, \mathcal Q_-}(f_1, \ldots, f_m) = \frac{\int_H \prod_{k=0}^{m+1} f_k^{c_k} (B_k x) \,dx}{\prod_{k=1}^m \Big( \int_{H_k} f_k \Big)^{c_k}},
\end{equation}
where
\begin{equation}\label{eq:def-f0-fm1}
f_0 = e^{-\mathcal Q_+} \quad \text{and} \quad f_{m+1} = e^{-\mathcal Q_-}.
\end{equation}

Assuming~\eqref{eq:isomorphism-condition} and~\eqref{eq:kerBplus-contained-in-kerBm1}, the condition defined below turns out to be equivalent to the  positivity of the infimum of $J_{\mathcal Q_+, \mathcal Q_-}$ over all functions $f_1, \ldots, f_m$.

\begin{dfn}
We say that $H$ together with the maps $B_k$ and the exponents $c_k$ ($k=0, \ldots, m+1$) satisfies \emph{Condition (C)} if for every admissible split
\[
  \begin{CD} 0 @>>> V @>>> H @>>> \sfrac{H}{V} @>>> 0\end{CD},
\]
the following two conditions are satisfied:
\begin{enumerate}
\item[(i)] if $b_{m+1}$ is a trivial map (i.e. $V \subseteq \ker B_{m+1}$ thus $V_{m+1} = \{0\}$) then $V$ is a \emph{supercritical subspace} of $H$, i.e.
\[ \dim V \ge \sum_{k=1}^m c_k \dim V_k; \]
\item[(ii)] if $\beta_0$ is a trivial map (i.e. $B_0 V = B_0 H$ thus $\sfrac{H_0}{V_0} = \{0\}$) then $\sfrac{H}{V}$ is a \emph{subcritical quotient} of $H$, i.e.
\[ \dim \sfrac{H}{V} \le \sum_{k=1}^m c_k \dim \sfrac{H_k}{V_k}. \]
\end{enumerate}
\end{dfn}
Later on we will also use a similar notion which we call \emph{criticality}:
\begin{dfn}
Suppose $V\subset H$ induces an admissible split. We say that $V$ is a \emph{critical subspace} of $H$ if
\[ b_{m+1} \text{ is trivial and } \dim V = \sum_{k=1}^m c_k \dim V_k. \]
Similarly, we say that $\sfrac{H}{V}$ is a \emph{critical quotient} of $H$ if
\[ \beta_0 \text{ is trivial and } \dim \sfrac{H}{V} = \sum_{k=1}^m c_k \dim \sfrac{H_k}{V_k}. \]
\end{dfn}

\begin{theorem}\label{thm:positivity-of-constant-general-case}
In the setting described above, suppose that~\eqref{eq:isomorphism-condition} and~\eqref{eq:kerBplus-contained-in-kerBm1} hold.
\begin{enumerate}
\item[(i)] If for some positive definite quadratic forms $\mathcal Q_+ \colon H_0 \to \R$ and $\mathcal Q_- \colon H_{m+1} \to \R$,
\[
  \inf_{f_1, \ldots, f_m} J_{\mathcal Q_+, \mathcal Q_-}(f_1, \ldots, f_m) > 0
\]
then $(H,B,c)$ satisfies Condition (C).
\item[(ii)] If $(H,B,c)$ satisfies Condition (C) then for all positive definite quadratic forms $\mathcal Q_+$ and $\mathcal Q_-$,
\[
  \inf_{f_1, \ldots, f_m} J_{\mathcal Q_+, \mathcal Q_-}(f_1, \ldots, f_m) > 0.
\]
\end{enumerate}
\end{theorem}

The above theorem easily implies the characterization of positivity of the functional $J$, for which we present now an intrinsic formulation (in terms of $\mathcal Q$ only, and not
of its decomposition involving $B_0, B_{m+1}$).

\begin{theorem}\label{th:characterization-in-general-case}
Consider the functional $J$ as defined in~\eqref{def:J} along with a quadratic form $\mathcal Q \colon H \to \R$, surjective linear maps $B_k \colon H \to H_k$ ($k=1,\ldots,m$) and exponents $c_1, \ldots, c_{m^+} > 0, c_{m^+ +1}, \ldots, c_m \le 0$. Suppose the non-degeneracy condition~\eqref{eq:injectivity-condition} and~\eqref{eq:surjectivity-condition} hold. Then $\inf J > 0$ if and only if for every subspace $V \subseteq H$ such that
\begin{equation}\label{eq:admissibility-in-terms-of-Q-via-surjectivity-of-b0plus}
  \dim\big(V \cap (\ker B_+)^{\perp_{\mathcal Q}}\big) = \sum_{i=1}^{m^+} \dim B_i V,
\end{equation}
 the following two implications hold true:
\begin{enumerate}
\item[(i)] if $V \subseteq \rad \mathcal Q + \ker B_+$ then $$\dim V \ge \sum_{k=1}^m c_k \dim B_k V;$$
\item[(ii)] if $V + (\ker B_+)^{\perp_{\mathcal Q}} = H$ then $$\dim H - \dim V \le \sum_{k=1}^m c_k (\dim H_k - \dim B_k V).$$
\end{enumerate}
\end{theorem}
\begin{remark}
When no kernel is involved (i.e. $\mathcal Q=0$) we recover, in a slightly different form, the condition of Theorem~\ref{th:positivity-general-rank-no-kernel} from the introduction. To see the connection, observe that $(i)$ for $V=H$ and $(ii)$ for $V=\{0\}$ yield $\dim H=\sum_{k=1}^m c_k \dim H_k$.  Then, it is clear that the inequalities in $(i)$ and $(ii)$ are equivalent.
\end{remark}

\begin{proof}[Proof of Theorem \ref{th:characterization-in-general-case}]
First we construct the maps $B_0 \colon H \to H_0$ and $B_{m+1} \colon H \to H_{m+1}$ as in the proof of the implication (1) $\implies$ (2) from Lemma~\ref{lem:decompose-Q}. Recall from that proof that $$\ker B_0=H_0^{\perp_{\mathcal Q}} = (\ker B_+)^{\perp_{\mathcal Q}}$$ and $$\ker B_{m+1} = \rad \mathcal Q + \ker B_+.$$

Next, we apply Theorem~\ref{thm:positivity-of-constant-general-case}, and reformulate it in terms of the quadratic form $\mathcal Q$ as follows. By~\eqref{eq:isomorphism-condition} and  Lemma~\ref{lem:admissibility-of-subspace-and-quotient}, a subspace $V\subset H$ is admissible if and only if 
\[
  \dim V = \sum_{i=0}^{m^+} \dim B_i V.
\]
The last equation is equivalent to~\eqref{eq:admissibility-in-terms-of-Q-via-surjectivity-of-b0plus}, thanks to the following relation
\[
  \dim B_0 V = \dim V - \dim (V \cap \ker B_0) = \dim V - \dim \big(V \cap (\ker B_+)^{\perp_{\mathcal Q}}\big).
\]
Finally note that the following equivalences hold true:
\[ \begin{split}
  V \subseteq \ker B_{m+1} \quad &\iff \quad V \subseteq \rad \mathcal Q + \ker B_+, \\
  B_0 V = B_0 H \quad &\iff \quad V + \ker B_0 = H \quad \iff \quad V + (\ker B_+)^{\perp_{\mathcal Q}} = H.
\end{split} \]
\end{proof}

Eventually, let us note that Lemma~\ref{lem:admissibility-of-subspace-and-quotient} (i) $\iff$ (iv) shows that  \eqref{eq:admissibility-in-terms-of-Q-via-surjectivity-of-b0plus} is  equivalent to 
\[
  \big(V + (\ker B_+)^{\perp_{\mathcal Q}}\big) \cap \bigcap_{i=1}^{m^+} (V + \ker B_i) = V.
\]

\subsection{Useful notation for the proof of Theorem~\ref{thm:positivity-of-constant-general-case}}\label{sec:prelim-results-characterization-theorem}

Consider any split
\[
  \begin{CD} 0 @>>> V @>>> H @>>> \sfrac{H}{V} @>>> 0\end{CD}.
\]
We fix any linear injective maps
\[ \begin{split}
  j &\colon \sfrac{H}{V} \to H, \\
  j_k &\colon \sfrac{H_k}{V_k} \to H_k \quad \textup{for $k = 0, \ldots, m+1$}
\end{split} \]
such that $j$ (respectively $j_k$) composed with the canonical quotient map $H \to \sfrac{H}{V}$ (resp. $H_k \to \sfrac{H_k}{V_k}$) is the identity on $\sfrac{H}{V}$ (resp. $\sfrac{H_k}{V_k}$). For example, $j$ can be chosen so that its range is the orthogonal complement of $V$ in $H$ (and similarly $j_k$).

For any $x \in V$ and $y \in \sfrac{H}{V}$ we write
\[
  B_k (x + j(y)) = B_k x + B_k j(y) = b_k x + \rho_k y + j_k(\beta_k y),
\]
where
\[
  \rho_k y = B_k j(y) - j_k(\beta_k y) \colon \sfrac{H}{V} \to V_k.
\]
To see that $\rho_k$ has range in $V_k$, compose it with the quotient map $\pi_k \colon H_k \to \sfrac{H_k}{V_k}$ to see that $\pi_k \rho_k(y) = \pi_k B_k j(y) - \beta_k \pi j(y) = (\pi_k B_k - \beta_k \pi)(j(y)) = 0$, by definition of $\beta_k$. 

Fix any positive definite quadratic forms $\mathcal Q_+$ on $H_0$ and $\mathcal Q_-$ on $H_{m+1}$ and thus fix $f_0$ and $f_{m+1}$ as in~\eqref{eq:def-f0-fm1}. Next,
let $f_k \colon H_k \to \R$ ($k = 1, 2, \ldots, m$) be non-negative, integrable functions with $\int_{H_k} f_k > 0$. By identifying each function $f_k \colon H_k \to \R$ (for $k = 0, 1, \ldots, m + 1$) with a function $f_k \colon V_k \times \sfrac{H_k}{V_k} \to \R$ and using Fubini theorem we rewrite~\eqref{eq:J-Qplus-Qminus} as 
\begin{equation}\label{eq:J-iterative-integral}
  J_{\mathcal Q_+, \mathcal Q_-}(f_1, \ldots, f_m) = C \frac{\int_{\sfrac{H}{V}} \int_V \prod_{k=0}^{m+1} f_k^{c_k}(b_k x + \rho_k y, \beta_k y) \, dx \, dy}{\prod_{k=1}^m \left(\int_{\sfrac{H_k}{V_k}} \int_{V_k} f_k(x, y) \, dx \, dy \right)^{c_k}},
\end{equation}
where $C \in (0, +\infty)$ is a constant resulting from changes of variables $V \times \sfrac{H}{V} \ni (x, y) \mapsto x + j(y) \in H$ and $V_k \times \sfrac{H_k}{V_k} \ni (x, y) \mapsto x + j_k(y) \in H_k$ (for $k=1,\ldots,m$). (If $j$ and $j_k$ are chosen according to Euclidean structures of $H$ and $H_k$ as in the example mentioned above, then $C = 1$. However, in what follows the exact value of $C$ has no importance).

\subsection{Necessity of Condition (C)}
Here we prove the first assertion of Theorem~\ref{thm:positivity-of-constant-general-case}.
\begin{proof}[Proof of Theorem~\ref{thm:positivity-of-constant-general-case}, part (i)]
Recall the discussion from Section~\ref{sec:prelim-results-characterization-theorem}. Assume that the subspace $V$  is admissible and that we choose the functions $f_1, \ldots, f_{m^+}$ so that they are bounded of compact support and the functions $f_{m^+ +1}, \ldots, f_m$ which are strictly positive with polynomial decay at infinity.

First consider the case $V \subseteq \ker B_{m+1}$ (i.e. $V_{m+1} = \{0\}$ and thus $b_{m+1}$ and $\rho_{m+1}$ are trivial). We aim at showing that $V$ is a supercritical subspace of $H$.
To this end, for any $R \in [1, \infty)$, set
\[
  f^{(R)}_k(x,y) = f_k(x/R,y) \quad \textup{for $k=1,\ldots,m$.}
\]
By the hypothesis, for all $R \ge 1$, $J_{\mathcal Q_+, \mathcal Q_-}(f^{(R)}_1, \ldots, f^{(R)}_m)$ is uniformly bounded from below by a positive constant. On the other hand, using \eqref{eq:J-iterative-integral}  for $(f^{(R)}_1,\ldots,f^{(R)}_m) $ and  rescaling the variables of integration $x$ in the numerator and the denominator  (i.e. replacing $x$ with $Rx$) gives
\begin{equation}\label{eq:J-interative-integral-with-R-big} 
\begin{split}
  & J_{\mathcal Q_+, \mathcal Q_-}(f^{(R)}_1, \ldots, f^{(R)}_m) = C \times R^{\dim V - \sum_{k=1}^m c_k \dim V_k} \\
  & \times \frac{\int_{\sfrac{H}{V}} \int_V f_0(R b_0 x + \rho_0 y, \beta_0 y) f_{m+1}^{-1}(0, \beta_{m+1} y) \prod_{k=1}^m f_k^{c_k}(b_k x + \frac{1}{R} \rho_k y, \beta_k y) \, dx \, dy}{\prod_{k=1}^m \left(\int_{\sfrac{H_k}{V_k}} \int_{V_k} f_k(x, y) \, dx \, dy \right)^{c_k}}.
\end{split}
\end{equation}
Now it is enough to show that the double integral in the numerator is uniformly bounded from above as $R \to \infty$. Doing so, the positive lower bound on the l.h.s. of~\eqref{eq:J-interative-integral-with-R-big} implies that $R^{\dim V - \sum_{k=1}^m c_k \dim V_k}$ is bounded away from $0$ as $R \to \infty$ and thus $\dim V - \sum_{k=1}^m c_k \dim V_k \ge 0$.

Due to our choice of the functions $f_k$, for $k=1, \ldots, m^+$,
\[
  \supp f_k \subseteq F_k \times G_k,
\]
for some compact, star-shaped sets $F_k \subseteq V_k$ and $G_k \subseteq \sfrac{H_k}{V_k}$ (by star-shaped we mean that if $x$ is in a set then so does $\lambda x$ for any $\lambda \in [0, 1]$).

Thanks to the assumption~\eqref{eq:isomorphism-condition} and the admissibility of the split of $H$, the maps $b_{0+}$ and $\beta_{0+}$ are linear isomorphisms (see Lemma~\ref{lem:admissibility-of-subspace-and-quotient}). 

Observe that we can restrict the domain of the outer integral in the numerator of~\eqref{eq:J-interative-integral-with-R-big} to the set
\[
  G := \beta_+^{-1}(G_1 \times \cdots \times G_{m^+}) \subseteq \sfrac{H}{V},
\]
because outside $G$ the terms $f_k^{c_k}$ with $k \in \{1, \ldots, m^+\}$ make the integrand vanish.
Although $G$ is not necessarily compact, this allows us to bound the exponentially large term $f_{m+1}^{-1}$ in~\eqref{eq:J-interative-integral-with-R-big}. Indeed, the first assertion of Corollary~\ref{cor:inheritance-of-kernel-conditions}(ii) is
\[
  \ker \beta_+ \subseteq \ker \beta_{m+1},
\]
hence by Lemma~\ref{lem:compact-image}, $\beta_{m+1}(G)$ is a compact subset of $\sfrac{H_{m+1}}{V_{m+1}}$ and thus we can bound from above the integrand by replacing $f_{m+1}^{-1}$ with
\[
  \sigma:=\sup_{\{0\} \times \beta_{m+1}(G)} f_{m+1}^{-1} < \infty.
\]

In order to deal with the terms $f_k^{c_k}$ for $k \in \{ m^+ + 1, \ldots, m\}$ that grow (at most) polynomially at infinity, we take advantage of the exponential decay of  $f_0$. In order to use a compactness argument we decompose $f_0$ into slices. Namely, note that for some compact, star-shaped sets $F_0 \subseteq V_0$, $G_0 \subseteq \sfrac{H_0}{V_0}$, which depend on $\mathcal Q_+$ and the map $j_0$ only, we have
\[ \begin{split}
  f_0(x_0, y_0) &= \int_0^1 \ind{(x, y) \in V_0 \times \sfrac{H_0}{V_0} \colon \exp(-\mathcal Q_+(x + j_0 y)) \ge u}(x_0, y_0) \, du \\
      &= \int_0^\infty t e^{-t^2/2} \ind{(x, y) \in V_0 \times \sfrac{H_0}{V_0} \colon \exp(-\mathcal Q_+(x + j_0 y)) \ge \exp(-t^2/2)}(x_0, y_0) \, dt \\
      &= \int_0^\infty t e^{-t^2/2} \ind{(x, y) \in V_0 \times \sfrac{H_0}{V_0} \colon \mathcal Q_+(x + j_0 y) \le t^2/2}(x_0, y_0) \, dt \\
      &\le \int_0^\infty t e^{-t^2/2} \Ind{t F_0}(x_0) \Ind{t G_0}(y_0) \, dt
\end{split} \]
for all $(x_0, y_0) \in V_0 \times \sfrac{H_0}{V_0}$. Using Fubini, we can thus bound the numerator of~\eqref{eq:J-interative-integral-with-R-big} by
\begin{equation}\label{eq:J-interative-integral-with-R-big-numerator}
 \sigma  \int_0^\infty t e^{-t^2/2} \int_{\sfrac{H}{V}} \int_V \Ind{t F_0}(R b_0 x + \rho_0 y) \Ind{t G_0}(\beta_0 y) \prod_{k=1}^m f_k^{c_k}\Big(b_k x + \frac{1}{R} \rho_k y, \beta_k y\Big) \, dx \, dy \, dt.
\end{equation}
%(Above we already skipped the term $f_{m+1}^{-1}$ as already dealt with).

\medskip

Now we argue that for some polynomials $p$ and $q$, for any $t > 0$  and all $R\ge 1$, the integrand of the double integral w.r.t $x$ and $y$ in~\eqref{eq:J-interative-integral-with-R-big-numerator} is bounded from above by $q(t)$ and is supported in a compact set of measure at most   $p(t)$.

To this end, fix any $R \ge 1$ and $t > 0$. The integrand in question vanishes if $y$ is outside the set $\beta_{0+}^{-1}(tG_0 \times G_1 \times \cdots \times G_{m^+})$. Clearly we have
\[
  \beta_{0+}^{-1}(tG_0 \times G_1 \times \cdots \times G_{m^+}) \subseteq (t+1) \beta_{0+}^{-1}(G_0 \times G_1 \times \cdots \times G_{m^+}).
\]
Since $\beta_{0+}$ is an isomorphism, the set
\[
  \mathbf{G} = \beta_{0+}^{-1}(G_0 \times G_1 \times \cdots \times G_{m^+})
\]
is a compact (and star-shaped) subset of $\sfrac{H}{V}$. Thus we can restrict the domain of  integration w.r.t. $y$ to $(t+1)\mathbf{G}$.

Next, fix $y \in (t+1) \mathbf{G}$ and $R \ge 1$. Take any $x \in V$ such that
\[ \begin{split}
  R b_0 x + \rho_0 y &\in t F_0, \\
  b_k x + \frac{1}{R} \rho_k y &\in F_k \quad \text{for all $k = 1,\ldots,m^+$}
\end{split} \]
(otherwise the integrand is zero). Then we have
\[ \begin{split}
  b_0 x &\in \frac{t}{R} F_0 + \Big(- \frac{1}{R} \rho_0((t+1)\mathbf{G}) \Big) \subseteq (t+1)(F_0 + \rho_0(-\mathbf{G})), \\
  b_k x &\in F_k + \Big(-\frac{1}{R} \rho_k((t+1) \mathbf{G})\Big) \subseteq (t+1)(F_k + \rho_k(-\mathbf{G})) \quad \textup{for $k = 1,\ldots,m^+$,}
\end{split} \]
where the inclusion follows from the fact that $F_0, F_1, \ldots, F_{m^+}$ and $-\mathbf{G}$ are star-shaped. Consider compact sets
\[
  \tilde{F}_k = F_k + \rho_k(-\mathbf{G}) \subseteq V_k, \quad \textup{for $k = 0,1,\ldots,m^+$}.
\]
Put
\[
  \mathbf{F} =  b_{0+}^{-1}(\tilde{F_0} \times \tilde{F}_1 \times \cdots \times \tilde{F}_{m^+}).
\]
Clearly $x \in (t+1) \mathbf{F}$ for all $y \in (t+1) \mathbf{G}$ and all $R \ge 1$ and hence one can restrict the integral w.r.t. $x$ to the domain $(t+1) \mathbf{F}$ which is compact, because $b_{0+}$ is an isomorphism. Therefore we have shown that for all $R \ge 1$, the domain of the double integral in~\eqref{eq:J-interative-integral-with-R-big-numerator} can be restricted to the compact set $(t+1) (\mathbf{F} \times \mathbf{G})$. Moreover, the measure of this set is a  polynomial 
function of $t$.

\smallskip

Now we proceed with bounding the integrand inside $(t+1) (\mathbf{F} \times \mathbf{G})$. The functions $f_1, \ldots, f_{m^+}$ are bounded, so we may focus on the terms involving $f_k$ for $k \in \{m^+ + 1, \ldots m\}$. Set
\begin{equation}
\label{eq:def-F-k-G-k}
\begin{split}
  F_k &= b_k(\mathbf{F}) + \rho_k(\mathbf{G}), \\
  G_k &= \beta_k(\mathbf{G})
\end{split}
\end{equation}
for $k = m^+ +1, \ldots, m$. Then for all $(x,y) \in (t+1) (\mathbf{F} \times \mathbf{G})$, all $R \ge 1$ and $k = m^+ +1, \ldots, m$,
\[
  b_k x + \frac{1}{R} \rho_k y \in b_k\big((t+1) \mathbf{F}\big) + \frac{1}{R} \rho_k\big((t+1) \mathbf{G}\big) \subseteq (t+1) F_k
\]
since $\mathbf{G}$ is star-shaped, and, of course,
\[
  \beta_k y \in (t+1) G_k.
\]
Therefore the integrand can be bounded from above by
\[
  \prod_{k=1}^{m^+} \Big( \sup_{H_k} f_k \Big)^{c_k} \times \prod_{k=m^+ +1}^m \Big( \sup_{(t+1)(F_k \times G_k)} f_k^{-1} \Big)^{-c_k},
\]
where the first product is finite by boundedness of the functions $f_1, \ldots, f_{m^+}$ and the second product is bounded by a polynomial in $t$ due to polynomial decay of the functions $f_{m^+ +1}, \ldots, f_m$. Consequently, \eqref{eq:J-interative-integral-with-R-big} is upper bounded independently of $R$, as claimed.

\bigskip

Now we pass to the case when $B_0 V = B_0 H$ (i.e. $\sfrac{H_0}{V_0} = \{0\}$ and thus $\beta_0$ is trivial). Using a similar reasoning to the one used in the first case, we will show that $\sfrac{H}{V}$ is a subcritical quotient of $H$. For any $r \in (0, 1]$ set
\[
  f_k^{(r)}(x,y) = f_k(x,y/r) \quad \textup{for $k=1, \ldots, m$}.
\]
We apply ~\eqref{eq:J-iterative-integral} for $(f_1^{(r)} ,\ldots,f_m^{(r)} )$ and then
we rescale the variables $y$ in the numerator and the denominator (i.e. we replace 
$y$ with $ry$ in all integrals with respect to $y$). We get
\begin{equation}\label{eq:J-interative-integral-with-r-small} 
\begin{split}
  & J_{\mathcal Q_+,\mathcal  Q_-}(f^{(r)}_1, \ldots, f^{(r)}_m) = C \times r^{\dim \sfrac{H}{V} - \sum_{k=1}^m c_k \dim \sfrac{H_k}{V_k}} \\
  & \times \frac{\int_{\sfrac{H}{V}} \int_V f_0(b_0 x + r \rho_0 y, 0) f_{m+1}^{-1}(b_{m+1} x + r \rho_{m+1} y, r \beta_{m+1} y) \prod_{k=1}^m f_k^{c_k}(b_k x + r \rho_k y, \beta_k y) \, dx \, dy}{\prod_{k=1}^m \left(\int_{\sfrac{H_k}{V_k}} \int_{V_k} f_k(x, y) \, dx \, dy \right)^{c_k}}.
\end{split}
\end{equation}
As before, it is enough to show that the double integral in the numerator is uniformly bounded from above as $r \to 0$.

First we deal with the term $f_{m+1}^{-1}$. The map $\beta_{0+}$ is a linear isomorphism, and since the map $\beta_0$ is trivial, also the map $\beta_+$ is an isomorphism. Therefore, the set
\[
  \mathbf{G} := \beta_+^{-1}(G_1 \times \cdots \times G_{m^+}) \subseteq \sfrac{H}{V}
\]
to which we can restrict the integral w.r.t. $y$ in~\eqref{eq:J-interative-integral-with-r-small} is compact (and star-shaped). Now, fix any $y \in \mathbf{G}$ and take any $x \in V$ such that $b_k x + r \rho_k y \in F_k$ for all $k = 1, \ldots, m^+$. Then we have
\[
  b_k x \in F_k + \big(-r \rho_k(\mathbf{G})\big) \subseteq F_k + \rho_k(-\mathbf{G}) := \tilde{F}_k
\]
and the sets $\tilde{F}_k \subseteq V_k$ are compact and star-shaped. Put
\[
  F = b_+^{-1}(\tilde{F}_1 \times \cdots \times \tilde{F}_{m^+}) \subseteq V.
\]
The set $F$ is not necessarily compact, but the first assertion of Corollary~\ref{cor:inheritance-of-kernel-conditions} says that $\ker b_+ \subseteq \ker b_{m+1}$ and hence by Lemma~\ref{lem:compact-image}, $b_{m+1}(F)$ is a compact subset of $V_{m+1}$. Therefore for all $r \in (0,1]$ we have the bound
\[
  f_{m+1}^{-1}(b_{m+1} x + r \rho_{m+1} y, r \beta_{m+1} y) \le
  \sup_{(b_{m+1}(F) + \rho_{m+1}(\mathbf{G})) \times \mathbf{G}} f_{m+1}^{-1} < \infty.
\]

\medskip

In order to deal with the terms $f_k^{c_k}$ for $k \in \{ m^+ + 1, \ldots, m\}$ we  decompose $f_0(\cdot, 0)$ into slices. Namely, defining the compact, star-shaped set $F_0=\{ x\in V_0 \colon \mathcal Q_+(x)\le 1/2\} $, we have that for all $x_0 \in V_0 = H_0$,
\[ \begin{split}
  f_0(x_0, 0) &= \int_0^1 \ind{x \in V_0 \colon \exp(-\mathcal Q_+(x)) \ge u}(x_0) \, du \\
      &= \int_0^\infty t e^{-t^2/2} \ind{x \in V_0 \colon \exp(-\mathcal Q_+(x)) \ge \exp(-t^2/2)}(x_0) \, dt \\
  % &= \int_0^\infty t e^{-t^2/2} \ind{x \in V_0 \colon \mathcal Q_+(x) \le t^2/2}(x_0) \, dt \\
      &= \int_0^\infty t e^{-t^2/2} \Ind{t F_0}(x_0) \, dt.
\end{split} \]
Using Fubini, we can thus bound the numerator of~\eqref{eq:J-interative-integral-with-r-small} by a constant (our bound on the terms involving $f_{m+1}^{-1}$) times
\begin{equation}\label{eq:J-interative-integral-with-r-small-numerator}
  \int_0^\infty t e^{-t^2/2} \int_{\sfrac{H}{V}} \int_V \Ind{t F_0}(b_0 x + r \rho_0 y) \prod_{k=1}^m f_k^{c_k}(b_k x + r \rho_k y, \beta_k y) \, dx \, dy \, dt.
\end{equation}
%(Above we already skipped the term $f_{m+1}^{-1}$ as already dealt with).

As discussed above, the domain of the integration w.r.t. $y$ can be restricted to the compact set $\mathbf{G} \subseteq \sfrac{H}{V}$. Now fix $t > 0$, $y \in \mathbf{G}$ and $r \in (0, 1]$. Suppose that $x \in V$ is such that the integrand in~\eqref{eq:J-interative-integral-with-r-small-numerator} does not vanish. Then we must have
\[ \begin{split}
  b_0 x + r \rho_0 y &\in t F_0, \\
  b_k x + r \rho_k y &\in F_k \quad \textup{for $k = 1, \ldots, m^+$,}
\end{split} \]
which implies
\[ \begin{split}
  b_0 x &\in t F_0 + \big(-r \rho_0(\mathbf{G})\big) \subseteq (t+1) (F_0 + \rho_0(-\mathbf{G})) =: (t+1) \tilde{F}_0, \\
  b_k x &\in F_k + \big(-r \rho_k(\mathbf{G})\big) \subseteq F_k + \rho_k(-\mathbf{G}) =: \tilde{F}_k \quad \textup{for $k = 1, \ldots, m^+$.}
\end{split} \]
Set
\[
  \mathbf{F} = b_{0+}^{-1}(\tilde{F}_0 \times \tilde{F}_1 \times \cdots \times \tilde{F}_{m^+}).
\]
Clearly $\mathbf{F}$ is a compact ($b_{0+}$ is an isomorphism), star-shaped subset of $V$ and $x \in (t+1) \mathbf{F}$.

Finally define the sets $F_k$ and $G_k$ for $k = m^+ + 1, \ldots, m$ as in~\eqref{eq:def-F-k-G-k}. Then for all $(x, y) \in ((t+1) \mathbf{F}) \times \mathbf{G}$ and all $r \in (0,1]$, the arguments of the functions $f_k$ for $k = m^+ + 1, \ldots, m$ are in $\big((t+1) F_k\big) \times G_k$ and therefore the integrand can be bounded from above by
\[
  \prod_{k=1}^{m^+} \Big( \sup_{H_k} f_k \Big)^{c_k} \times \prod_{k=m^+ +1}^m \Big( \sup_{((t+1) F_k) \times G_k} f_k^{-1} \Big)^{-c_k}.
\]
We conclude as in the first case.
\end{proof}

\subsection{Sufficiency of Condition (C)}
The inductive proof of the second part of Theorem~\ref{thm:positivity-of-constant-general-case} relies on the following lemma. It shows that under \eqref{eq:isomorphism-condition}, if one of the components of an admissible split (the subspace or the quotient) is critical, then Condition (C) is inherited by  both components of the split.

\begin{lemma}[Inheritance of Condition (C) through a critical split]
\label{lem:inheritence-of-condition-C}
Suppose $H$ together with the maps $B_k$ and the exponents $c_k$ satisfy Condition (C) and
\[
  \begin{CD} 0 @>>> V @>>> H @>>> \sfrac{H}{V} @>>> 0\end{CD}
\]
is an admissible split. If $V$ is a critical subspace of $H$ or $\sfrac{H}{V}$ is a critical quotient of $H$ then $V$ with the maps $b_k$ and the exponents $c_k$, as well as $\sfrac{H}{V}$ with the maps $\beta_k$ and the exponents $c_k$ satisfy Condition (C).
\end{lemma}
\begin{proof}
First we present the  scheme of the proof:
\begin{enumerate}
\item[\emph{Part I.}] We suppose that $V$ is a critical subspace of $H$.
\begin{itemize}
\item Condition (C) for $V$:
\begin{itemize}
  \item \emph{Subspace of $V$:} supercriticality of a subspace $U$ of $V$ is inherited directly from supercriticality of $U$ as a subspace of $H$.
  \item \emph{Quotient of $V$:} subcriticality of the quotient $\sfrac{V}{U}$ of $V$ follows from criticality of $V$ in $H$ and supercriticality of $U$ in $V$ just proved.
\end{itemize}
\item Condition (C) for $\sfrac{H}{V}$:
\begin{itemize}
  \item \emph{Subspace of $\sfrac{H}{V}$:} supercriticality of a subspace $U$ of $\sfrac{H}{V}$ follows from supercriticality of a subspace $\tilde{U} = \pi^{-1}(U)$ (where $\pi \colon H \to \sfrac{H}{V}$ is a natural quotient map) in $H$ and criticality of $V$ in $H$.
  \item \emph{Quotient of $\sfrac{H}{V}$:} subcriticality of the quotient $\sfrac{\sfrac{H}{V}}{U}$ of $\sfrac{H}{V}$ follows from subcriticality of the quotient $\sfrac{H}{\tilde{U}}$ of $H$.
\end{itemize}
\end{itemize}

\item[\emph{Part II.}] We suppose that $\sfrac{H}{V}$ is a critical quotient of $H$. After dualizing, i.e. interchanging subspaces with quotient and supercriticality with subcriticality, the arguments are analogous to the ones from Part I.
\begin{itemize}
\item Condition (C) for $\sfrac{H}{V}$:
\begin{itemize}
  \item \emph{Quotient of $\sfrac{H}{V}$:} subcriticality of a quotient $\sfrac{\sfrac{H}{V}}{U}$ of $\sfrac{H}{V}$ is inherited directly from subcriticality of $\sfrac{H}{\tilde{U}}$ as a quotient of $H$, where $\tilde{U} = \pi^{-1}(U)$.
  \item \emph{Subspace of $\sfrac{H}{V}$:} supercriticality of the subspace $U$ of $\sfrac{H}{V}$ follows from criticality of $\sfrac{H}{V}$ in $H$ and subcriticality of $\sfrac{\sfrac{H}{V}}{U}$ in $\sfrac{H}{V}$ just proved.
\end{itemize}
\item Condition (C) for $V$:
\begin{itemize}
  \item \emph{Quotient of $V$:} subcriticality of a quotient $\sfrac{V}{U}$ of $V$ follows from subcriticality of a quotient $\sfrac{H}{U}$ in $H$ and criticality of $\sfrac{H}{V}$ in $H$.
  \item \emph{Subspace of $V$:} supercriticality of the subspace $U$ of $V$ follows from supercriticality of the subspace $U$ of $H$.
\end{itemize}
\end{itemize}
\end{enumerate}

\medskip

Next we give the arguments in details.
In the first part we assume that $V$ is a critical subspace:  $V \subseteq \ker B_{m+1}$, and thus the map $b_{m+1}$ is trivial, and the following equality holds
\begin{equation}\label{eq:proof-inheritenceC-critical-subspace}
\dim V = \sum_{k=1}^m c_k \dim B_kV.
\end{equation}
Let us  check Condition (C) for $V$ equipped with the maps $(b_k)$ and
 the coefficients $(c_k)$ (for shortness, we will write $(V,b)$, since the coefficients
 $c$ are the same for all sub-structures).
 
  Suppose that
\[ \begin{CD} 0 @>>> U @>>> V @>>> \sfrac{V}{U} @>>> 0\end{CD} \]
is an admissible split of $(V,b)$. Since $U \subseteq V=\ker b_{m+1} $, we must check supercriticality of $U$ as a subspace of $(V,b)$. 
  The admissible split of $(V,b)$ induced by $U$ obviously leads to an admissible
  split of $(H,B)$. Moreover,
  $U \subseteq V \subseteq \ker B_{m+1}$, so that Condition (C) for $(H,B)$ 
  yields
  $$ \dim U\ge \sum_{k=1}^m c_k \dim B_k U.$$
  Using $B_iU=b_iU$, we conclude that $U$ is a supercritical subspace of $(V,b)$. 

%In fact supercriticality of $U$ follows immediately from supercriticality of $U$ as a subspace of $H$, which holds true since $U \subseteq \ker B_{m+1}$ and the split of $H$ via $U$ is clearly admissible.

For the same $U$, $\sfrac{V}{U}$ is a subcritical quotient of $(V,b)$ (even regardless whether $b_0(U) = b_0(V)$ or not) because $V$ is a critical subspace of $H$ and $U$ is a supercritical subspace of $V$ (subtract the inequality $ \dim U\ge \sum_{k=1}^m c_k\dim b_kU$ from \eqref{eq:proof-inheritenceC-critical-subspace}).

\medskip

Now we check Condition (C) for $(\sfrac{H}{V},\beta)$. Suppose
\[ \begin{CD} 0 @>>> U @>>> \sfrac{H}{V} @>>> \sfrac{\sfrac{H}{V}}{U} @>>> 0\end{CD} \]
is an admissible split of $(\sfrac{H}{V},\beta)$, which by Lemma~\ref{lem:admissibility-of-subspace-and-quotient} (i) $\iff$ (iv) means that
\[
  \bigcap_{i=0}^{m^+} (U + \ker \beta_i) = U.
\]
Taking the preimage w.r.t. $\pi \colon H \to \sfrac{H}{V}$ we get
\[
  \bigcap_{i=0}^{m^+} (\pi^{-1}(U) + V + \ker B_i) = \pi^{-1}(U),
\]
where we used~\eqref{eq:kernel-of-beta} and the relation $\pi^{-1}(A+B)=\pi^{-1}(A)+\pi^{-1}(B)+\ker \pi$ which is valid for any linear surjective map. Denote $\tilde{U} = \pi^{-1}(U)$. Of course $\tilde{U}$ contains $V$, hence the above assertion means that 
\[ \begin{CD}
0 @>>> \tilde{U} @>>> H @>>> \sfrac{H}{\tilde{U}} @>>> 0
\end{CD} \]
is an  admissible split of $(H,B)$ (again use Lemma~\ref{lem:admissibility-of-subspace-and-quotient} (i) $\iff$ (iv)).

We need to check supercriticality of $U$ as a subspace of $\sfrac{H}{V}$ whenever $U \subseteq \ker \beta_{m+1}$. 
%To this end we shall refer to the split of $(H,B)$ via $\tilde{U}$.
%We claim that $\tilde{U} \subseteq \ker B_{m+1}$, which in turn implies that $\tilde{U}$ is a supercritical subspace of $H$.
 By criticality of $V$ in $H$ we have $V \subseteq \ker B_{m+1}$, which combined with the assertion $U \subseteq \ker \beta_{m+1}$ and~\eqref{eq:kernel-of-beta} yields
\[
  \tilde{U} = \pi^{-1}(U) \subseteq \pi^{-1}(\ker \beta_{m+1}) = V + \ker B_{m+1} = \ker B_{m+1}.
\]
Applying Condition (C) for $(H,B)$, we know that:
\[\dim \tilde{U}\ge \sum_{k=1}^m c_k \dim B_k\tilde{U}.\]
If we subtract from the last inequality  the relationship \eqref{eq:proof-inheritenceC-critical-subspace} corresponding to criticality of $V$ in $(H,B)$,
we get  the supercriticality of $U$ in $(\sfrac{H}{V},\beta)$.
Indeed, up to isomorphism 
\begin{equation}\label{eq:identification-quotients}
 U\approx \sfrac{\tilde U}{V} \quad \mathrm{and} \quad\beta_k U\approx \sfrac{B_k\tilde U}{B_kV},
\end{equation} 
as one readily checks by considering the ranges and kernels of the maps $\pi:\tilde{U}\to U$
and $\phi: B_k\tilde{U} \to \sfrac{B_kH}{B_kV}$ defined by $\phi(x)=x+B_kV$. 
\medskip

Now we check subcriticality of the quotient $\sfrac{\sfrac{H}{V}}{U}$ of $\sfrac{H}{V}$ whenever $\beta_0(U) = \beta_0(\sfrac{H}{V})$, i.e. $U + \ker \beta_0 = \sfrac{H}{V}$.
% Again, we shall use the split of $H$ via $\tilde{U}$. 
Notice that using~\eqref{eq:kernel-of-beta}, we have
\[
  H = \pi^{-1}(U + \ker \beta_0) = \tilde{U} + V + \ker B_0 = \tilde{U} + \ker B_0
\]
(the last equality follows from the fact that $V \subseteq \tilde{U}$), that is $B_0(\tilde{U}) = B_0(H)$. Therefore, by Condition (C) for $(H,B)$, the quotient $\sfrac{H}{\tilde{U}}$  must be subcritical in $(H,B)$, that is 
\[
\dim  \sfrac{H}{\tilde{U}} \le \sum_{k=1}^m c_k \dim \sfrac{B_k H}{B_k \tilde U}.
\]
Using again \eqref{eq:identification-quotients} together with the relation $\beta_k \sfrac{H}{V}=\sfrac{B_kH}{B_kV}$, we may rewrite the above inequality as
\[
\dim  \sfrac{\sfrac{H}{V}}{U} \le \sum_{k=1}^m c_k \dim \sfrac{\beta_k\sfrac{H}{V}}{\beta_k U}.
\]
In other words, $\sfrac{\sfrac{H}{V}}{U}$ is a subcritical quotient of $(\sfrac{H}{V},\beta)$.

\bigskip

In the second part, suppose $\sfrac{H}{V}$ is a critical quotient of $(H,B)$. More specifically, $B_0(V) = B_0(H)$, i.e. $\beta_0$ is trivial, and
\begin{equation}\label{eq:proof-inheritenceC-critical-quotient}
\dim \sfrac{H}{V} = \sum_{k=1}^m c_k \dim \sfrac{B_kH}{B_kV}.
\end{equation}
 
 The reasoning below is analogous to the first part after interchanging subspaces with quotient and supercriticality with subcriticality.

We check Condition (C) for $(\sfrac{H}{V},\beta)$. Consider a split of $\sfrac{H}{V}$ via its subspace $U$ and suppose this split is admissible. We need to check subcriticality of $\sfrac{\sfrac{H}{V}}{U}$ in $\sfrac{H}{V}$ whatever $U$ is, because $\beta_0$ is trivial and so $\beta_0(U) = \beta_0(\sfrac{H}{V})$ always holds. To this end, consider a split of $H$ via $\tilde{U} = \pi^{-1}(U)$, which is admissible, as we already showed. Moreover,
\[
  B_0(\tilde{U}) \supseteq B_0(V) = B_0(H),
\]
so we can use subcriticality of $\sfrac{H}{\tilde{U}}$ in $(H,B)$. From the latter, subcriticality of 
$\sfrac{\sfrac{H}{V}}{U}$ in $(\sfrac{H}{V},\beta)$ follows, as we have already explained.

For the same $U$, $U$ is a supercritical subspace of $\sfrac{H}{V}$ (regardless whether $U \subseteq \ker \beta_{m+1}$ or not), because $\sfrac{H}{V}$ is a critical quotient of $H$ and $\sfrac{\sfrac{H}{V}}{U}$ is a subcriticial quotient of $\sfrac{H}{V}$ (write the corresponding equality and inequality relations and subtract them).

\medskip

Next, we check Condition (C) for $(V,b)$. Consider a split of $V$ via a subspace $U$ of $V$ which is admissible. We use an induced split of $H$ via $U$, which is also admissible.

We need to check subcriticality of $\sfrac{V}{U}$ in $V$ whenever $b_0(U) = b_0(V)$. Since we know that
\[
  B_0(U) = b_0(U) = b_0(V) = B_0(V) = B_0(H)
\]
(the last equality is due to criticality of $\sfrac{H}{V}$ in $H$), we can use subcriticality of $\sfrac{H}{U}$ in $H$. Writing the corresponding inequality for dimensions and subtracting from it the equality \eqref{eq:proof-inheritenceC-critical-subspace} related to criticality of $\sfrac{H}{V}$ in $(H,B)$, we get subcriticality of $\sfrac{V}{U}$ in $(V,b)$.

Finally we check supercriticality of $U$ in $(V,b)$ whenever $U \subseteq \ker b_{m+1}$. Since the latter implies $U \subseteq \ker B_{m+1}$, we can invoke the fact that $U$ is a supercritical subspace of $H$ to conclude.
\end{proof}

Our next result is about tensorization through a split. We say that $(H,B,c)=\big(H,(B_k)_{k=0}^{m+1},(c_k)_{k=1}^{m}\big)$ admits the \emph{strong positivity property} if for every positive definite quadratic forms $\mathcal Q_+$ on $H_0$ and $\mathcal Q_-$ on $H_{m+1}$,
 \[ 
 \inf_{(f_1,\ldots,f_m)} \frac{\int_H e^{-\mathcal Q_+(B_0x)+\mathcal Q_-(B_{m+1}x)} \prod_{k=1}^m f_k^{c_k}(B_kx)\, dx}{\prod_{k=1}^m \left(\int_{H_k} f_k\right)^{c_k}}>0.
 \]

\begin{lemma}[Tensorization]\label{lem:tensorization}
Assume that $(H,B)$ satisfies~\eqref{eq:isomorphism-condition} and~\eqref{eq:kerBplus-contained-in-kerBm1}.
Let $V$ be a linear subspace of $H$ which induces an admissible split. 

If $(V,b,c)$ and $(\sfrac{H}{V},\beta,c)$ have the strong positivity property, then $(H,B,c)$ has it too.
\end{lemma}

\begin{proof}
Fix any non-negative, integrable functions $f_k \colon H_k \to \R$ ($k=1, \ldots, m$) satisfying $\int_{H_k} f_k > 0$.
Set $f_0=e^{-\mathcal Q_+}$, $f_{m+1}=e^{-\mathcal Q_-}$, $c_0=1$ and $c_{m+1}=-1$. 
Our goal is to bound from below the quantity $J_{\mathcal Q_+,\mathcal Q_-}(f_1,\ldots,f_m)$ defined in \eqref{eq:J-Qplus-Qminus} by a positive
constant (not depending on $(f_1,\ldots,f_m)$). Recall the discussion from Section~\ref{sec:prelim-results-characterization-theorem} and in particular Formula~\eqref{eq:J-iterative-integral}. Our aim is to bound from below the quantity
\begin{equation}\label{eq:proof-tensorization1}
I:=\int_{\sfrac{H}{V}} \int_V  \prod_{k=0}^{m+1} f_k^{c_k}(b_k x + \rho_k y, \beta_k y) \, dx \, dy,
\end{equation}
by application of an inequality of inverse Brascamp-Lieb type on $V$ and on $\sfrac{H}{V}$.

 By Corollary \ref{cor:inheritance-of-kernel-conditions}, $b_{0+}=(b_0,b_+)$ is surjective and $\ker b_+\subseteq \ker b_{m+1}$.
 Hence, by Lemma~\ref{lem:kerS-kerT-H}
\[ V=\ker b_0+\ker b_+\subseteq \ker b_0+\ker b_{m+1} \subseteq V.\]
Using Lemma~\ref{lem:kerS-kerT-H} once again, we obtain that $(b_0,b_{m+1})$ is surjective. This allows us to remove the cross-terms
from the Gaussian kernel: indeed for every $y\in \sfrac{H}{V}$, there exists $v_y\in V$ such that $b_0v_y=\rho_0 y$ and 
$b_{m+1}v_y=\rho_{m+1}y$.  Using the translation invariance of the Lebesgue measure on $V$, we apply the change of variable $V \ni x \mapsto x - v_y \in V$ to the inner integral of \eqref{eq:proof-tensorization1}, and get that it is equal to 
\begin{equation}\label{eq:proof-tensorization2}
  \int_{\sfrac{H}{V}} \int_V f_0(b_0 x, \beta_0 y) f_{m+1}^{-1}(b_{m+1}x, \beta_{m+1} y) \prod_{k=1}^m f_k^{c_k}(b_k x - b_k v_y + \rho_k y, \beta_k y) \, dx \, dy.
\end{equation}
Next, we bound from below the Gaussian kernel by a product kernel. 
Since $f_0=\exp(-\mathcal Q_+)$ where  $\mathcal Q_+$ is viewed as a quadratic form on $V_0 \times \sfrac{H_0}{V_0}$, we can bound $f_0$ from below by
\[ %begin{equation}\label{eq:domination-of-quadratic-forms}
  f_0(x_0, y_0) \ge f_{0, V}(x_0) f_{0, \sfrac{H}{V}}(y_0),
\] %end{equation}
where $f_{0, V} = \exp(-\mathcal Q_{+, V})$ and $f_{0, \sfrac{H}{V}} = \exp(-\mathcal Q_{+, \sfrac{H}{V}})$ for some positive definite quadratic forms $\mathcal Q_{+, V} \colon V_0 \to \R$ and $\mathcal Q_{+, \sfrac{H_0}{V_0}} \colon \sfrac{H}{V} \to \R$.

 For $f_{m+1}$ we use a reverse bound, namely for some positive definite quadratic forms $\mathcal Q_{-, V} \colon V_{m+1} \to \R$ and $\mathcal Q_{-, \sfrac{H_{m+1}}{V_{m+1}}} \colon \sfrac{H}{V} \to \R$ we have
\[   %begin{equation}\label{eq:domination-of-quadratic-forms-2}
  f_{m+1}^{-1}(x_0, y_0) \ge f_{m+1, V}^{-1}(x_0) f_{m+1, \sfrac{H}{V}}^{-1}(y_0),
\] %end{equation}
where $f_{m+1, V} = \exp(-\mathcal Q_{-, V})$ and $f_{m+1, \sfrac{H}{V}} = \exp(-\mathcal Q_{-, \sfrac{H}{V}})$.
Observe that we have used here the fact that $\mathcal Q_-$ is positive definite. We get that $I$ from \eqref{eq:proof-tensorization1} is at least 
\[
  \int_{\sfrac{H}{V}} f_{0,\sfrac{H}{V}}(\beta_0 y) f_{m+1,\sfrac{H}{V}}^{-1}( \beta_{m+1} y)\int_V   f_{0,V}(b_0x) f_{m+1,V}^{-1}(b_{m+1}x)\prod_{k=1}^m f_k^{c_k}(b_k x - b_k v_y + \rho_k y, \beta_k y) \, dx \, dy.
\]
By the strong positivity property for $(V,b,c)$ there exists a constant $C_V>0$ such that for all $y \in \sfrac{H}{V}$,
\[ \begin{split}
\int_V f_{0, V}(b_0 x) & f_{m+1,V}^{-1}(b_{m+1}x)\prod_{k=1}^m f_k^{c_k}(b_k x - b_k v_y + \rho_k y, \beta_k y) \, dx\\
 &\ge
C_V \prod_{k=1}^m \left(\int_{V_k} f_k(\cdot - b_k v_y + \rho_k y, \beta_k y) \right)^{c_k} \\
&= C_V \prod_{k=1}^m \left(\int_{V_k} f_k(\cdot, \beta_k y) \right)^{c_k},
\end{split} \]
where the equality follows from translation invariance of the Lebesgue measure on each $V_k$. 
Denoting $f_{k, \sfrac{H}{V}}(y) := \int_{V_k} f_k(\cdot, y)$ for $y \in \sfrac{H_k}{V_k}$ ($k=1, \ldots, m$),  we obtain
\begin{equation}\label{eq:J-iterative-integral-V-critical-2}
 I \ge
  C_V \int_{\sfrac{H}{V}} f_{0, \sfrac{H}{V}}(\beta_0 y) f_{m+1, \sfrac{H}{V}}^{-1}(\beta_{m+1} y) \prod_{k=1}^m f_{k, \sfrac{H}{V}}^{c_k}(\beta_k y) \,dy.
\end{equation}
Now it remains to apply the strong positivity property for $(\sfrac{H}{V}, \beta, c)$ and the functions $f_{k, \sfrac{H}{V}}$ in order to get
\[
 I \ge
  C_V C_{\sfrac{H}{V}} \prod_{k=1}^m \left( \int_{\sfrac{H_k}{V_k}} f_{k, \sfrac{H}{V}} \right)^{c_k} =
  C_V C_{\sfrac{H}{V}} \prod_{k=1}^m \left(\int_{\sfrac{H_k}{V_k}} \int_{V_k} f_k(x,y) \,dx\,dy \right)^{c_k}
\]
for some constant $C_{\sfrac{H}{V}} > 0$ (which depends on $\mathcal Q_{+, \sfrac{H}{V}}$ and $\mathcal Q_{-, \sfrac{H}{V}}$).
\end{proof}

\begin{prop}\label{prop:c-i-ge-1}
Let $0 \le m^+ \le m$ be integers and consider surjective maps $B_k \colon H \to H_k$ for $k = 0, 1, \ldots, m+1$ and real numbers $c_k$ such that $c_k > 0$ for $k = 1, \ldots, m^+$ and $c_k \le 0$ for $k = m^+ + 1, \ldots, m$. Assume that Conditions~\eqref{eq:isomorphism-condition} and  (C) hold.  Then
\begin{itemize}
\item for all $i= 1, \ldots, m^+$, $
  \dim H_i > 0 \implies c_i \ge 1$
\item if $H$ is a critical subspace, then $B_0=B_{m+1}=0$.
\end{itemize}
\end{prop}
\begin{proof}
Fix $1 \le i\le m^+$ such that $\dim H_i > 0$, i.e. $\ker B_i \neq H$. Consider $V = \ker B_i$ and a related split of $H$ by $V$. Since clearly
\[
  \bigcap_{k=0}^{m^+} (V + \ker B_k) = V,
\]
by Lemma~\ref{lem:admissibility-of-subspace-and-quotient}, the split is admissible. Moreover, since $B_{0+}$ is surjective, the map $(B_0, B_i)$ is surjective too, hence Lemma~\ref{lem:kerS-kerT-H} yields
\[
  H = \ker B_0 + \ker B_i = \ker B_0 + V,
\]
i.e. $B_0 V = B_0 H$. Therefore we can use the fact that the quotient $\sfrac{H}{V}$ of $H$ is subcritical, from which it follows that 
\begin{equation}\label{eq:consequence-of-subcriticality}
  \dim \sfrac{H}{V} \le \sum_{k=1}^m c_k \dim \sfrac{H_k}{V_k} \le \sum_{k=1}^{m^+} c_k \dim \sfrac{H_k}{V_k}.
\end{equation}
For all $1 \le k \le m^+$ with $k \neq i$, the map $(B_i, B_k)$ is surjective, hence again by Lemma~\ref{lem:kerS-kerT-H},
\[
  H = \ker B_i + \ker B_k = V + \ker B_k,
\]
which means that $B_k V = B_k H$, i.e. $\sfrac{H_k}{V_k} = \{0\}$. Thus \eqref{eq:consequence-of-subcriticality} boils down to  
$$\dim \sfrac{H}{V} \le c_i 
\dim \sfrac{H_i}{V_i}.$$
Recall that $V_i=B_iV$ is reduced to $\{0\}$ since by definition $V=\ker B_i$.  Moreover
\[
  \dim \sfrac{H}{V}= \dim H- \dim  \ker B_i = \dim B_iH=\dim H_i,
\]
so the last inequality can be rewritten as $\dim H_i\le c_i \dim H_i$. Therefore
 $c_i\ge 1$ if $\dim H_i > 0$
 
 \medskip
 The proof of the second item  follows the same lines. Firstly, $H$ is admissible by hypothesis. Since it is assumed to be a critical subspace, we know that $H\subset \ker B_{m+1}$ hence $B_{m+1}=0$, and that 
 \[ \dim H=\sum_{k=1}^m c_k \dim B_kH .\]
We set $V=\ker B_0$. As above, we can check that $V$ is admissible. Since $V\subset H=\ker B_{m+1}$, it is a supercritical subspace thanks to Condition (C). Therefore, using the above dimension equality, we get after subtraction
\begin{equation}\label{eq:consequence-of-subcriticality2}
  \dim \sfrac{H}{V} \le \sum_{k=1}^m c_k \dim \sfrac{B_kH}{B_kV} \le \sum_{k=1}^{m^+} c_k \dim \sfrac{B_kH}{B_kV}  
\end{equation}
again. Since $B_{0+}$ is a bijection and $V$ is admissible, we know by Lemma~\ref{lem:admissibility-of-subspace-and-quotient} that 
$\dim H=\sum_{i=0}^{m^+} \dim B_iH$ and $\dim V=\sum_{i=0}^{m^+} \dim B_iV$. Hence $\dim \sfrac{H}{V}=\sum_{i=0}^{m^+} \dim \sfrac{B_iH}{B_iV}$.
Plugging this equality into \eqref{eq:consequence-of-subcriticality2} yields after rearranging
\begin{equation}\label{eq:consequence-of-subcriticality3}
\dim \sfrac{B_0H}{B_0V}\le  \sum_{i=1}^{m^+} (c_i-1) \dim \sfrac{B_iH}{B_iV}.
\end{equation}
Since $B_{0+}$ is surjective, the map $(B_0, B_i)$ is surjective too for any $1\le i\le m^+$. Hence Lemma~\ref{lem:kerS-kerT-H} yields
$ H = \ker B_0 + \ker B_i = V+\ker B_i$, which ensures that $B_iH=B_iV$. Therefore, \eqref{eq:consequence-of-subcriticality3} becomes
$ \dim \sfrac{B_0H}{B_0V}\le 0$. Recall that by definition $B_0V=\{0\}$. We can conclude that $\dim B_0H=0$, that is $B_0=0$. 
\end{proof}

The next statements  will help to initialize the inductive proof of Theorem~\ref{thm:positivity-of-constant-general-case} (ii).

\begin{lemma}\label{lem:initialization-dimension1}
Assertion (ii) of Theorem   \ref{thm:positivity-of-constant-general-case}
is true when $\dim H=1$.
\end{lemma}
\begin{proof}
 The main tool here  is the reverse H\"older inequality for several functions:
 Let $c_1\ge 0\ge c_2,\ldots c_m$ with $\sum_k c_k=1$ then 
 \begin{equation}\label{eq:inverse-Holder-m-functions}
 \int_{\R^d} \prod_k f_k^{c_k}\ge \prod_k \left( \int_{\R^d}  f_k\right)^{c_k}
  \end{equation}
 holds for all integrable non-negative functions with $\int_{\R^d} f_k\in (0,+\infty)$.
This inequality follows from its version for two functions applied with $\lambda=c_1\ge 1$:
$$ \int_{\R^d} \prod_k f_k^{c_k}\ge \left( \int f_1\right)^{c_1} \left( \int \prod_{j=2}^m f_j^{\frac{c_j}{c_2+\cdots+c_m}}\right)^{c_2+\cdots+c_m},$$
 and from the classical H\"older inequality applied to the second integral
 (observe that the inner exponents sum up to 1 and are all non-negative, while the
 outer exponent $c_2+\cdots+c_m$ is non-positive).
 
 \medskip
 Since $\dim H=1$ and for $1\le k\le m$, $B_k:H\to H_k$ is surjective and $H_k$ is non-trivial, it follows that the maps $B_k$, $k\ge 1$ are bijections.
 Therefore we may reduce to the case $H=H_k=\R$ and $B_k=\mathrm{Id}$ for $1\le k\le m$. In this simple setting, the only possible subspaces $V$ are ${0}$ and $\R$. The former is trivially admissible, while the latter is admissible by hypothesis. 
Hence Condition (C) rewrites as:
\begin{itemize}
	\item if $\R \subset \ker B_{m+1}$ (i.e. $B_{m+1}=0$), then $ 1\ge \sum_{k=1}^m c_k$,
	\item if $B_0\{0\}=B_0\R$ (i.e. $B_0=0$) then  $ 1\le \sum_{k=1}^m c_k$,
\end{itemize}
Also the hypothesis of bijectivity of $(B_0,B_+)$ reduces to two cases:
either $B_0=0$, $B_+=B_1$ and $c_1\ge0\ge c_2,\ldots, c_m$,
or $B_0\neq 0$, $B_+=0$  and $0\ge c_1,\ldots,c_m$. 

In order to prove the lemma, we consider several cases:

Case 1: If $B_0=B_{m+1}=0$, then Condition (C) rewrites as $\sum_{k=1}^m c_k=1$.
Moreover there is no kernel and, as explained above $c_1\ge0\ge c_2,\ldots, c_m$.
The positivity of the Brascamp-Lieb functional is a direct consequence of the reverse H\"older inequality \eqref{eq:inverse-Holder-m-functions}.

Case 2: if $B_0\neq 0$ and $B_{m+1}=0$, then $0\ge c_1,\ldots, c_m$ and Condition (C) amounts to $1\ge \sum_{k=1}^m c_k$. We define $c_0:=1- \sum_{k=1}^m c_k\ge 1$ and we are ready to apply the inverse H\"older inequality
with $m+1$ functions:
$$ \int e^{-\mathcal Q_+(B_0x)}\prod_{j=1}^m f_j(x)^{c_j} dx
\ge \left(\int e^{-\frac{1}{c_0}\mathcal Q_+(B_0x)}\right)^{c_0} \prod_{j=1}^m \left(\int  f_j\right)^{c_j}.$$
 
 Case 3: if $B_0=0$ and $B_{m+1}\neq 0$, then $c_1\ge0\ge c_2,\ldots, c_m$ and
 Condition (C) reads as $1\le \sum_{k=1}^m c_k$ (actually the inequality is strict. If it where an equality then $H=\R$ would be a critical space, which is not compatible with $B_{m+1}\neq 0$ as explained by Proposition \ref{prop:c-i-ge-1}). We define $c_{m+1}:=1- \sum_{k=1}^m c_k<0$ and we  apply the inverse H\"older inequality
with $m+1$ functions:
$$ \int e^{\mathcal Q_-(B_{m+1}x)}\prod_{k=1}^m f_k(x)^{c_k} dx
\ge \left(\int e^{\frac{1}{c_{m+1}}\mathcal Q_-(B_0x)}\right)^{c_{m+1}} \prod_{k=1}^m \left(\int  f_k\right)^{c_k}.$$
Since $c_{m+1}<0$ and $\mathcal Q_-$ is positive definite, the first integral of the right-hand side term is finite.

Case 4: $B_0\neq 0$ and $B_{m+1}\neq 0$ does not happen. Indeed it implies that 
$B_+=0$ but then the condition $\ker B_+\subset \ker B_{m+1}$ is violated. 
A more conceptual explanation is that a quadratic form on $\R$ is either zero, definite positive or definite negative, so that the above three cases cover all possibilities.

\end{proof}

\begin{lemma}\label{lem:initialization-2functions-no-kernel}
Assertion (ii) of Theorem   \ref{thm:positivity-of-constant-general-case}
is true when $m=2$ and $B_0=B_{m+1}=0$.
\end{lemma}
\begin{proof}
   Our goal is to prove the positivity of the Brascamp-Lieb functional for two functions and no kernel. Our hypothesis is that $B_+$ is bijective and that Condition (C) holds. Since $1\le m^+\le m=2$ we can consider two cases:

   Case 1: $m=m^+=2$. Proposition \ref{prop:c-i-ge-1} yields $c_1, c_2\ge 1$.
   Condition (C) ensures that $H$ is a critical space, hence
   $$ \dim H=c_1\dim H_1+c_2\dim H_2\ge \dim H_1+\dim H_2=\dim H,$$
   where the latter inequality comes from the fact that $B_+=(B_1,B_2):H\to H_1\times H_2$ is a linear isomorphism.  The intermediate inequality cannot be strict, therefore $c_1=c_2=1$. The inverse Brascamp-Lieb inequality in this case follows from Fubini theorem, after changing variables:
\begin{eqnarray*}   
    \int_H f_1(B_1x)f_2(B_2x) \,dx&=& |\det ((B_1,B_2))|^{-1} \int_{H_1\times H_2}
   f_1(y)f_2(z) \,dy\,dz\\
   &= &
   |\det ((B_1,B_2))|^{-1} \int_{H_1} f_1  \int_{H_2} f_2.
\end{eqnarray*}  
  Case 2: $m^+=1$ and therefore $B_1$ is bijective. Any linear subspace is admissible in this case ($\dim V=\dim B_1V$). Thus, for any subspace $V$, Condition (C) yields 
  $$ \dim V\ge c_1 \dim B_1V+c_2 \dim B_2 V= c_1 \dim V+  c_2 \dim B_2 V,$$
  and after rearranging the terms
  $$ (c_1-1) \dim V\le |c_2| \dim B_2V.$$
  Choosing $V=\ker B_2$, we get that $(c_1-1) \dim \ker B_2=0$.
  
  Subcase 1:  If $\ker B_2=0$, then $B_2$ is an isomorphism. Since $B_1$ is also 
  an  isomorphism, the relation $\dim H=c_1\dim H_1+c_2\dim H_2$ implies
  $c_1+c_2=1$. Recall that $c_1\ge 0\ge c_2$. We can  conclude with the inverse  H\"older inequality:
  \begin{eqnarray*} \int f_1(B_1x)^{c_1} f_2(B_2 x)^{c_2} dx &\ge& \left(\int f_1(B_1x) dx \right)^{c_1}
   \left(\int f_2(B_2x) dx \right)^{c_2}\\
  & =& \left( |\det B_1|^{-1}\int_{H_1} f_1 \right)^{c_1} \left( |\det B_2|^{-1}\int_{H_2} f_2 \right)^{c_2}.
  \end{eqnarray*} 
  
  Subcase 2:  $c_1=1$. Using also that $\dim H=\dim H_1$, the equality $\dim H=c_1\dim H_1+c_2\dim H_2$ implies that $c_2 \dim H_2=0$, hence $c_2=0$. The inverse Brascamp-Lieb inequality is trivial in this case:
$\int f_1(B_1x) dx= |\det B_1|^{-1} \int f_1$.

\end{proof}

\begin{lemma}\label{lem:initialization-1function}
Assertion (ii) of Theorem   \ref{thm:positivity-of-constant-general-case}
is true when $m=0$ and when $m=1$.
\end{lemma}
\begin{proof}
Since $0\le m^+\le m\le 1$, we consider three cases.

Case 0: when $m^+=m=0$, i.e. there are no functions $f_k$. By hypothesis $\mathcal Q$ is positive definite. Condition (C) is empty. The conclusion holds as $\int e^{-\mathcal Q}>0$.

Case 1: when $m^+=0$, $c_1\le 0$, $B_+=0$ and by hypothesis $B_0$ is a linear isomorphism.
Moreover the condition $\ker B_+\subset \ker B_{m+1}$ implies that $B_{m+1}=0$.
In this setting, Condition (C) is empty. Indeed the inequality  $\dim V\ge c_1 \dim B_1V$ is valid for every subspace since $c_1\le 0$. In addition, $B_0$ being and isomorphism, the only subspace $V$ such that $B_0V=B_0H$ is $H$ (and the quotient dimension condition is empty). Consequently, our task is to show that
$$\inf_{f_1} \frac{\int_H e^{-\mathcal Q_+(B_0x)} f_1(B_1x)^{c_1} dx}{\left(\int_{H_1}f_1\right)^{c_1}}>0.$$
The case $c_1=0$ is obvious since $\mathcal Q_+\circ B_0$ is positive definite. Next, we assume that $c_1<0$. By definition $B_1:H\to H_1$ is surjective. We complete it 
to a bijective map $\Phi:H\to H_1\times \tilde{H}$ of the form $\Phi(x)=(B_1(x),\tilde B(x))$. Using the bijective change of variables $x=\Phi^{-1}(y,\tilde y)$, there exits $\alpha\in (0,+\infty)$ such that for any $f_1$,
$$ 
\int_H e^{-\mathcal Q_+(B_0x)} f_1(B_1x)^{c_1} dx = \alpha \int_{H_1\times \tilde H} e^{-\mathcal Q_+\circ B_0 \circ \Phi^{-1}(y,\tilde y)} f_1(y)^{c_1} dy d\tilde y.
$$
There exists positive definite quadratic forms $\mathcal Q_1$ on $H_1$ and  $\tilde{\mathcal Q}$ on $\tilde H$ such that for all $(y,\tilde y)$, 
$$\mathcal Q_+\circ B_0 \circ \Phi^{-1}(y,\tilde y) \le \mathcal Q_1(y)+\tilde{\mathcal  Q}(\tilde y).$$
Therefore the latter integral is at most 
$$ \int_{\tilde H} e^{-\tilde{\mathcal Q}(\tilde y)} d\tilde y  \int_{H_1} e^{-\mathcal Q_1(y)} f^{c_1}(y) dy
\ge \left(\int_{\tilde H} e^{-\tilde{\mathcal  Q}(\tilde y)} d\tilde y\right)
\times \left(   \int_{H_1} e^{-\frac{1}{1-c_1}\mathcal Q_1(y)}dy\right)^{1-c_1} 
\left( \int_{H_1} f_1\right) ^{c_1},
$$
where the latter inequality is a consequence of the inverse H\"older inequality.

\medskip
Case 2: $m^+=m=1$. In this case $c_1\ge 0$, $(B_0,B_1)$ is bijective and $\ker B_+=\ker B_1\subset \ker B_{m+1}$.
We may assume that $\ker B_1\neq H$, otherwise $H_1=\{0\}$ and we can discard the function $f_1$ and we are back to Case 0.

Condition (C) asserts that every admissible subspace $V$  verifies:
\begin{itemize}
	\item if $V\subset \ker B_{m+1}$ then $\dim V\ge c_1\dim B_1V$
	\item if $B_0V=B_0H$ then $\dim \sfrac{H}{V}\le c_1 \dim \sfrac{B_1H}{B_1V}$
\end{itemize}
Observe that the latter ``quotient condition'' boils down to $c_1\ge 1$:
Indeed by hypothesis $\dim H= \dim B_0H+ \dim B_1H$,  and $V$ is admissible
if and only if $\dim V=\dim B_0V+\dim B_1V$. Therefore, when $V$ also satisfies
that $B_0V=B_0H$, taking the difference of the latter two dimension equalities yields
$ \dim \sfrac{H}{V}= \dim \sfrac{B_1H}{B_1V}$.  Hence the condition
$\dim \sfrac{H}{V}\le c_1 \dim \sfrac{B_1H}{B_1V}$   becomes $(c_1-1)\dim\sfrac{H}{V}\ge 0$ which can be an empty condition (when $H=V$) or equivalent to
$c_1\ge 1$ e.g. for $V=\ker B_1$ (see the argument of Proposition \ref{prop:c-i-ge-1}).

Given $(B_0, B_1, B_{m+1}=B_2)$, the set $\mathcal C_1$ of indices $c_1\ge 0$ satisfying Condition 
(C) is clearly a closed convex subset of $[1,+\infty)$. Indeed, it is defined by the inequality  $c_1\ge 1$ and conditions of the form $\dim V\ge c_1\dim B_1V$ (and there are finitely many of them since the dimensions are bounded).
Obviously $1\in \mathcal C_1$. For $c_1=1$, the corresponding Brascamp-Lieb inequality holds with a positive constant. 
Indeed for every non-negative function $f_1$,
\begin{eqnarray*}
\int_H e^{-\mathcal Q_+(B_0x)+\mathcal Q_-(B_2 x)} f_1(B_1x)\, dx&\ge& \int_H e^{-\mathcal Q_+(B_0x)} f_1(B_1x) \, dx\\
&=& |\det((B_0,B_1))|^{-1} \int_{H_0} e^{-\mathcal Q_+} \int_{H_1} f_1,
\end{eqnarray*}
where we have used the bijection $(B_0,B_1)$ in order to change variables.

Consider the subspace $V=\ker B_0\cap \ker B_{m+1}$. Using our hypothesis $\ker B_1\subset \ker B_{m+1}$, we get that
$$ V\subset (V+\ker B_0)\cap (V+\ker B_1)\subset \ker B_0\cap \ker B_{m+1}=V.$$
Hence, by \eqref{lem:admissibility-of-subspace-and-quotient}, $V\subset \ker B_{m+1}$ is admissible. Consequently $\dim V=\dim B_0V+\dim B_1V=\dim B_1V$
and any $c_1\in \mathcal C_1$ verifies $\dim V\ge c_1 \dim B_1V$ which can be rewritten as $0\ge (c_1-1)\dim B_1V$.

Subcase 1: if $B_1V\neq\{0\}$ then the latter inequality implies that $c_1\le 1$. We have shown that  $\mathcal C_1=\{1\}$
and we have established a non-trivial inverse Brascamp-Lieb inequality for $c_1=1$.

Subcase 2: if $B_1V=\{0\}$. This condition can be rephrased as $V\subset \ker B_1$.  From this, we deduce that
$$ V=\ker B_0\cap \ker B_{m+1} \subset \ker B_0\cap \ker B_1=\{0\},$$
where the last equality comes from the injectivity of $(B_0,B_1)$. The latter is actually  bijective so that 
$$ \ker B_0 \oplus \ker B_1=H.$$
 Since $\ker B_1\subset \ker B_{m+1}$, and $V=\ker B_0\cap \ker B_{m+1}=\{0\}$,
 we also have 
  $$ \ker B_0 \oplus \ker B_{m+1}=H.$$ 
The previous two decompositions of $H$ into direct sums, and the inclusion $\ker B_1\subset \ker B_{m+1}$ imply that $\ker B_1=\ker B_{m+1}$. Because of this equality, the subspace constraint in Condition (C) is empty: indeed, if $V\subset \ker B_{m+1}=\ker B_1$ then $B_1V=0$ and $\dim V\ge c_1 \dim B_1V=0$ is true.
Therefore the set of numbers $c_1$ verifying Condition (C) is $[1,+\infty)$ and
our task is to prove a non-trivial inverse Brascamp-Lieb inequality for all exponents
$c_1\ge 1$. We have already dealt with $c_1=1$, so we may restrict our attention to $c_1>1$.
 Observe that the equality $\ker B_1=\ker B_{m+1}$ ensures the existence of a linear isomorphism $\Psi:H_1\to H_{m+1}$ such that $B_{m+1}=\Psi\circ B_1$. Thus for every non-negative function $f_1$:
 \begin{eqnarray*}
  \lefteqn{\int_H e^{-\mathcal Q_+(B_0x)-\mathcal Q_-(B_{m+1}x)} f_1^{c_1}(B_1 x)\, dx }\\
    &=& \int_H  e^{-\mathcal Q_+(B_0x)}e^{\mathcal Q_-\circ \Psi (B_1x)} f_1^{c_1}(B_1 x)\, dx \\
    &=& |\det((B_0,B_1))|^{-1} \int_{H_0} e^{-\mathcal Q_+(y)} \int_{H_1} e^{\mathcal Q_-\circ \Psi( z)} f_1^{c_1}(z) dz\\
    &\ge& |\det((B_0,B_1))|^{-1} \int_{H_0} e^{-\mathcal Q_+(y)}  
    \left(\int_{H_1} e^{\frac{1}{1-c_1}\mathcal Q_-\circ \Psi( z)} dz\right)^{1-c_1}\left( \int_{H_1} f_1\right)^{c_1},
\end{eqnarray*}  
where we have used the change of variables $(y,z)=(B_0x,B_1x)$ and the inverse H\"older inequality. The proof is complete.
\end{proof}

\begin{proof}[Proof of Theorem~\ref{thm:positivity-of-constant-general-case} (ii)]
First of all, we can assume $\dim H_k \ge 1$ for all $k \in \{1, \ldots, m\}$, otherwise one can reduce the problem by discarding all functions $f_k$ for which $\dim H_k = 0$ while Condition (C) and the strong positivity property related to the reduced problem remain equivalent to those related to the original problem.

Let   $D_1 := [1,+\infty)^{m^+} \times (-\infty, 0]^{m-m^+} \subseteq \R^m$. 
Our main interest is in the following set:
\begin{align*}
 \mathcal C :&= \big\{ c\in (0,+\infty)^{m^+}\times (-\infty,0]^{m-m^+};\; (H,B,c)\, \mathrm{satisfies}\, \mathrm{Condition\, (C)} \big\}\\
 &= \big\{ c\in D_1;\; (H,B,c)\, \mathrm{satisfies}\, \mathrm{Condition\, (C)} \big\},
\end{align*}
where the latter equality comes from Proposition~\ref{prop:c-i-ge-1}.
Since Condition (C) means that the vector $c$ verifies several closed linear inequalities, the second expression of $\mathcal C$ proves that it
is closed and convex.  More specifically, the triplet $(H,B,c)$ satisfies Condition (C) if and only if $c$ belongs to all sets in  the following two families:
\begin{itemize}
\item 
For any non-trivial subspace $V \subseteq H$ which induces an admissible split and satisfies $V \subseteq \ker B_{m+1}$ consider
\[
  S_V = \Big\{ x \in \R^m \colon \sum_{k=1}^m x_k \dim B_k V  \le \dim V \Big\}.
\]
Typically, $S_V$ is a closed half-space of $\R^m$, but it can happen that $S_V$ is the whole $\R^m$.
\item
Similarly, for any proper subspace $V \subsetneq H$ which induces an admissible split and satisfies $B_0 V = B_0 H$ consider
\[
  S^{\sfrac{H}{V}} = \Big\{ x \in \R^m \colon \sum_{k=1}^m  x_k \dim \sfrac{B_k H}{B_k V} \ge \dim \sfrac{H}{V} \Big\}.
\]
The set $S^{\sfrac{H}{V}}$ is always a half-space of $\R^m$, since at least for some $1 \le k \le m^+$, $B_k V \neq B_k H$ (otherwise for each $k = 0, 1, \ldots, m^+$, $B_k V = B_k H$, i.e. $V + \ker B_k = H$, which by Lemma~\ref{lem:admissibility-of-subspace-and-quotient} would contradict admissibility of the split).
\end{itemize}
Even though there are infinitely many subspaces $V$, the coefficients $\dim B_kV$ and $\dim \sfrac{B_kH}{B_kV}$ take finitely many values. Hence there are finitely many different  half-spaces in the above families, and the set $\mathcal C$ is a closed convex polyhedron (which may be unbounded).

Since $\mathcal C$ is closed, it follows from its original definition that its boundary is covered by the union of the affine hyperplanes in the following three families:
\[ \begin{split}
  \mathcal{P} &= \{ \partial S_V \colon \{0\} \neq V \subsetneq H \textup{ induces an admissible split and  } V \subseteq \ker B_{m+1} \} \\
  &\cup 
  \{ \partial S^{\sfrac{H}{V}} \colon \{0\} \neq V \subsetneq H \textup{ induces an admissible split and  } B_0 V = B_0 H \},
\end{split} \]
$\mathcal{P}_0 = \{ \partial S_H \}$ if $B_{m+1}=0$  or $B_0=0$  (note that $\partial S_H = \partial S^H$) otherwise $\mathcal{P}_0 = \emptyset$ and
\[
  \mathcal{B} = \Big\{ \{ x  \in \R^m \colon x_k = 0 \} \colon k = m^+ + 1, \ldots, m \Big\}.
\]

Our aim is to show that $c \in \mathcal{C}$  (i.e. Condition (C)) implies the strong positivity property for $(H, B, c)$. 
We proceed by induction in $(\dim H, m)$ with a partial order on $(n, m) \in \mathbb{Z}_+^2$ given by $(n_1, m_1) \preceq (n_2, m_2)$ if and only if $n_1 \le n_2$ and $m_1 \le m_2$. The founding cases are $m\in\{0,1\}$ with any $\dim H$ (this is treated in Lemma~\ref{lem:initialization-1function}) and $\dim H = 1$ with any $m \ge 1$ (see Lemma~\ref{lem:initialization-dimension1}).

The induction step, in which $\dim H \ge 2$ and $m \ge 2$, goes as follows. Take any vector $p = (p_1, \ldots, p_m) \in \R^m$ such that $p_1, \ldots, p_{m^+} > 0$ and $p_{m^+ + 1}, \ldots, p_m < 0$ and let $b \in \R$ be a constant such that the affine hyperplane
\[
  P = \Big\{ x \in \R^m \colon \sum_{k=1}^m p_k x_k = b \Big\}
\]
contains $c$.
Then $D_1 \cap P$ is a compact subset of $\R^m$ and thus $\mathcal{C}_P := \mathcal{C} \cap P$ is a compact convex set (actually it is a compact convex polytope). In what follows, we restrict our considerations to the hyperplane $P$.

Due the convexity property established in Proposition~\ref{prop:interpolation}, it is enough to prove the strong positivity property for the vectors of exponents being vertices of $\mathcal{C}_P$. Therefore assume $c$ is a vertex of $\mathcal{C}_P$. Consequently, $c$ belongs to an intersection of $P$ and some $m-1$ distinct affine hyperplanes from the family $\mathcal{B} \cup \mathcal{P} \cup \mathcal{P}_0$. Consider three cases:
\begin{itemize}
  \item[Case 1.] Among these $m-1$ affine hyperplanes there is at least one which belongs to $\mathcal{B}$. This means that for some $k \in \{ m^+ + 1, \ldots, m\}$, $c_k = 0$. Then we can discard the function $f_k$ and in this way reduce the number of functions considered from $m$ to $m-1$. Since neither Condition (C) nor the strong positivity property is affected by this reduction (both assertions remain equivalent for the original and the reduced problem), we are done by the induction hypothesis.
  \item[Case 2.] Among these $m-1$ affine hyperplanes there is at least one which belongs to $\mathcal{P}$. Hence for some $\{0\} \neq V \subsetneq H$, $V$ is a critical subspace or $\sfrac{H}{V}$ is a critical quotient (this may happen only when $\dim H\ge 2$). Since both $V$ and $\sfrac{H}{V}$ have dimension strictly smaller than $\dim H$ and thanks to Lemma~\ref{lem:inheritence-of-condition-C}, Condition (C) is satisfied for $(V, b, c)$ and $(\sfrac{H}{V}, \beta, c)$, we can apply the induction hypothesis and get the strong positivity property for $(V, b, c)$ and $(\sfrac{H}{V}, \beta, c)$. Now the strong positivity property for $(H, B, c)$ follows from Lemma~\ref{lem:tensorization}.
  \item[Case 3.] Neither Case 1 nor Case 2 holds, i.e. all $m-1$ distinct affine hyperplanes are in $\mathcal{P}_0$. This is possible only when $m = 2$ and $\mathcal{P}_0 = \{ \partial S_H \}$, i.e. $B_0$ or $B_{m+1}$ is trivial. Then $c \in \partial S_H \cap P$. If $B_{m+1}$ is trivial then $H$ is a critical subspace and Proposition~\ref{prop:c-i-ge-1} implies that $B_0$ is also trivial and we can conclude using Lemma~\ref{lem:initialization-2functions-no-kernel}. If $B_0$ is trivial then $\sfrac{H}{\{0\}}$ is a critical quotient, i.e.
  \begin{equation}\label{eq:criticality-of-quotient-H}
    \dim H = \sum_{k=1}^m c_k \dim B_k H.
  \end{equation}
Condition (C) tells us that for every subspace $V \subseteq H$ which induces an admissible split, the quotient $\sfrac{H}{V}$ is subcritical. Subtracting the corresponding inequality from~\eqref{eq:criticality-of-quotient-H} gives that $V$ satisfies
\[
  \dim V \ge \sum_{k=1}^m c_k \dim B_k V,
\]
regardless whether $V \subseteq \ker B_{m+1}$ or not. Therefore Condition (C) for our problem implies Condition (C) which corresponds to the problem with the same vector of exponents $c$ and the same maps $B_0, \ldots, B_m$ and with $B_{m+1} = 0$. Lemma~\ref{lem:initialization-2functions-no-kernel}  ensures the strong positivity property for the modified problem which in turn clearly implies the strong positivity property for the original problem.
\end{itemize}

\end{proof}

\providecommand{\bysame}{\leavevmode\hbox to3em{\hrulefill}\thinspace}
\providecommand{\MR}{\relax\ifhmode\unskip\space\fi MR }
\providecommand{\MRhref}[2]{%
  \href{http://www.ams.org/mathscinet-getitem?mr=#1}{#2}
}
\providecommand{\href}[2]{#2}

%------------------------
%\bibliographystyle{amsplain}
%\bibliography{rev-b-l}

\end{document}